\documentclass[11pt]{article}

\usepackage{verbatim,latexsym,amsfonts,amsmath,amssymb,graphicx,fancyhdr,hyperref,
}
\usepackage{appendix,latexsym,amsfonts,amsmath,amssymb,graphicx,hyperref,amsthm,soul,verbatim,authblk}
\usepackage[framemethod=tikz]{mdframed}

\newcommand{\EE}{\mathbb{ E}}
\newcommand{\PP}{\mathbb{P}}

\newcommand{\R}{\mathbb{R}}
\newcommand{\C}{\mathbb{C}}
\newcommand{\Q}{\mathbb{Q}}

\newcommand{\HH}{\mathbb{H}}
\newcommand{\N}{\mathbb{N}}
\newcommand{\D}{\mathbb{D}}
\newcommand{\Z}{\mathbb{Z}}

\newcommand{\TT}{\mathbb{T}}

\newcommand{\pa}{\partial}

\newcommand{\F}{{\cal F}}

\def\eps{\varepsilon}
\def\til{\widetilde}
\def\ha{\widehat}
\def\sem{\setminus}
\def\lin{\overline}
\def\ulin{\underline}

\def\L{{\cal L}}
\def\conf{\stackrel{\rm Conf}{\twoheadrightarrow}}

 \DeclareMathOperator{\diam}{diam}
\DeclareMathOperator{\dist}{dist} \DeclareMathOperator{\dcap}{dcap}
 \DeclareMathOperator{\id}{id}
\DeclareMathOperator{\Imm}{Im }

\DeclareMathOperator{\mA}{m} 
 
 \DeclareMathOperator{\cc}{c}
\DeclareMathOperator{\bb}{b} \DeclareMathOperator{\doub}{doub}

\theoremstyle{plain}
\newtheorem{Theorem}{Theorem}[section]
\newtheorem{Lemma}[Theorem]{Lemma}
\newtheorem{Corollary}[Theorem]{Corollary}
\newtheorem{Proposition}[Theorem]{Proposition}
\theoremstyle{definition}
\newtheorem{Definition}[Theorem]{Definition}
\newtheorem{Remark}[Theorem]{Remark}
\numberwithin{equation}{section}
\newcommand{\BGE}{\begin{equation}}
\newcommand{\BGEN}{\begin{equation*}}
\newcommand{\EDE}{\end{equation}}
\newcommand{\EDEN}{\end{equation*}}

\begin{document}
\title{Two-curve Green's function for $2$-SLE: the interior case}
\author{Dapeng Zhan\thanks{Research partially supported by NSF grants  DMS-1056840 and DMS-1806979.}}
\affil{Michigan State University}
\maketitle

\begin{abstract}
  A $2$-SLE$_\kappa$ ($\kappa\in(0,8)$) is a pair of random curves $(\eta_1,\eta_2)$ in a simply connected domain $D$  connecting two pairs of boundary points  such that conditioning on any curve, the other is a chordal SLE$_\kappa$ curve in a complement domain. In this paper we prove that for any $z_0\in D$, the limit $\lim_{r\to 0^+}r^{-\alpha_0} \PP[\dist(z_0,\eta_j)<r,j=1,2]$, where $\alpha_0=\frac{(12-\kappa)(\kappa+4)}{8\kappa}$, exists. Such limit is called a two-curve Green's function. We find the convergence rate and the exact formula of the Green's function in terms of a hypergeometric function up to a multiplicative constant. For $\kappa\in(4,8)$, we also prove the convergence of $\lim_{r\to 0^+}r^{-\alpha_0} \mathbb{P}[\dist(z_0,\eta_1\cap \eta_2)<r]$, whose limit is a constant times the previous Green's function. To derive these results, we work on two-time-parameter stochastic processes, and use orthogonal polynomials to derive the transition density of a two-dimensional diffusion process that satisfies some system of SDE.
\end{abstract}

\section{Introduction}
The Schramm-Loewner evolution (SLE), first introduced by Oded Schramm in 1999 (\cite{S-SLE}), is a one-parameter ($\kappa\in(0,\infty)$) family of measures on non-self-crossing curves, which has received a lot of attention over the past two decades. 
It has been shown that, modulo time parametrization,  many discrete random paths on grids
have SLE with different parameters as their scaling limits. We refer the reader to Lawler's textbook \cite{Law-SLE} for basic properties of SLE.

One of the most important functions associated to SLE is the Green's function, which can be roughly defined as the scaling limit of the probability that an SLE curve hits a small disc around an interior or boundary point of its domain. The existence of chordal SLE   Green's function for an interior point was given in \cite{Law4}, where conformal radius was used instead of Euclidean distance. The existence of the original one-point Green's function (using Euclidean distance) was proved later in \cite{LR}. 
The existence of boundary point Green's function for chordal SLE was given in \cite{Mink-real}. Other related works include the Green's function for radial SLE (\cite{AKL}), multipoint Green's function for chordal SLE (\cite{LW,Mink-real,LR} for $2$-point, \cite{existence} for $n$-point), and Green's function for SLE$_\kappa(\rho)$ and hSLE (\cite{Vik-Green}).


A $2$-SLE$_\kappa$ (also called bi-chordal SLE$_\kappa$) is a pair of random curves  in a simply connected domain connecting two pairs of boundary points, which satisfy   that, when any one curve is given, the conditional law of the other curve is that of a chordal SLE$_\kappa$ curve in one complement domain of the first curve. It is a special case of multiple $N$-SLE$_\kappa$ (when $N=2$) studied in \cite{multiple}, and exists for all $\kappa\in(0,8)$ and any admissible link pattern. A $2$-SLE arises naturally as a scaling limit of some lattice model with alternating boundary conditions (\cite{Wu-hSLE,KS-hSLE}), as interacting flow lines in imaginary geometry (\cite{MS1,MS2}), and as two exploration curves of a CLE (\cite{MSW,MW-connection}).

Suppose $(\eta_1,\eta_2)$ is a $2$-SLE$_\kappa$  in a simply connected domain $D$, and $z_0\in \lin D$. Then the probability that both $\eta_1$ and $\eta_2$ visit a small disc centered at $z_0$ with radius $\eps$ tends to $0$ as $\eps\to 0$. It is expected that this probability decays like some power of $\eps$, and the rescaled probability tends to a nontrivial limit, which is called the two-curve Green's function for this $2$-SLE$_\kappa$. A similar object considered in \cite{Vik-Green}  is the rescaled probability that either $\gamma_1$ or $\gamma_2$ gets close to a given interior point. Their Green's function is a sum of two one-curve Green's functions for the $2$-SLE$_\kappa$, and is different from the one considered here. In this paper we focus on the interior point case, i.e., $z_0\in D$. In the subsequent paper \cite{Two-Green-boundary}, we will work on the boundary point case, which uses a similar approach.

Below is our first main theorem, which holds for all $\kappa\in(0,8)$.

\begin{Theorem}
  Let $\kappa\in(0,8)$. Let
  \BGE \alpha_0=\frac{(12-\kappa)(\kappa+4)}{8\kappa}>0.\label{alpha0}\EDE
  Let $F$ be the hypergeometric function $\,_2F_1(\frac 4\kappa,1-\frac 4\kappa;\frac 8\kappa,\cdot)$, which is known to be positive on $[0,1]$. Let $D$ be a simply connected domain  with four distinct boundary points (prime ends)  $a_1,b_1,a_2,b_2$ such that $a_1$ and $a_2$ together separate $b_1$ from $b_2$ on $\pa D$. Let $(\ha \eta_1,\ha \eta_2)$ be a $2$-SLE$_\kappa$ in $D$ with link pattern $(a_1, b_1;a_2, b_2)$. Let  $z_0\in D$, and  $f_{z_0}$ be the conformal map from $D$ onto $\D=\{|z|<1\}$ such that $f_{z_0}(z_0)=0$ and $f_{z_0}'(0)>0$.
   Let
  $$G_{D;a_1,b_1;a_2,b_2}(z_0):=4^{1-\frac{12}\kappa}|f'(z_0)|^{\alpha_0} \prod_{j=1}^2|f_{z_0}(a_j)-f_{z_0}(b_j)|^{\frac 8\kappa -1} \prod_{x\in\{a,b\}}|f_{z_0}(x_1)-f_{z_0}(x_2)|^{\frac 4\kappa } $$ $$\times F\Big(\frac{|f_{z_0}(a_1)-f_{z_0}(b_2)||f_{z_0}(a_2)-f_{z_0}(b_1)|}{|f_{z_0}(a_1)-f_{z_0}(a_2)||f_{z_0}(b_1)-f_{z_0}(b_2)|}\Big)^{-1}.$$
  Let  $\beta_0=\frac{2+\frac\kappa 8}{3+\frac \kappa 8}$. Let $R=\dist(z_0,\pa D)$.
  Then there is a constant $C_0>0$  depending only on $\kappa$ such that
\BGE \PP[\dist(z_0,\ha \eta_j)<r,j=1,2]= C_0G_{D;a_1,b_1;a_2,b_2}(z_0) r^{\alpha_0}\Big(1+O\Big(\Big(\frac{r}{R}\Big)^{\beta_0}\Big)\Big),\quad\mbox{as }   r\to 0^+.\label{main-est-1}\EDE
Here the implicit constants in the $O$ symbol depend only on $\kappa$. In particular, it implies that there is a constant $C_0' >0$ depending only on $\kappa$ such that
\BGE \PP[\dist(z_0,\ha \eta_j)<r,j=1,2]\le C_0' \Big(\frac rR\Big)^{\alpha_0},\quad \forall r>0.\label{main-est-2}\EDE \label{main-thm1}
\end{Theorem}

Below is our second main theorem, which makes sense only for $\kappa\in(4,8)$.

\begin{Theorem}
	Let $\kappa\in(4,8)$. We adopt the notation in the last theorem. Then there is a constant $C_1>0$  depending only on $\kappa$ such that
	$$\PP[\dist(z_0,\ha \eta_1\cap\ha \eta_2)<r]= C_1G_{D;a_1,b_1;a_2,b_2}(z_0) r^{\alpha_0}\Big(1+O\Big( \frac{r}{R} \Big)^{\beta_0}\Big),\quad\mbox{as }  r\to 0^+.$$ \label{main-thm2}
\end{Theorem}

The exponent $\alpha_0$ appeared earlier in \cite{MW}: it agrees with the decay rate $A$ in \cite[Formula (4.13)]{MW}  with $\rho=0$, and equals $2$ minus the Hausdorff dimension of the double points of SLE$_\kappa$ (\cite[Theorem 1.1]{MW}). We expect that Theorem \ref{main-thm2} may be used to prove the existence of the $(2-\alpha_0)$-dimensional  Minkowski content of $\eta_1\cap\eta_2$, following the approach of \cite{LR}. For that purpose, one needs the existence of two-curve two-point Green's function, i.e., the rescaled probability that the two curves both get near two marked points.

Once the existence of the  Minkowski content of $\eta_1\cap\eta_2$ is established, one may further derive the existence of Minkowski content of double points of a single SLE$_\kappa$ curve. When $\kappa=6$, this is the analogue of the Minkowski content of the pivotal points in the critical lattice model (cf.\ \cite{GPS}).
The approach is related to another object called two-sided radial SLE$_\kappa$. A two-sided radial SLE$_\kappa$ curve grows in a simply connected domain, say $D$, from one marked boundary point, say $a_1$, passes through a marked interior point, say $z_0$, and ends at another marked boundary point, say $a_2$. The marked interior point $z_0$ breaks the curve into two arms, which satisfy the property that, given any arm the other arm is a chordal SLE$_\kappa$ in one complement domain of the given arm. So we may view the two arms as $2$-SLE$_\kappa$ in $D$ with link pattern $(a_1,z_0;a_2,z_0)$. This is not the same as the $2$-SLE$_\kappa$ studied in Theorems \ref{main-thm1} and \ref{main-thm2}. But we expect that the $(2-\alpha_0)$-dimensional  Minkowski content of the intersection of the two arms also exists.

To study the Minkowski content of double points of a chordal SLE$_\kappa$ curve, one may use the decomposition of chordal SLE (\cite{decomposition}): if one weights the law of a chordal SLE$_\kappa$ curve $\eta$ in $D$ from $a_1$ to $a_2$ by its $(1+\frac \kappa 8)$-dimensional Minkowski content, and samples a point $z_0$ on $\eta$ according to the Minkowski content measure, then one gets a measure on the pairs $(\eta,z_0)$, which agrees with that of the pair obtained by first sampling a random point $z_0$ in $D$ according to the chordal SLE$_\kappa$ Green's function in $D$ from $a_1$ to $a_2$, and then sampling a two-sided radial SLE$_\kappa$ in $D$ from $a_1$ to $a_2$ passing through $z_0$. Note that changing the law of $\eta$ does not affect the a.s.\ existence of the Minkowski content of double points of $\eta$. If $z$ is a double point of $\eta$ (under the new measure), then there are $t_1<t_2$ such that $\eta(t_1)=\eta(t_2)=z$. When the sampled point $z_0$ falls on the arc $\eta(t_1,t_2)$, the two parts of $\eta$ before $z_0$ and after $z_0$ form two arms of a two-sided radial SLE$_\kappa$ in $D$, and the double point $z$ lies in the intersection of the two arms.

Similar theorems also hold in the case that $z_0$ lies on the boundary, assuming that $\pa D$ is smooth near $z_0$ (\cite{Two-Green-boundary}), where the exponent $\alpha_0$ is replaced by another exponent: $\frac2\kappa(12-\kappa)$. That result may serve as the boundary estimate when one tries to derive the two-curve two-point Green's function for $2$-SLE$_\kappa$.

\begin{Definition}
We call  $G_{D;a_1,b_1;a_2,b_2}$ in Theorem \ref{main-thm1} the  two-curve Green's function for  $2$-SLE$_\kappa$  in $D$ with link pattern $(a_1,b_1;a_2,b_2)$.
\end{Definition}

\begin{Remark}  It is easy to derive the following properties of the  two-curve Green's function.
  \begin{itemize}
    \item[(i)] Using Koebe's $1/4$ Theorem and the boundedness of $F$ on $[0,1]$, we see that there is a constant $C>0$ depending only on $\kappa$ such that
    \BGE G_{D;a_1,b_1;a_2,b_2}(z_0)\le C \dist(z_0,\pa D)^{-\alpha_0}.\label{Green-bound}\EDE
    \item[(ii)] For $a_1,b_1,a_2,b_2$ in the definition, there is another admissible link pattern, which is $(a_1,b_2;a_2,b_1)$. It is easy to see that $\frac{G_{D;a_1,b_2;a_2,b_1}(z_0)}{G_{D;a_1,b_1;a_2,b_2}(z_0)}$ does not depend on $z_0$, but only on the cross-ratio of $a_1,b_1,a_2,b_2$ in $D$.
  \end{itemize}
\end{Remark}

The approach of the main theorems is somehow similar to that of the Green's function for a single chordal SLE$_\kappa$, where one parametrizes the curve according to the conformal radius viewed from the marked point and obtains an invariant measure on a process of harmonic measures. The conformal radius parametrization is convenient because on the one hand, it is comparable to the Euclidean distance by Koebe's $1/4$ theorem, and on the other hand, it is easier to do calculation using It\^o's calculus. Here is how it goes for the two-curve setting here. In order to apply It\^o's calculus, we choose a common single-time-parametrization for both of $\eta_1$ and $\eta_2$, which is equivalent to letting the two curves grow simultaneously. The success of the chordal SLE Green's function theory suggests us to choose a parametrization such that at any time $t$, the conformal radius of the remaining domain viewed from $z_0$ is $e^{-t}$. However,  this requirement does not uniquely determine the parametrization, and Koebe's $1/4$ theorem only tells us that the minimum of $\dist(z_0,\eta_j[0,t])$, $j=1,2$, is comparable to $e^{-t}$. It is not necessarily true that both curves get close to $z_0$ for big $t$. So we need a second rule to determine the parametrization, which should guarantee that the maximum of $\dist(z_0,\eta_j[0,t])$, $j=1,2$, is comparable to $e^{-t}$, and should also be convenient for calculation.

We find a second rule using the marked points $b_1$ and $b_2$. By conformal invariance, we may assume that $D$ is the unit disc $\D=\{|z|<1\}$ and $z_0$ is the center $0$. We may further reduce it to the case that $b_1$ and $b_2$ are opposite points on the circle, say $b_1=1$ and $b_2=-1$, by growing a part of $\eta_1$ or $\eta_2$ and mapping the remaining domain back to $\D$. Then we choose to grow $\eta_1$ and $\eta_2$ simultaneously  so that at any time $t$, (i) the conformal radius of the remaining domain viewed from $0$ is $e^{-t}$; and (ii) the harmonic measure in the remaining domain viewed from $0$ of any boundary arc bounded by $b_1$ and $b_2$ is $1/2$. These two rules uniquely determine the growth up to the time when either $\eta_1$ or $\eta_2$ finishes its journey or the two curves together disconnect $0$ from $b_1$ or $b_2$, and we stop at that time.
By Beurling's estimate, the second rule implies that $\dist(0,\eta_1[0,t])$ is comparable to $\dist(0,\eta_2[0,t])$ for each $t$ within the lifetime. Combining this with the estimate using Koebe's $1/4$ theorem, we then find that both $\dist(0,\eta_1[0,t])$ and $\dist(0,\eta_2[0,t])$ are comparable to $e^{-t}$

%

For any fixed time $t$,   on the event that the process does not end at time $t$, by the second rule we may choose $g_t$ which maps the remaining domain conformally onto $\D$ and fixes $0,1,-1$. Then $g_t(\eta_j(t))$, $j=1,2$, are two points on $\pa\D$, which are separated by $b_1=1$ and $b_2=-1$.
We then study the growth of the two-dimensional Markov process $(Z_1(t) ,Z_2(t) )$ in $(0,\pi)^2$, where $Z_j(t)=\arg(g_t(\eta_j(t))/b_j)$. Using a framework of two-parameter martingales, we are able to show that $(Z_1,Z_2)$ is a semi-martingale, and derive the system of SDEs for them.

Then we follow the approach of \cite[Appendix B]{tip} and use orthogonal two-variable polynomials to derive the explicit transition density for this Markov process. Here we use the known fact that each of $\eta_1$ and $\eta_2$ is a hypergeometric SLE (\cite{Wu-hSLE}). Using the transition density, we find that $(Z_1,Z_2)$ has a quasi-invariant measure, say $\mu_*^\#$ on  $(0,\pi)^2$, which means that if we start $(Z_1,Z_2)$ from a random point with law $\mu_*^\#$, then for any deterministic $t>0$, the probability that the process survives at time $t$ is $e^{-\alpha_0 t}$, and the law of $(Z_1(t),Z_2(t))$ conditional on this event is still $\mu_*^\#$. Furthermore, if we start $(Z_1,Z_2)$ from any deterministic point, then the conditional distribution of $(Z_1(t),Z_2(t))$ approaches exponentially to $\mu_*^\#$. With this quasi-invariant measure in hand, the remaining part of the proofs of the main theorems are finished by using Koebe's distortion theorem.

The technique of this paper also works for the boundary Green's function for $2$-SLE$_\kappa$ (\cite{Two-Green-boundary}), where the marked point $z_0$ lies on the boundary, and we can derive an estimate similar to Theorem \ref{main-thm1}. We may apply the technique to study the $n$-curve Green's function of $n$-SLE$_\kappa$, $n\ge 3$, when the link pattern satisfy that none of the curves separates any two other curves. In this case we may assume $D=\D$, $z_0=0$, and $a_1,b_1,\dots,a_n,b_n$ are ordered clockwise on $\pa\D$. Then we grow the $n$ curves from $a_1,\dots,a_n$ simultaneously such that the conformal radius at any time $t$ equals $e^{-t}$, and for any $1\le k\le n$, the harmonic measure of the boundary arc between $b_{k-1}$ and $b_{k}$  stays constant. These two rules guarantee that at any time $t$, each $\dist(0,\eta_k[0,t])$ is comparable to $e^{-t}$, and we may choose $g_t$ which maps the remaining domain conformally onto $\D$ and fixes $0,b_1,\dots,b_n$. Then we obtain an $n$-dimensional process $(Z_1,\dots,Z_n)$, where $Z_k(t)=\arg(g_t(\eta_k(t))/b_k)$. When we have the explicit formula for the marginal law of each curve in the $n$-SLE$_\kappa$, which is known so far only for $\kappa=2$ (\cite[Theorem 4.1]{KKP17}) and $\kappa=4$ (\cite[Theorem 1.5]{PW17}), we can obtain the transition density and quasi-invariant density for the $Z$-process, which can be used to conclude the existence of the Green's function.

The rest of the paper is organized as follows. In Sections \ref{section-Loewner} and \ref{section-hSLE}, we review Loewner equations, SLE, $2$-SLE, and hypergeometric SLE. In Section \ref{section-two-parameter} we develop a framework on stochastic processes that depend on two time parameters. In Section \ref{section-Ensemble} we describe the interaction between two radial Loewner chains, whose chordal counterpart appeared earlier in the works on the reversibility and duality of SLE (\cite{reversibility,duality}). The essential new stuff starts from Section \ref{section-Time-curve}, in which we grow the two curves in a $2$-SLE simultaneously as described above and derive the SDEs for the  process $(Z_1(t),Z_2(t))$. In Section \ref{section-transition-density} we derive the transition density and quasi-invariant density of this process. In the last section, we finish the proofs.

\section*{Acknowledgments}
The author thanks Xin Sun for suggesting this problem and inspiring discussions on the project. The final part of the paper was finished when the author attended the conference ``Random Conformal Geometry and Related Fields'' held by KIAS. The author thanks the participants of the conference for valuable comments on the paper.


\section{Preliminary} \label{section-prel}
\subsection{Loewner equations, SLE and $2$-SLE} \label{section-Loewner}
In this subsection, we recall the definitions of Loewner equations, SLE and $2$-SLE. Let $\HH$ denote the upper half-plane $\{z\in\C:\Imm z>0\}$. Let $\D$ and $\TT$ denote the unit disc $\{z\in\C:|z|<1\}$ and its boundary, respectively. We use $\cot_2,\tan_2,\sin_2,\cos_2$ to denote the functions $\cot(\cdot/2),\tan(\cdot/2),\sin(\cdot/2),\cos(\cdot/2)$, respectively.

We will extensively use radial Loewner equation in the paper. For the definition, we start with hulls in $\D$.
A set $K\subset\D$ is called a $\D$-hull if $\D\sem K$ is a simply connected domain that contains $0$. For a $\D$-hull $K$, there is a unique conformal map $g_K$ from $\D\sem K$ onto $\D$ such that $g_K(0)=0$ and $g_K'(0)>0$. By Schwarz Lemma, $g_K'(0)\ge 1$, and the equality holds only when $K=\emptyset$. By Schwarz reflection principle, we may view $g_K$ as a conformal map from $\C\sem K^{\doub}$ onto $\C\sem S_K$, where $K^{\doub}$ is the union of the closure of $K$ and the reflection of $K$ about $\TT$, i.e., $\{1/\lin z: z\in K\}$, and $S_K$ is a compact subset of $\TT$. 
Let $\dcap(K):=\log(g_K'(0))\ge 0$ be called the $\D$-capacity of $K$.
If $K_1\subset K_2$ are two $\D$-hulls, then we define $K_2/K_1:=g_{K_1}(K_2\sem K_1)$, which is also a $\D$-hull, and satisfies $\dcap(K_2/K_1)=\dcap(K_2)-\dcap(K_1)$. 

Let $\ha w \in C([0,T),\R)$ for some $T\in(0,\infty]$. The radial Loewner equation driven by $\ha w$  is
$$\pa_t g_t(z)= g_t(z)\cdot\frac{e^{i\ha w(t)}+g_t(z)}{e^{i\ha w(t)}-g_t(z)},\quad 0\le t<T;\quad g_0(z)=z.$$
For each $t\in[0,T)$, let $K_t$ be the set of $z\in\D$ such that the solution $g_{\cdot}(z)$ blows up before or at $t$ (so that $g_t$ is well defined on $\D\sem K_t$). Then we call $g_t$ and $K_t$ the radial Loewner maps and hulls, respectively, driven by $\ha w$. It turns out that, for each $t$, $K_t$ is a $\D$-hull with $\dcap(K_t)=t$, and $g_{K_t}=g_t$. If for every $t\in[0,T)$, $g_t^{-1}$ as a conformal map from $\D$ onto $\D\sem K_t$ extends continuously to $\lin\D$, and $\eta(t):=g_t^{-1}(e^{i\ha w(t)})$, $0\le t<T$, is continuous in $t$, then we say that $\eta$ is a radial Loewner curve driven by $\ha w$. Such $\eta$ may not exist in general; when it exists, the hulls $(K_t)$ are generated by $\eta$ in the sense that for every $t$, $\D\sem K_t$ is the connected component of $\D\sem \eta([0,t])$ that contain $0$. 

Let $\ha w$ be as above. Let $u$ be a continuous and strictly increasing function defined on $[0,T)$ with $u(0)=0$. Suppose that the two families $g^u_t$ and $K^u_t$, $0\le t<T$, satisfy that $g^u_{u^{-1}(t)}$ and $K^u_{u^{-1}(t)}$, $0\le t<u(T)$, are radial Loewner maps and hulls, respectively, driven by $\ha w\circ u^{-1}$. Then we say that $g^u_t$ and $K^u_t$, $0\le t<T$, are radial Loewner maps and hulls, respectively, driven by $\ha w$ with speed $du$. If $u$ is absolutely continuous, then we say that the speed is $u'$. 

The following lemma is  well known, and has appeared in the literature in different forms. 

\begin{Lemma}
	Suppose $K_t$, $0\le t<T$, are radial Loewner hulls driven by some $\ha w\in C([0,T),\R)$. Let $L$ be a $\D$-hull such that $\lin L\cap \lin{K_t}  =\emptyset$ for all $t\in[0,T)$. Then for any $t\in[0,T)$, $g_{K_t}(L)$ is a $\D$-hull that has positive distance from $e^{i\ha w(t)}$, so that $g_{g_{K_t}(L)}$ is analytic at $e^{i\ha w(t)}$; and $g_L(K_t)$, $0\le t<T$, are radial Loewner hulls driven by some $\ha w^L\in C([0,T),\R)$ with speed $|g_{g_{K_t}(L)}'(e^{i\ha w(t)})|^2$, where $\ha w^L$ satisfies $e^{i\ha w^L(t)}=g_{g_{K_t}(L)}(e^{i\ha w(t)})$, $0\le t<T$. \label{Loewner-KL}
\end{Lemma}

It will be useful to work on the covering radial Loewner equation. Let $e^i$ denote the covering map $z\mapsto e^{iz}$ from $\HH$ onto $\D\sem\{0\}$. 
The covering radial Loewner equation driven by $\ha w \in C([0,T),\R)$   is
$$ \pa_t \til g_t(z)= \cot_2(\til g_t(z)-\ha w(t)),\quad g_0(z)=z.$$ 
For each $t\in[0,T)$, let $\til K_t$ denote the set of $z\in\HH$ such that the solution $\til g_{\cdot}(z)$ blows up before or at $t$. We call $\til g_t$ and $\til K_t$, $0\le t<T$, the covering radial Loewner maps and hulls, respectively, driven by $\ha w$. It turns out that $\til K_t$ has period $2\pi$, $\til g_t$ maps $\HH\sem K_t$ conformally onto $\HH$ with $\til g_t(z+2\pi)=\til g_t(z)+2\pi$; and if $(g_t)$ and $(K_t)$ are the radial Loewner maps and hulls driven by $\ha w$, then $\til K_t=(e^i)^{-1}(K_t)$ and $e^i\circ \til g_t=g_t\circ e^i$. If $u$ is a continuous and strictly increasing function on $[0,T)$, we may similarly define covering radial Loewner maps $\til g^u_t$ and hulls $\til K^u_t$ with speed $du$ driven by $\ha w$. 

If $\ha w(t)=\sqrt\kappa B(t)$, $0\le t<\infty$, where $\kappa>0$ and $B(t)$ is a standard Brownian motion, then the radial Lowner curve $\eta$ driven by $\ha w$ is known to exist, and is called a radial SLE$_\kappa$ curve in $\D$ from $1$ to $0$.
What will be used in this paper is a generalization of radial SLE$_\kappa$: radial SLE$(\kappa;\ulin\rho)$, whose growth is affected by one or more force points lying on the boundary or the interior. For the generality needed here, we assume that all force points lie on the boundary and are distinct from the initial point of the curve. We start with the definition of radial SLE$(\kappa;\ulin\rho)$ in $\D$. Let $\rho_1,\dots,\rho_n\in\R$. Let $e^{iw},e^{iv_1},\dots,e^{iv_n}$ be distinct points on $\TT$.  Let $B(t)$ be a standard Brownian motion. Suppose that $\ha w(t)$ and $\ha v_j(t)$, $1\le j\le n$, $0\le t<T$, solve the following system of SDEs with the maximal solution interval:
$$d\ha w(t)=\sqrt\kappa dB(t)+\sum_{j=1}^n \frac{\rho_j}2 \cot_2(\ha w(t)-\ha v_j(t))dt,\quad \ha w(0)=w;$$
$$d\ha v_j(t)=\cot_2(\ha v_j(t)-\ha w(t)),\quad \ha v_j(0)=v_j,\quad 1\le j\le n.$$
Then we call the radial Loewner curve driven by $\ha w$  the SLE$(\kappa;\rho_1,\dots,\rho_n)$ curve in $\D$ started from $e^{iw}$ aimed at $0$ with force points $e^{iv_1},\dots,e^{iv_n}$.  The covering radial Loewner maps implicitly appear in the definition: if $\til g_t$ are covering radial Loewner maps, then $\ha v_j(t)=\til g_t(v_j)$.

Although we say that $\eta$ is aimed at $0$, it often happens that $\eta$ does not end at $0$.
A radial SLE$(\kappa;\ulin\rho)$ curve in a general simply connected domain $D$ started from a boundary point aimed at an interior point with force points on the boundary is defined by a conformal map from $\D$ onto $D$. 
The targeted interior point actually acts as another force point with force value $\kappa-6-\sum_{j=1}^n \rho_j$ (cf.\ \cite{SW}).

At the end of this subsection, we briefly recall chordal Loewner equation, chordal SLE$_\kappa$, and $2$-SLE$_\kappa$. Let $\ha w \in C([0,T),\R)$ for some $T\in(0,\infty]$. The chordal Loewner equation driven by $\ha w$  is
$$\pa_t g_t(z)= \frac 2{g_t(z)-\ha w(t)},\quad 0\le t<T;\quad g_0(z)=z.$$
For each $t\in[0,T)$, let $K_t$ be the set of $z\in\HH$ such that the solution $g_{\cdot}(z)$ blows up before or at $t$ (so that $g_t$ is well defined on $\HH\sem K_t$). Then we call $g_t$ and $K_t$ the chordal Loewner maps and hulls, respectively, driven by $\ha w$. It turns out that, for each $t$, $K_t$ is a bounded and relatively closed subset of $\HH$, and $g_t$ maps $\HH\sem K_t$ conformally onto $\HH$. If for every $t\in[0,T)$, $g_t^{-1}$ as a conformal map from $\HH$ onto $\HH\sem K_t$ extends continuously to $\lin\HH$, and $\eta(t):=g_t^{-1}({\ha w(t)})$, $0\le t<T$, is continuous in $t$, then we say that $\eta$ is a chordal Loewner curve driven by $\ha w$. 

If $\ha w(t)=\sqrt\kappa B(t)$, $0\le t<\infty$, where $\kappa>0$ and $B(t)$ is a standard Brownian motion, then the chordal Lowner curve $\eta$ driven by $\ha w$ is known to exist, and is called a chordal SLE$_\kappa$ curve in $\HH$ from $0$ to $\infty$. In fact, we have $\eta(0)={\ha w(0)}=0$ and $\lim_{t\to\infty} \eta(t)=\infty$ (\cite{RS}). If $D$ is a simply connected domain with two distinct marked boundary points (prime ends) $a$ and  $b$, the chordal SLE$_\kappa$ curve in $D$ from $a$ to $b$ is defined to be the conformal image of a chordal SLE$_\kappa$ curve in $\HH$ from $0$ to $\infty$ under a conformal map from $(\HH;0,\infty)$ onto $(D;a,b)$.

For any $\kappa>0$, both radial SLE$_\kappa$ and chordal SLE$_\kappa$ satisfy conformal invariance and Domain Markov Property (DMP). The DMP means that if $\eta$ is a radial (resp.\ chordal) SLE$_\kappa$ curve in $D$ from $a$ to $b$, and $T$ is a stopping time, then conditionally on the part of $\eta$ before $T$ and the event that $\eta$ does not reach $b$ at time $T$, the part of $\eta$ after $T$ is a radial (resp.\ chordal) SLE$_\kappa$ curve from $\eta(T)$ to $b$ in one connected component of $D\sem \eta([0,T])$.
If $\kappa\in(0,8)$, chordal SLE$_\kappa$ satisfies reversibility: the time-reversal of a chordal SLE$_\kappa$ curve in $D$ from $a$ to $b$ is a chordal SLE$_\kappa$ curve in $D$ from $b$ to $a$, up to a time-change (\cite{reversibility,MS3}).

Let $D$ be a simply connected domain with distinct boundary points  $a_1,b_1,a_2,b_2$ such that $a_1$ and $a_2$ together separate $b_1$ from $b_2$ on $\pa D$ (and vice versa). Let $\kappa\in(0,8)$. A $2$-SLE$_\kappa$ in $D$ with link pattern $(a_1, b_1;a_2, b_2)$ is a pair of random curves $(\eta_1,\eta_2)$ in $\lin D$ such that for $j=1,2$, $\eta_j$ connects $a_j$ with $b_j$, and conditionally on $\eta_{3-j}$, $\eta_j$ is a chordal SLE$_\kappa$ curve in the connected component of $D\sem \eta_{3-j}$ whose boundary contains $a_j$ and $b_j$. Because of reversibility, we do not need to specify the orientation of $\eta_1$ and $\eta_2$. If we want to emphasize the orientation, then we use an arrow like $a_1\to b_1$ in the link pattern.
The existence of $2$-SLE$_\kappa$ was proved in \cite{multiple} for $\kappa\in(0,4]$ using Brownian loop measure and in \cite{MS1,MS3} for $\kappa\in(4,8)$ using flow line theory. The uniqueness of $2$-SLE$_\kappa$ (for a fixed domain and link pattern) was proved in \cite{MS2} (for $\kappa\in(0,4]$) and \cite{MSW} (for $\kappa\in(4,8)$) using an ergodicity argument.

Using the DMP for chordal SLE$_\kappa$, it is easy to derive the following DMP for $2$-SLE$_\kappa$: If $(\eta_1,\eta_2)$ is a $2$-SLE$_\kappa$ in $D$ with link pattern $(a_1\to b_1;a_2,b_2)$, and if $T$ is a stopping time for $\eta_1$, then conditionally on the part of $\eta_1$ before $T$ and the event that $\eta_1$ neither reaches $b_1$ or disconnects $b_1$ from $a_2,b_2$ at time $T$, the rest part of $\eta_1$ and the complete $\eta_2$ form a $2$-SLE$_\kappa$  with link pattern $(\eta_1(T)\to b_1;a_2,b_2)$ in the connected component of $D\sem \eta_1([0,T])$ whose boundary contains $b_1,a_2,b_2$. We will have a stronger DMP later in Lemma \ref{DMP-bi-chordal}. 


\subsection{Hypergeometric SLE}\label{section-hSLE}
We now review the hypergeometric SLE defined earlier in \cite{kappa-rho} (called intermediate SLE$_\kappa(\rho)$ there) and \cite{QW}.
Let $\kappa\in(0,8)$. Let $F$ be the hypergeometric function $\,_2F_1(\frac 4\kappa,1-\frac 4\kappa;\frac 8\kappa,\cdot)$ in Theorem \ref{main-thm1}. Such $F$ is the solution of
\BGE x(1-x)F''(x)+\Big[\frac 8\kappa -2x\Big]F'(x)-\frac 4\kappa\Big(1-\frac 4\kappa\Big)F(x)=0.\label{ODE-hyper}\EDE
Since $\frac 8\kappa>(1-\frac 4\kappa)+\frac 4\kappa$, $F$ extends continuously to $1$ with
$F(1)=\frac{\Gamma(\frac 8\kappa)\Gamma(\frac 8\kappa-1)}{\Gamma(\frac 4\kappa)\Gamma(\frac {12}\kappa-1)}>0$.
Using (\ref{ODE-hyper}) one can prove that $F$ is positive on $[0,1]$. Then we let $G(x)=\kappa x\frac{F'(x)}{F(x)}$, $\til F(x)  =x^{\frac 2\kappa} F(x)$, and $\til G(x)=\kappa x\frac{\til F'(x)}{\til F(x)}=G(x)+2$.

\begin{Definition}
   Let $0,v_1,v_2\in\R$ be such that $0<v_1<v_2$ or $0>v_1>v_2$. Let $B(t)$ and $B'(t)$ be two independent standard real Brownian motion. Suppose $\ha w,\ha v_1,\ha v_2\in C([0,\infty),\R)$ satisfy the following properties. There is $T\in(0,\infty]$ such that $\ha w(t)$ and $\ha v_j(t)$, $0\le t<T$, $j=1,2$, are continuous random process such that they together solve the following SDE with the maximal solution interval and respective initial values $0$ and $v_j$, $j=1,2$:
   \begin{align*}
     d\ha w(t)=&\sqrt\kappa dB(t)+\Big(\frac{1}{\ha w(t)-\ha v_1(t)}-\frac 1{\ha w(t)-\ha v_2(t)}\Big)\til G\Big(\frac{\ha w(t)-\ha v_1(t)}{\ha w(t)-\ha v_2(t)}\Big)dt;\\
        d\ha v_j(t)=&\frac{2dt}{\ha v_j(t)-\ha w(t)}, \quad j=1,2.
   \end{align*}
Moreover, if $T<\infty$, then $\ha w(T+t)=\ha w(T)+\sqrt\kappa B'(t)$, $0\le t<\infty$. Then the chordal Loewner curve driven by $\ha w$ is called a full hSLE$_\kappa$ curve in $\HH$ from $0$ to $\infty$  with force points $v_1,v_2$.

If $f$ maps $\HH$ conformally onto a simply connected domain $D$, then the $f$-image of a full hSLE$_\kappa$ curve in $\HH$ from $0$ to $\infty$ with force points $v_1,v_2$ is called a full hSLE$_\kappa$ curve in $D$ from $f(0)$ to $f(\infty)$ with force points $f(v_1),f(v_2)$.
\end{Definition}

\begin{Remark}
In the definition of full hSLE$_\kappa$ in $\HH$, if $\kappa\in(0,4]$, then a.s.\ $T=\infty$, and we do not need the $B'$ in the definition. If $\kappa\in(4,8)$, then a.s.\ $T<\infty$; and $\eta(t)$ tends to some point on $\R$ between $\infty$ and $v_2$. The assumption that  $\ha w(T+t)=\ha w(T)+\sqrt\kappa B'(t)$, $0\le t<\infty$,  means that  given the part of $\eta$ up to $T$, the rest of  $\eta$ is a chordal SLE$_\kappa$ curve from $\eta(T)$ to $\infty$ in the remaining domain. In both cases, a full hSLE$_\kappa$ curve always ends at its target.
\end{Remark}


We now describe hSLE using radial Loewner equation.
Let   $w_0,v_1,v_2,w_\infty\in\R$ be such that   $w_0>v_1>v_2>w_\infty>w_0-2\pi$ or $w_0<v_1<v_2<w_\infty<w_0+2\pi$. Let $B(t)$ be a standard Brownian motion.  Let $\ha w_0(t)$, $\ha w_\infty(t)$, and $\ha v_j(t)$, $j=1,2$, $0\le t<T$, be the solution of the SDEs:
\begin{align*}
  d\ha w_0(t)&=\sqrt\kappa dB(t)+ \frac{\kappa-6}2 \cot_2(\ha w_0(t)-\ha w_\infty(t))dt+\\
&+\frac 12 (\cot_2(\ha w_0(t)-\ha v_1(t))-\cot_2(\ha w_0(t)-\ha v_2(t))) \til G(R(t))dt,\\
R(t)&= \frac{\sin_2(\ha w_0(t) -\ha v_1(t) )\sin_2(\ha v_2(t)-\ha w_\infty(t))}{\sin_2(\ha w_0(t)-\ha v_2(t))\sin_2(\ha v_1(t)-\ha w_\infty(t))}\\
d\ha w_\infty(t)&=\cot_2(\ha w_\infty(t)-\ha w_0(t))dt,\\
d\ha v_j(t)&=\cot_2(\ha v_j(t)-\ha w_0(t))dt,\quad j=1,2,
\end{align*}
with initial values $w_0$, $w_\infty$, and $v_j$, $j=1,2$, respectively, such that $[0,T)$ is the maximal solution interval. Then we call the radial Loewner curve driven by $\ha w_0$ a radial hSLE$_\kappa$ curve in $\D$  from $e^{iw_0}$ to $e^{iw_\infty}$ with force points $e^{iv_1},e^{iv_2}$, viewed from $0$.

\begin{Proposition}
Let $w_0,w_\infty,v_1,v_2$ be as above.
  Suppose $\eta(t)$, $0\le t<T'$, is a full hSLE$_\kappa$ curve in $\D$  from $e^{iw_0}$ to $e^{iw_\infty}$ with force points $e^{iv_1},e^{iv_2}$. Let $T$ be the first time that $\eta$ separates $0$ from any of $e^{iw_0}, e^{iv_1},e^{iv_2}$. If such time does not exist, then we set $T=T'$. Then up to a time-change, $\eta(t)$, $0\le t< T$, is a radial hSLE$_\kappa $ curve in $\D$ from $e^{iw_0}$ to $e^{iw_\infty}$ with force points $e^{iv_1},e^{iv_2}$, viewed from $0$. \label{hSLE-coordinate}
\end{Proposition}
\begin{proof}
  This follows from the standard argument as in \cite{SW}.
\end{proof}

One important property of hSLE is its connection with $2$-SLE. If $(\eta_1,\eta_2)$ is a $2$-SLE$_\kappa$ in $D$ with link pattern $(a_1\to b_1;a_2\to b_2)$, then for $j=1,2$, $\eta_j$ is a full hSLE$_\kappa$ curve in $D$ from ${a_j}$ to ${b_j}$ with force points ${b_{3-j}}$ and ${a_{3-j}}$ (see e.g., \cite[Proposition 6.10]{Wu-hSLE}). The two curves $\eta_1$ and $\eta_2$  commute with each other in the following sense: if we run one curve, say $\eta_{3-j}$ up to a stopping time $T$ before reaching $b_{3-j}$ or separating $b_{3-j}$ from $b_j$ or $a_j$, and condition on this part of $\eta_{3-j}$, then the whole $\eta_j$ is a full hSLE$_\kappa$ curve from $a_j$ to $b_j$ in the remaining domain with force points $\eta_{3-j}(T)$ and $b_{3-j}$. This easily follows from the DMP of $2$-SLE. 

\subsection{Two-parameter Stochastic Processes} \label{section-two-parameter}
	We work on a measurable space $(\Omega,\F)$. Let ${\cal Q}$ denote the first quadrant $[0,\infty)^2$ with partial order $\le$ such that $\ulin t=(t_1,t_2)\le(s_1,s_2)= \ulin s$ iff $t_1\le s_1$ and $t_2\le s_2$. It has a minimal element $\ulin 0=(0,0)$. We write $\ulin t<\ulin s$ if $t_1<s_1$ and $t_2<s_2$. Moreover, we define $\ulin t\wedge \ulin s=(t_1\wedge s_1,t_2\wedge s_2)$. Given $\ulin t,\ulin s\in\cal Q$, we define $[\ulin t,\ulin s]=\{\ulin r\in{\cal Q}:\ulin t\le \ulin r\le \ulin s\}$.  For example, $[\ulin 0,\ulin t\wedge \ulin s]=[\ulin 0,\ulin t]\cap[\ulin 0,\ulin s]$. For $n\in\N$, we define $t^{\lfloor n\rfloor}=\frac{\lfloor 2^n t_1\rfloor}{2^n}$ and $t^{\lceil n\rceil}=\frac{\lceil 2^n t_1\rceil}{2^n}$ for $t\in[0,\infty)$, and $\ulin t^{\lfloor n\rfloor}=(t_1^{\lfloor n\rfloor},t_2^{\lfloor n\rfloor}) $ and $\ulin t^{\lceil n\rceil}=( t_1^{\lceil n\rceil}, t_2^{\lceil n\rceil}) $ for $\ulin t=(t_1,t_2)\in\cal Q$. Note that $\ulin t^{\lfloor n\rfloor}, \ulin t^{\lceil n\rceil}\in\cal Q$ and $\ulin t^{\lfloor n\rfloor}\le \ulin t\le \ulin t^{\lceil n\rceil}$.

	\begin{Definition}
	A family of sub-$\sigma$-fields $(\F_{\ulin t})_{\ulin t\in\cal Q}$ of $\F$ is called a ${\cal Q}$-indexed filtration if $\F_{\ulin t}\subset \F_{\ulin s}$ whenever $\ulin t\le \ulin s$. A family of random variables $(X({\ulin t}))_{\ulin t\in\cal Q}$ defined on $(\Omega,\F)$ is called an $(\F_{\ulin t})_{\ulin t\in\cal Q}$-adapted process if for any $\ulin t\in\cal Q$, $X({\ulin t})$ is $\F_{\ulin t}$-measurable. It is called continuous if   $\ulin t\mapsto X({\ulin t})$ is sample-wise continuous.
	\end{Definition}
\begin{Definition}
A random map $\ulin T:\Omega\to[0,\infty]^2$ is called an $(\F_{\ulin t})_{\ulin t\in\cal Q}$-stopping time if for any deterministic $\ulin t\in{\cal Q}$, $\{\ulin T\le \ulin t\}\in \F_{\ulin t}$.   
For such $\ulin T$, we define a new $\sigma$-field $\F_{\ulin T}$ by
$$\F_{\ulin T}=\{A\in\F:A\cap \{\ulin T\leq \ulin t\}\in \F_{\ulin t},\quad \forall \ulin t\in{\cal Q}\}.$$
The stopping time $\ulin T$ is called finite if it takes values in $\cal Q$, and is called bounded if there is a deterministic $\ulin t\in \cal Q$ such that $\ulin T\le \ulin t$.
\end{Definition}
	
	Note that any deterministic number $\ulin t\in\cal Q$ is an $(\F_{\ulin t})_{\ulin t\in\cal Q}$-stopping time, and the $\F_{\ulin t}$ defined by considering $\ulin t$ as a stopping time agrees with the $\F_{\ulin t}$ as in the filtration.
	
	\begin{Lemma}
		Let $\ulin T$ and $\ulin S$ be two $(\F_{\ulin t})_{\ulin t\in\cal Q}$-stopping times. Then (i)  $\{\ulin T\le \ulin S\} \in\F_{\ulin S}$; (ii) if $\ulin S$ is a deterministic time $\ulin s\in\cal Q$, then $\{\ulin T\le \ulin S\} \in\F_{ \ulin T}$; and (iii) if $f$ is an $\F_{ \ulin T}$-measurable function, then ${\bf 1}_{\{\ulin T\le \ulin S\}}f$ is $\F_{ \ulin S}$-measurable. In particular, if $\ulin T\le \ulin S$, then $\F_{ \ulin T}\subset \F_{ \ulin S}$. \label{T<S}
	\end{Lemma}
	\begin{proof}
	(i)	Let $\ulin t=(t_1,t_2)\in\cal Q$. We have
		$$\{\ulin T\le \ulin S\}\cap \{\ulin S\le \ulin t\}=\{\ulin T\le \ulin t\}\cap \{\ulin S\le \ulin t\}\cap \{\ulin T\not\le \ulin S\}^c$$
		$$=\{\ulin S\le \ulin t\}\cap \{\ulin T\le \ulin t\}\cap [(\{S_1<T_1\le t_1\}\cap \{T_2\vee S_2\le t_2\})\cup (\{T_1\vee S_1\le t_1\}\cap \{S_2<T_2\le t_2\})].$$
		Since $\ulin T$ and $\ulin S$ are stopping times, $\{\ulin S\le \ulin t\}\cap \{\ulin T\le \ulin t\}\in\F_{\ulin t}$. Since we may write
		$$\{S_1<T_1\le t_1\}\cap \{T_2\vee S_2\le t_2\}=\bigcup_{t_1'\in\Q\cap (0,t_1)}\{S_1\le t_1'<T_1\le t_1\}\cap \{T_2\vee S_2\le t_2\}$$
		$$=\bigcup_{t_1'\in\Q\cap (0,t_1)} \{\ulin T\le \ulin t\}\cap \{\ulin S\le (t_1',t_2)\}\cap \{\ulin T\le (t_1',t_2)\}^c,$$
		we get $\{S_1<T_1\le t_1\}\cap \{T_2\vee S_2\le t_2\}\in \F_{\ulin t}$. Similarly, $\{T_1\vee S_1\le t_1\}\cap \{S_2<T_2\le t_2\}\in\F_{\ulin t}$. Combining, we get $\{\ulin T\le \ulin S\}\cap \{\ulin S\le \ulin t\}\in\F_{\ulin t}$. Thus,  $\{\ulin T\le \ulin S\}\in\F_{\ulin S}$.

(ii) If $\ulin S=\ulin s$ for some $\ulin s\in\cal Q$, then $\{\ulin T\le \ulin S\} \in\F_{\ulin T}$ because for any $\ulin t\in\cal Q$,
$$\{\ulin T\le \ulin S\}\cap \{\ulin T\le \ulin t\}=\{\ulin T\le \ulin s\wedge \ulin t\}\in\F_{\ulin s\wedge \ulin t}\subset \F_{\ulin t}.$$
		
	(iii) By monotone convergence, it suffices to consider the case that $f={\bf 1}_A$, where $A\in \F_{\ulin T}$. Then for any $\ulin t \in\cal Q$,
		$$A\cap \{\ulin T\le \ulin S\}\cap \{\ulin S\le \ulin t\}=(A\cap \{\ulin T\le \ulin t\})\cap (\{\ulin T\le \ulin S\}\cap \{\ulin S\le \ulin t\})\in \F_{ \ulin t}.$$
		So $A\cap \{\ulin T\le \ulin S\}\in \F_{\ulin S}$, which implies that ${\bf 1}_{\{\ulin T\le \ulin S\}}f={\bf 1}_{A\cap \{\ulin T\le \ulin S\}}$ is $\F_{ \ulin S}$-measurable.
	\end{proof}
	
	\begin{Remark}
		In general, we do not have $\{\ulin T\le \ulin S\}\in\F_{\ulin T}$ unless $\ulin S$ is deterministic or separable (see Definition \ref{separable} and Lemma \ref{separable-lemma}). %
	\end{Remark}

	\begin{Lemma}
		Let $(X_{\ulin t})_{\ulin t\in\cal Q}$ be a continuous $(\F_{\ulin t})_{\ulin t\in\cal Q}$-adapted process, and let $\ulin T$ be an $(\F_{\ulin t})_{\ulin t\in\cal Q}$-stopping time. Then $X_{\ulin T}$ defined on $\{\ulin T\in\cal Q\}$ is $\F_{\ulin T}$-measurable. \label{measurable}
	\end{Lemma}
	\begin{proof}
		Since $\ulin T^{\lfloor n\rfloor}\uparrow \ulin T$, as $n\to\infty$, by the continuity of $X$, it suffices to show that for every $n\in\N$, $X_{\ulin T^{\lfloor n\rfloor}}$  is $\F_{\ulin T}$-measurable.
		For a fixed $n\in\N$, since  $\ulin T^{\lfloor n\rfloor}|_{\{\ulin T\in\cal Q\}}$ takes values in $(\frac{\Z}{2^n})^2$; and for every $\ulin t\in (\frac{\Z}{2^n})^2$, by Lemma \ref{T<S} (i,ii),
		$\{\ulin T^{\lfloor n\rfloor}=\ulin t\}=\{\ulin t\le  \ulin T\}\cap \{\ulin T< \ulin t+(\frac{1}{2^{n}},\frac{1}{2^{n}})\}\in \F_{\ulin T}$,
		it suffices to show that $X_{\ulin T^{\lfloor n\rfloor}}$ restricted to $\{\ulin T^{\lfloor n\rfloor}=\ulin t\}$ is $\F_{\ulin T}$-measurable.
To see this, we may write
		${\bf 1}_{\{\ulin T^{\lfloor n\rfloor}=\ulin t\}}X_{\ulin T^{\lfloor n\rfloor}}={\bf 1}_{\{\ulin T^{\lfloor n\rfloor}=\ulin t\}}{\bf 1}_{\{\ulin t\le \ulin T\}} X_{\ulin t}$. Since $X_{\ulin t}$ is $\F_{\ulin t}$-measurable, by Lemma \ref{T<S}, ${\bf 1}_{\{\ulin t\le \ulin T\}} X_{\ulin t}$ is $\F_{\ulin T}$-measurable. So ${\bf 1}_{\{\ulin T^{\lfloor n\rfloor}=\ulin t\}}X_{\ulin T^{\lfloor n\rfloor}}$ is $\F_{\ulin T}$-measurable, as desired.
	\end{proof}
	
	From now on, we fix a  a probability measure $\PP$ on $(\Omega,\F)$, and let $\EE$ denote the corresponding expectation.
	
	\begin{Definition}
	An  $(\F_{\ulin t})_{\ulin t\in\cal Q}$-adapted process $(X_{\ulin t})_{\ulin t\in\cal Q}$ is called an  $(\F_{\ulin t})_{\ulin t\in\cal Q}$-martingale (w.r.t.\ $ \PP$) if every $X_{\ulin t}$ is integrable, and for any $\ulin s\le \ulin t\in\cal Q$, $\EE[X_{\ulin t}|\F_{\ulin s}]=X_{\ulin s}$. If there is $X\in L^1(\Omega,\F,\PP)$ such that $X_{\ulin t}=\EE[X|\F_{ \ulin t}]$ for every $\ulin t\in\cal Q$, then it is clear that  $(X_{\ulin t})$ is an $(\F_{\ulin t})$-martingale. We call such $(X_{\ulin t})$  an $X$-Doob martingale or simply a Doob martingale.
	\end{Definition}
		
	\begin{Lemma} [Optional Stopping Theorem]
		Let $(X_{\ulin t})_{\ulin t\in\cal Q}$ be a continuous $(\F_{\ulin t})_{\ulin t\in\cal Q}$-martingale.
The following are true. (i) If $(X_{\ulin t})$ is an $X$-Doob martingale for some $X\in L^1$, then for any finite $(\F_{\ulin t})_{\ulin t\in\cal Q}$-stopping time $\ulin T$, $X_T=\EE[X|\F_{ \ulin T}]$. (ii) If $\ulin T\le \ulin S$ are two bounded $(\F_{\ulin t})_{\ulin t\in\cal Q}$-stopping times, then $\EE[X_{\ulin S}|\F_{\ulin T}]=X_{\ulin T}$.\label{OST}
	\end{Lemma}
	\begin{proof}
(i) Assume that $(X_{\ulin t})$ is an $X$-Doob martingale.
	First, we assume that $\ulin T$ takes values in $(\frac{\Z}{2^n})^2$ for some $n\in\N$. 		Since $X_{\ulin T}$ is $\F_{ \ulin T}$-measurable by Lemma \ref{measurable}, it suffices to show that, for any $A\in\F_{\ulin T}$, $\EE[{\bf 1}_A X_{\ulin T}]=\EE[{\bf 1}_A X]$. We now fix $A\in\F_{\ulin T}$. For any $\ulin t\in {\cal Q}\cap (\frac{\Z}{2^n})^2$, since
	 $A\cap \{\ulin T=\ulin t\}\in\F_{\ulin t}$,		using $\EE[X|\F_{\ulin t}]=X_{\ulin t}$, we get $\EE[{\bf 1}_{A\cap \{\ulin T=\ulin t\}} X_{\ulin t}]=\EE[{\bf 1}_{A\cap \{\ulin T=\ulin t\}} X ]$.		Summing up over  $\ulin t\in {\cal Q}\cap (\frac{\Z}{2^n})^2$, we get $\EE[{\bf 1}_A X_{\ulin T}]=\EE[{\bf 1}_A X ]$ in this special case.
		
		Now we consider the general case. Note that for every $n\in\N$, $\ulin T^{\lceil n\rceil}$ takes values in $(\frac{\Z}{2^n})^2$, and is a stopping time because for any $\ulin t\in \cal Q$, $\{\ulin T^{\lceil n\rceil}\le \ulin t\}=\{T\le \ulin t^{\lfloor n\rfloor}\}\in\F_{\ulin t^{\lfloor n\rfloor}}		\subset \F_{\ulin t}$. Applying the special case to $\ulin T^{\lceil n\rceil}$, we get $\EE[X |\F_{\ulin T^{\lceil n\rceil}}]=X_{\ulin T^{\lceil n\rceil}}$.  Since $ \ulin T^{\lceil n\rceil}\downarrow T$ as $n\to\infty$. By the continuity of $X$, we have $X_{\ulin T^{\lceil n\rceil}}\to X_{\ulin T}$. Since $\F_{\ulin T }\subset \F_{\ulin T^{\lceil n\rceil}}$ by Lemma \ref{T<S}, a standard argument involving uniform integrability shows that $\EE[X |\F_{\ulin T }]=X_{\ulin T}$.
		
		(ii) First assume that $\ulin S$ is a constant $\ulin s\in\N^2$. Then $\ulin S\ge \ulin T^{\lceil n\rceil}$ for all $n\in\N$. Using the same argument as in (i) with $X_{\ulin S}$ in place of $X$, we get $\EE[X_{\ulin S}|\F_{ \ulin T}]=X_{\ulin T}$.
		
		Finally, we consider the general case. Since  $\ulin S$ is bounded, there is $\ulin r\in {\cal Q}\cap \N^2$ such that $\ulin T\le \ulin S\le \ulin r$. Let   $A\in\F_{ \ulin T}\subset\F_{ \ulin S}$. From the special case of (ii), we get $\EE [{\bf 1}_A X_{\ulin S}]=\EE [{\bf 1}_A X_{\ulin r}]=\EE [{\bf 1}_A X_{\ulin T}]$, which implies that $\EE[X_{\ulin S}|\F_{\ulin T}]=X_{\ulin T}$.
	\end{proof}

	\begin{Definition}
Suppose that there are two filtrations  $(\F^1_{t_1})$ and $(\F^2_{t_2})$ such that $\F_{(t_1,t_2)}=\F^1_{t_1}\vee \F^2_{t_2}$, $(t_1,t_2)\in\cal Q$. Then we say that $(\F_{ \ulin t})_{\ulin t\in\cal Q}$ is a separable filtration generated by $(\F^1_{t_1})$ and $(\F^2_{t_2})$. For such a separable filtration, if $T_j$ is an $(\F^j_t)$-stopping time, $j=1,2$, then $(T_1,T_2)$ is clearly an $(\F_{ \ulin t})_{\ulin t\in\cal Q}$-stopping time, and is called separable  (w.r.t.\ $(\F^1_{t_1})$ and $(\F^2_{t_2})$).  \label{separable}
\end{Definition}

\begin{Lemma}
Let $\ulin T$ and $\ulin S$ be two stopping times w.r.t.\ a separable filtration $(\F_{\ulin t})_{\ulin t\in\cal Q}$. If $\ulin S$ is separable, then $\{\ulin T\le \ulin S\}\in\F_{ \ulin T}$. \label{separable-lemma}
\end{Lemma}
\begin{proof}
	For any $\ulin t\in\cal Q$, $\{\ulin T\le \ulin S\}\cap\{\ulin T\le \ulin t\}=\{\ulin T\le \ulin S\wedge\ulin t\}\in \F_{\ulin S\wedge \ulin t}\subset \F_{ \ulin t}$. So $\{\ulin T\le \ulin S\}\in\F_{ \ulin T}$.
	Here we use  Lemma \ref{T<S} (i,ii) and the fact that $\ulin S\wedge \ulin t$ is an  $(\F_{\ulin t})_{\ulin t\in\cal Q}$-stopping time (and so $\F_{ \ulin S\wedge \ulin t}$ is well defined), which follows from the assumption that $\ulin S$ is separable.
\end{proof}

	\begin{Definition}
		A relatively open subset $\cal R$ of $\cal Q$ is called a history complete region, or simply an HC region, if for any $\ulin t\in \cal R$, we have $[\ulin 0, \ulin t]\subset\cal R$. Given an HC region $\cal R$, we may define two functions $T^{\cal R}_1,T^{\cal R}_2:[0,\infty)\to[0,\infty]$ such that
		$$[0,T^{\cal R}_1(t_2))=\{s_1\ge 0: (s_1,t_2)\in{\cal R}\},\quad [0,T^{\cal R}_2(t_1))=\{s_2\ge 0:(t_1,s_2)\in{\cal R}\},\quad t_1,t_2\ge 0.$$
A map $\cal D$ from $\Omega$ into the space of HC regions is called an $(\F_{ \ulin t})_{\ulin t\in\cal Q}$-stopping region if for any $\ulin t\in\cal Q$, $\{\omega\in\Omega:\ulin t\in {\cal D}(\omega)\}\in\F_{ \ulin t}$. A random function $X({\ulin t})$ with a random domain $\cal D$ is called an $(\F_{ \ulin t})_{\ulin t\in\cal Q}$-adapted HC process if $\cal D$ is an $(\F_{ \ulin t})_{\ulin t\in\cal Q}$-stopping region, and for every  $\ulin t\in\cal Q$, $X_{\ulin t}$ restricted to $\{\ulin t\in\cal D\}$ is $\F_{ \ulin t}$-measurable. \label{Def-HC}
	\end{Definition}

\section{Ensemble of Two Radial Loewner Chains}\label{section-Ensemble}
In this section, we consider two (deterministic or stochastic) radial Loewner curves growing in the same domain $\D$, and study how they interact with each other. The calculation used in Sections \ref{deterministic} and \ref{section-Probability-measures} is standard, which is analog of the one in \cite[Section 5]{LSW-8/3}. We fix distinct points $e^{iw_1},e^{iv_1},e^{iw_2},e^{iv_2}$ ordered clockwise on $\TT$, and consider several probability measures $\PP_{iB},\PP_{ih},\PP_{i4},\PP_{ch},\PP_{c4}$ on the space of pairs of driving functions $(\ha w_1,\ha w_2)$, which generate radial Loewner curves $(\eta_1,\eta_2)$. The letters $i$ and $c$ respectively  stand for independence and communication. The letter $B$ stands for Brownian motion. Under $\PP_{iB}$, $\ha w_j=w_1+\sqrt\kappa B_j$, $j=1,2$, where $B_1$ and $B_2$ are independent standard Brownian motions. The letter $h$ stands for hSLE. Under $\PP_{ih}$ or $\PP_{ch}$, each of $\eta_1$ and $\eta_2$ is an hSLE$_\kappa$. They are independent under  $\PP_{ih}$, and commute with each other under  $\PP_{ch}$. The number $4$ stands for $4$-SLE$_\kappa$. Under $\PP_{i4}$ or $\PP_{c4}$, for $j=1,2$, $\eta_j$ is a radial SLE$_\kappa(2,2,2)$ curve from $e^{iw_j}$ to $0$ with force points $e^{iv_1},e^{iv_2},e^{iw_{3-j}}$. They are independent under  $\PP_{i4}$, and commute with each other under $\PP_{c4}$. In the latter case the two curves intersect only at $0$, and if we further grow two independent chordal SLE$_\kappa$ curves in the two complement domains from $e^{iv_1}$ and $e^{iv_2}$ to $0$, we then get a $4$-SLE$_\kappa$ that connect $0$ to $e^{iw_1},e^{iw_2},e^{iv_1},e^{iv_2}$. We are not going to prove this fact since it is not used in the proof of our main theorems.

The  $\PP_{iB},\PP_{ih},\PP_{i4}$ are product measures, and are easy to construct. The measure $\PP_{c4}$ and $\PP_{ch}$ are constructed using the stochastic coupling technique in \cite{reversibility}. We define two-time-parameter processes $M_{iB\to c4}$ and $M_{iB\to ch}$ in (\ref{M-i-4'},\ref{M-i-2}), which respectively act as the Radon-Nikodym derivative processes of $\PP_{c4}$ and $\PP_{ch}$ against $\PP_{iB}$. We show that, under $\PP_{iB}$, they are two-variable local martingales: when one variable is fixed, they are local martingales in the other variable with SDEs (\ref{pajM4},\ref{pajM2}). We then use these local martingales to construct $\PP_{c4}$ and $\PP_{ch}$ in Section \ref{localization}  following the approach in \cite{reversibility}. We also define $\PP_2$ as the joint law of the two radial Loewner driving functions for the two curves in a $2$-SLE$_\kappa$ with link pattern $(e^{iw_1},e^{iv_1};e^{iw_2},e^{iv_2})$. Then $\PP_2$ agrees with $\PP_{ch}$ at any stopping time $\ulin T=(T_1,T_2)$ such that $\eta_1[0,T_1]\cap \eta_2[0,T_2]=\emptyset$. Combining this fact with the construction of $\PP_{c4}$ and $\PP_{ch}$, we obtain the main result (\ref{RN-4to2}) of this section: the Radon-Nikodym derivative   of $\PP_{2}$ against $\PP_{c4}$ at the time $\ulin T$ is $M_{c4\to ch}(\ulin T)/M_{c4\to ch}(\ulin 0)$, where $M_{c4\to ch}=M_{iB\to ch}/M_{iB\to c4}$.

\subsection{Deterministic ensemble}\label{deterministic}
Let $w_1,w_2,v_1,v_2\in\R$ be such that $w_1>v_1>w_2>v_2>w_2-2\pi$. For $j=1,2$, let $\ha w_j\in C([0,S_j),\R)$ be a radial Loewner driving function with $\ha w_j{(0)}=w_j $.  
Suppose $\ha w_j$ generates radial Loewner hulls $K_j(t)$, 
 radial Loewner maps $g_j(t,\cdot)$, covering Loewner hulls $\til K_j(t)$ and covering radial Loewner maps $\til g_j(t,\cdot)$, $0\le t<S_j$.  
 Let $\cal D$ denote the set of $(t_1,t_2)\in[0,S_1)\times [0,S_2)$ such that $\lin{K_1(t_1)}\cap \lin{K_2(t_2)}=\emptyset$ and $e^{iv_1 },e^{iv_2 }\not\in \lin{K_1(t_1)}\cup \lin{K_2(t_2)}$. Then $\cal D$ is an HC region as in Definition \ref{Def-HC}, and we may define functions $T^{\cal D}_1$ and $T^{\cal D}_2$.
For $(t_1,t_2)\in\cal D$, let $K(t_1,t_2)=K_1(t_1)\cup K_2(t_2)$. Then $K(t_1,t_2)$ is also an $\D$-hull. Let
 $g((t_1,t_2),\cdot)=g_{K(t_1,t_2)}$, and $\mA(t_1,t_2)=\dcap(K(t_1,t_2))$. For $(t_1,t_2)\in\cal D$ and $j\ne k\in\{1,2\}$, let 
$K_{j,t_k}(t_j)=g_k(t_k,K_j(t_j))$, and $g_{j,t_k}(t_j,\cdot)=g_{K_{j,t_k}(t_j)}$. Then we have
\BGE g_{1,t_2}(t_1,\cdot)\circ g_2(t_2,\cdot)=g((t_1,t_2),\cdot)=g_{2,t_1}(t_2,\cdot)\circ g_1(t_1,\cdot).\label{circ'}\EDE
Let $\til K(t_1,t_2),\til K_{j,t_k}(t_j)\subset\HH$ be the pre-images of $K(t_1,t_2),K_{j,t_k}(t_j)$, respectively, under the map $e^i$.
Let $\til g((t_1,t_2),\cdot)$, $(t_1,t_2)\in\cal D$, be the unique family of maps, such that $\til g((t_1,t_2),z)$ is joint continuous in $t_1,t_2,z$; $\til g((0,0),\cdot)=\id$; and for each $(t_1,t_2)\in\cal D$, $\til g((t_1,t_2),\cdot):\HH\sem  \til K(t_1,t_2)\conf\HH$, and $e^i\circ \til g((t_1,t_2),\cdot)=  g((t_1,t_2),\cdot)\circ e^i$.  Define $\til g_{1,t_2}(t_1,\cdot)$ and $\til g_{2,t_1}(t_2,\cdot)$, $(t_1,t_2)\in\cal D$, similarly. Using (\ref{circ'}) we get
\BGE \til g_{1,t_2}(t_1,\cdot)\circ \til g_2(t_2,\cdot)=\til g((t_1,t_2),\cdot)=\til g_{2,t_1}(t_2,\cdot)\circ \til g_1(t_1,\cdot).\label{circ}\EDE
Note also that $\til g((t_1,0),\cdot)=\til g_1(t_1,\cdot)$ and $\til g((0,t_2),\cdot)=\til g_2(t_2,\cdot)$. So $\til g_{1,t_2}(0,\cdot)$ (resp.\ $\til g_{2,t_1}(0,\cdot)$) is an identity if $(0,t_2)\in\cal D$ (resp.\ $(t_1,0)\in\cal D$).
Let $(t_1,t_2)\in{\cal D}$. From the assumption on $e^{iv_1},e^{iv_2}$, $g((t_1,t_2),\cdot)$ extends conformally to neighborhoods of $e^{iv_1}$ and $e^{iv_2}$. Thus,   $\til g((t_1,t_2),\cdot)$ extends conformally to neighborhoods of $v_1 $ and $v_2 $. Then we define real valued functions
\BGE V_j(t_1,t_2)=\til g((t_1,t_2),v_j ),\quad V_{j,1}(t_1,t_2)=\til g'((t_1,t_2),v_j ),\quad (t_1,t_2)\in\cal D.\label{Vj}\EDE
Here and below the prime means the partial derivative w.r.t.\ the last variable. Fix $j\ne k\in\{1,2\}$. From Lemma \ref{Loewner-KL} we know that $g_{k,t_j}(t_k,\cdot)$ extends conformally to a neighborhood of $e^{i\ha w_j(t_j)}$. Thus,  $\til g_{k,t_j}(t_k,\cdot)$ extends conformally to a neighborhood of ${\ha w_j(t_j)}$. Now we define
\BGE W_j(t_1,t_2)=\til g_{k,t_j}(t_{k},\ha w_j(t_j)),\quad W_{j,h}(t_1,t_2)=\til g_{k,t_j}^{(h)}(t_{k},\ha w_j(t_j)),\quad (t_1,t_2)\in\cal D.\label{Wj}\EDE
Here and below the superscript $(h)$ means the $h$-th partial derivative w.r.t.\ the last variable.
We then have $W_1>V_1>W_2>V_2>W_1-2\pi$.
For a function $X$ defined on $\cal D$, $k\in\{1,2\}$, and $t_k\ge 0$, we let $X^{(k,t_k)}$ be the function defined on $[0,T^{\cal D}_j(t_k))$ obtained from $X$ by fixing the $k$-th variable to be $t_k$. Since $\til g_{k,t_j}(0,\cdot)$ are identity maps, we get 
\BGE  W_{j,1}^{(k,0)}\equiv 1,\quad W_{j,2}^{(k,0)}=W_{j,3}^{(k,0)}\equiv 0,\quad j\ne k\in\{1,2\}.\label{Wat0}\EDE

Using Lemma \ref{Loewner-KL} we know that, for any $t_k\ge 0$, $K_{j,t_k}(t_j)$ and $\til g_{j,t_k}(t_j,\cdot)$, $0\le t_j<T^{\cal D}_j(t_k)$, are radial Loewner hulls and covering radial Loewner maps, respectively, driven by $W_j^{(k,t_k)}$ with speed $|W_{j,1}^{(k,t_k)}|^2$. This means that
\BGE \pa_j \mA=\pa_j(\dcap(K_k(t_k))+\dcap(K_{j,t_k}(t_j)))=W_{j,1} ^2\pa t_j;\label{pamA}\EDE
\BGE \pa_j\til g_{j,t_k}(t_j,z)=W_{j,1}(t_1,t_2)^2\cot_2(\til g_{j,t_k}(t_j,z)-W_j(t_1,t_2))\pa t_j.\label{eqn1}\EDE
Plugging $z=\ha w_k(t_k)$ and $z=\til g_k(t_k,v_s )$, respectively, into (\ref{eqn1}), we get
\BGE\pa_j W_k=W_{j,1}^2 \cot_2(W_k-W_j)\pa t_j,\quad \pa_j V_s=W_{j,1}^2\cot_2(V_s-W_j)\pa t_j,\quad s=1,2.\label{dWkVs}\EDE
Differentiating (\ref{eqn1}) w.r.t.\ $z$, we get
\BGE \frac{\pa_j\til g_{j,t_k}'(t_j,z)}{\til g_{j,t_k}'(t_j,z)}=W_{j,1}(t_1,t_2)^2\cot_2'(\til g_{j,t_k}(t_j,z)-W_j(t_1,t_2))\pa t_j.\label{eqn2}\EDE
 Plugging $z=\ha w_k(t_k)$ and $z=\til g_k(t_k,v_s )$, respectively, into (\ref{eqn2}), we get
\BGE \frac{\pa_j W_{k,1}}{W_{k,1}}=W_{j,1}^2\cot_2'(W_k-W_j)\pa t_j,\quad \frac{\pa_j V_{s,1}}{V_{s,1}}=W_{j,1}^2\cot_2'(V_s-W_j)\pa t_j,\quad s=1,2.\label{dWj1Vs1}\EDE
Since $\cot_2'(W_k-W_j)<0$, we see that $W_{k,1}$ is decreasing in $t_j$ and stays positive. From (\ref{Wat0}) we see that $W_{k,1}\in(0,1]$. Since $\mA(t_1,0)=t_1$ and $\mA(0,t_2)=t_2$, from (\ref{pamA}), we get
\BGE t_1\vee t_2\le  \mA(t_1,t_2)\le t_1+t_2,\quad (t_1,t_2)\in\cal D.\label{mA-est}\EDE
Let $Sg:=(\frac{g''}{g'})'-\frac 12(\frac{g''}{g'})^2$ denote the Schwarzian derivative of $g$, and let
\BGE W_{k,S}=S\til g_{j,t_k}(t_j,\ha w_k(t_k))=\frac{W_{k,3}}{W_{k,1}}-\frac 32 \Big(\frac{W_{k,2}}{W_{k,1}}\Big)^2.\label{WkS}\EDE
Differentiating (\ref{eqn2}) w.r.t.\ $z$, we get
$$ \pa_j \Big (\frac{\til g_{j,t_k}''(t_j,z)}{\til g_{j,t_k}'(t_j,z)}\Big )=W_{j,1}(t_1,t_2)^2\cot_2''(\til g_{j,t_k}(t_j,z)-W_j(t_1,t_2))\til g_{j,t_k}'(t_j,z)\pa t_j.$$
Further differentiating this equation w.r.t.\ $z$  and plugging $z=\ha w_k(t_k)$, we get
\BGE \pa_j W_{K,S}=W_{j,1}^2W_{k,1}^2 \cot_2'''(W_k-W_j)\pa t_j.\label{pajWKS}\EDE

Differentiating (\ref{circ}) w.r.t.\ $t_j$ and using (\ref{eqn1}), we get
$$\pa_{t_j} \til g_{k,t_j}(t_k,\til g_j(t_j,z))+ \til g_{k,t_j}'(t_k,\til g_j(t_j,z))  \cot_2(\til g_j(t_j,z)-\ha w_j(t_j ))$$ $$= W_{j,1}(t_1,t_2)^2\cot_2(\til g_{j,t_k}(t_j,\til g_k(t_k,z))-W_j(t_1,t_2)).$$
Let $\ha z=\til g_j(t_j ,z)$. Using (\ref{circ}) we get
\BGE \pa_{t_j} \til g_{k,t_j}(t_k,\ha z)=\til g_{k,t_j}'(t_k, \ha w_j(t_j))^2 \cot_2 (\til g_{k,t_j}(t_k,\ha z)-\til g_{k,t_j}(t_k, \ha w_j(t_j)))- \til g_{k,t_j}'(t_k,\ha z) \cot_2(\ha z-\ha w_j(t_j)).\label{haz}\EDE
Sending $\ha z\to \ha w_j(t_j)$, we get
\BGE \pa_{t_j}\til g_{k,t_j}(t_k,\ha z)|_{\ha z=\ha w_j(t_j)}=-3\til g_{k,t_j}''(t_k,\ha w_j(t_j))=-3W_{j,2};\label{-3}\EDE
Differentiating (\ref{haz}) w.r.t.\ $\ha z$ and then sending $\ha z\to \ha w_j(t_j)$, we get
\BGE \frac{\pa_{t_j} \til g_{k,t_j}'(t_k,\ha z)|_{\ha z=\ha w_j(t_j)}}{\til g_{k,t_j}'(t_k,\ha z)|_{\ha z=\ha w_j(t_j)}}=\frac 12 \Big(\frac{W_{j,2}}{W_{j,1}}\Big)^2-\frac 43 \frac{W_{j,3}}{W_{j,1}}-\frac 16 (W_{j,1}^2-1). \label{1/2-4/3}\EDE

Finally, suppose that $\ha w_1$ and $\ha w_2$ generate radial Loewner curves $\eta_1$ and $\eta_2$, respectively, and for any $j\ne k\in(1,2)$, and any $t_k\in[0,S_k)$, the radial Loewner process driven by $W_j^{(k,t_k)}$ with speed $|W_{j,1}^{(k,t_k)}|^2$ generates a radial Loewner curve $\eta_{j,t_k}$. Then we have $\eta_j(t_j)=g_k(t_k,\cdot)^{-1}(\eta_{j,t_k}(t_j))$, $0\le t_j<T^{\cal D}_j(t_k)$, where $g_k(t_k,\cdot)^{-1}$ is understood as the continuous extension of the original $g_k(t_k,\cdot)^{-1}$ from $\D$ to $\lin\D$.

\subsection{Two-variable local martingales}\label{section-Probability-measures}
We use the setup in the previous subsection. We view $(\ha w_1(t))_{0\le t<S_1}$ and $(\ha w_2(t))_{0\le t<S_2}$ as elements in $\Sigma:=\bigcup_{0<T\le\infty} C([0,T),\R)$. The space $\Sigma$ and a filtration $(\F_t)_{t\ge 0}$ were defined in \cite[Section 2]{decomposition}. Here is a brief review. For $f\in\Sigma$, let $T_f$ be such that $[0,T_f)$ is the domain of $f$.
For $0\le t<\infty$, the $\F_t$ is the $\sigma$-algebra on $\Sigma$ generated by the values of the function at the times before $t$. More precisely, $\F_t$ is the $\sigma$-algebra generated by $$\{f\in\Sigma: s<T_f, f(s)\in U\},\quad 0\le s\le t, U\in\cal{B}(\R).$$

Now we introduce randomness. Fix $\kappa\in(0,8)$ throughout. The boundary scaling exponent $\bb$ and central charge $\cc$ are defined by 
\BGE \bb=\frac{6-\kappa}{2\kappa},\quad \cc=\frac{(3\kappa-8)(6-\kappa)}{2\kappa}.\label{bc}\EDE
For $j=1,2$, we let $\PP_B^j$ denote the law of $w_j+\sqrt\kappa B(t)$, $0\le t<\infty$, where $B(t)$ is a standard Brownian motion, which is a probability measure on $(\Sigma,\F_{ \infty})$.  For $j=1,2$, let $\PP^j_4$ denote the law of the radial Loewner driving function with initial value $w_j$ for the radial SLE$_\kappa(2,2,2)$ curve in $\D$ started from $e^{iw_j}$ aimed at $0$ with force points $e^{iv_1},e^{iv_2},e^{iw_{3-j}}$. For $j=1,2$, let $\PP^j_h$ denote the law of the radial Loewner driving function with initial value $w_j$ for the radial hSLE$_\kappa$ curve in $\D$ from $e^{iw_j}$ to $e^{iv_j}$ with force points $e^{iv_{3-j}},e^{iw_{3-j}}$, viewed from $0$. 

From now on, when there is no ambiguity, we will not try to distinguish the law of a driving function and the law of the radial Loewner curve that it generates.
We will mainly work on the product measurable space, and use the notation in Section \ref{section-two-parameter}. We naturally have the following product measures: $\PP_{iB}:=\PP_B^1\times\PP_B^2$, $\PP_{i4}:=\PP_4^1\times \PP_4^2$, and $\PP_{ih}:=\PP_h^1\times\PP_h^2$. We use $\EE_{iB},\EE_{i4},\EE_{ih}$ to denote the corresponding expectations, respectively.

Now suppose $(\ha \eta_1,\ha \eta_2)$ is a $2$-SLE$_\kappa$ in $\D$ with link pattern $(e^{iw_1}\to e^{iv_1};e^{iw_2}\to e^{iv_2})$. Then $\ha \eta_j$ is a full hSLE$_\kappa$ curve in $\D$ from $e^{iw_j}$ to $e^{iv_j}$ with force points $e^{iv_{3-j}}$ and $e^{iw_{3-j}}$. For $j=1,2$, let $\eta_j$ be the part of $\ha \eta_j$ from $w_j$ up to its lifetime or the time that it separates $0$ from any of $e^{iv_j},e^{iw_{3-j}},e^{iv_{3-j}}$, if the later time exists. Then we may parametrize $\eta_j$ using radial capacity, and get a radial Loewner curve. By Proposition \ref{hSLE-coordinate},  $\eta_j$ is a radial hSLE$_\kappa$ curve in $\D$ from  $e^{iw_j}$ to $e^{iv_j}$ with force points $e^{iv_{3-j}},e^{iw_{3-j}}$, viewed from $0$. We use $\PP_{2}$ to denote the joint law of the radial driving functions for $\eta_1$ and $\eta_2$. Such measure $\PP_{2}$ is a coupling of $\PP_h^1$ and $\PP_h^2$, but is different from the product measure $\PP_{ih}$. Instead, $\eta_1$ and $\eta_2$ that jointly follow the law $\PP_{2}$ commute with each other in the following sense:  for any $j\in\{1,2\}$, conditionally on the part of $\eta_{3-j}$ up to a stopping time $\tau$ before its lifetime, if $g$ maps the remaining domain conformally onto $\D$ with $g(0)=0$ and $g'(0)>0$, then the $g$-image of $\eta_j$ up to the time that $\eta_j$ hits $\eta_{3-j}[0,\tau]$ is a radial hSLE$_\kappa$ curve in $\D$ from $g(e^{iw_j})$ to $g(e^{iv_j})$ with force points $g(e^{iv_{3-j}})$ and $g(\eta_{3-j}(\tau))$, viewed from $0$. The measure $\PP_{2}$ depends on the points $w_1,v_1,w_2,v_2$. When we want to emphasize the dependence, we use the symbol $\PP_{2}^{w_1,v_1;w_2,v_2}$.

For $j=1,2$, let $(\F^j_t)$ be the 
filtration generated by the $j$-th function as described at the beginning of this subsection. Let $(\F_{\ulin t})$ be the separable ${\cal Q}$-indexed filtration generated by $(\F^1_t)$ and $(\F^2_t)$.   Then $\cal D$ is an $(\F_{\ulin t})_{\ulin t\in\cal Q}$-stopping region, and
 for $j\ne k\in\{1,2\}$ and $h=1,2,3$, $\til g_j(t_j,\cdot)$, $w_j(t_j)$, $\til g((t_1,t_2),\cdot)$, $\til g_{j,t_k}(t_j,\cdot)$, $W_j(t_1,t_2)$, $W_{j,h}(t_1,t_2)$, $W_{j,S}(t_1,t_2)$, $V_j(t_1,t_2)$,  $V_{j,1}(t_1,t_2)$,   defined for $(t_1,t_2)\in\cal D$, are all continuous $(\F_{\ulin t})_{\ulin t\in\cal Q}$-adapted HC processes.

Let $B_1(t)$ and $B_2(t)$ be two independent standard Brownian motions. Suppose $\ha w_j(t)=w_j+\sqrt\kappa B_j(t)$, $0\le t<\infty$. 
Fix $j\ne k\in\{1,2\}$. Let $\F^{(k,\infty)}_{t_j}$ denote the $\sigma$-algebra $\F^j_{t_j}\vee \F^k_{\infty}$.  Then we get a filtration $(\F^{(k,\infty)}_{t_j})_{t_j\ge 0}$. Since $(\ha w_j)$ is independent of $(\F^k_{\infty})$, it is a rescaled $(\F^{(k, \infty)}_{t_j})_{t_j\ge 0}$-Brownian motion started from $w_j$.
Fix an $(\F^k_{\infty})$-measurable finite time $\tau_k$. From now on, we will repeatedly use It\^o's formula, where the variable $t_k$ is fixed to be $\tau_k$, the variable $t_j$  ranges in $[0,T^{\cal D}_j( \tau_k))$, and all SDE are $(\F^{(k,\infty)}_{t_j})_{t_j\ge 0}$-adapted. Recall that $X^{(k,\tau_k)}$ is the function obtained from a two-variable function $X$ by fixing the $k$-th variable to be $\tau_k$.  Using (\ref{Wj},\ref{-3}), we get
$$d W^{(k,\tau_k)}_j(t_j)=W^{(k,\tau_k)}_{j,1}(t_j)d \ha w_j(t_j)+\Big(\frac{\kappa 	 }{2}-3\Big)W_{j,2}^{(k,\tau_k)} d t_j.$$

To make the symbols less heavy, we will omit the superscripts $(k,\tau_k)$ and the variables $(t_j)$, and use the symbols $\pa_j$, $\pa\ha w_j$ and $\pa t_j$ to emphasize the role of $t_j$. The above SDE then becomes
$$\pa_j W_j=W_{j,1} \pa\ha w_j+\Big(\frac{\kappa 	}{2}-3\Big)W_{j,2}\pa t_j.$$
Combining this with (\ref{dWkVs}), we get (for $s=1,2$)
\begin{align*}
\pa_j (W_j-W_k)&=W_{j,1}\pa\ha w_j+\Big(\frac{\kappa 	}{2}-3\Big)W_{j,2}\pa t_j
+W_{j,1}^2\cot_2(W_j-W_k)\pa t_j;\\
\pa_j(W_j-V_s)&=W_{j,1}\pa\ha  w_j
+\Big(\frac{\kappa 	}{2}-3\Big)W_{j,2}\pa t_j
+W_{j,1}^2\cot_2(W_j-V_s)\pa t_j ;\\
\pa_j(W_k-V_s)&=
-W_{j,1}^2\cot_2(W_j-W_k)\pa t_j
+W _{j,1} ^2\cot_2(W _j -V _s )\pa t_j ;\\
\pa_j (V _j -V _k )&=
-W _{j,1} ^2\cot_2(W _j-V _j )\pa t_j
+W _{j,1}^ 2\cot_2(W _j-V_k)\pa t_j.
\end{align*}
Then we have
\begin{align}
  \frac{\pa_j\sin_2(W _j-W _k)}{\sin_2(W _j-W_k)}=&\frac 12\cot_2(W _j-W _k)W_{j,1}\pa \ha w_j+\frac 12 W _{j,1}^2\cot_2^2(W _j-W _k)\pa t_j \nonumber\\
 & +\frac 12\cot_2(W _j-W _k)\Big(\frac{\kappa 	}{2}-3\Big)W_{j,2}\pa t_j -\frac \kappa 8W _{j,1}^2\pa t_j; \label{dsin(Wj-Wk)}\\
  \frac{\pa_j\sin_2(W _j-V _s)}{\sin_2(W _j-V _s)}=&\frac 12\cot_2(W _j-V _s)W _{j,1}\pa \ha w_j+\frac 12 W _{j,1}^2\cot_2^2(W _j-V _s)\pa t_j\nonumber\\
 & +\frac 12\cot_2(W _j-V _s)\Big(\frac{\kappa 	}{2}-3\Big)W _{j,2}\pa t_j-\frac \kappa 8W _{j,1}^2\pa t_j;\label{dsin(Wj-Vs)}\\
   \frac{\pa_j \sin_2(W_k-V_s)}{\sin_2(W_k-V_s)}=&-\frac 12W_{j,1}^2[1+\cot_2(W_j-W_k)\cot_2(W_j-V_s)]\pa t_j;\label{dsin(Wk-Vs)}\\
    \frac{\pa_j \sin_2(V_j-V_k)}{\sin_2(V_j-V_k)}=&-\frac 12W_{j,1}^2[1+\cot_2(W_j-V_j)\cot_2(W_j-V_k)]\pa t_j.\label{dsin(V1-V2)}
\end{align}
Using (\ref{1/2-4/3}), we get
$$\frac{\pa_j W_{j,1}}{W_{j,1}}=\frac {W_{j,2}}{W_{j,1}}\pa \ha w_j+\frac 12\Big(\frac{W_{j,2}}{W_{j,1}}\Big)^2\pa t_j+\Big(\frac \kappa 2-\frac 43\Big)\frac{W_{j,3}}{W_{j,1}}\pa t_j-\frac 16 (W_{j,1}^2-1)\pa t_j.$$
Recall  the $W_{j,S}$ defined by (\ref{WkS}) and the $\bb,\cc$ defined by (\ref{bc}). The above SDE implies that
\BGE \frac{\pa_j W_{j,1}^{\bb}}{W_{j,1}^{\bb}}=\bb\frac{W_{j,2}}{W_{j,1}}\pa\ha w_j+\frac{\cc}6W_{j,S}\pa t_j-\frac {\bb}6 (W_{j,1}^2-1)\pa t_j. \label{paWj1b'}\EDE

Define a positive continuous function $M_{iB\to c4}$ on $\cal D$ by
$$M_{iB\to c4}:=e^{\frac {60}{8\kappa}\mA+\frac{\bb}{6}(\mA-t_1-t_2)}[W_{1,1}W_{2,1}V_{1,1}V_{2,1}]^{\bb} \Big [\prod_{j=1}^2\sin_2(W_j-V_j) \prod_{X,Y\in\{W,V\}} \sin_2(X_1-Y_2)\Big ]^{\frac 2\kappa}$$
\BGE \times \exp\Big(-\frac{\cc}6\int_0^{t_1}\!\int_0^{t_2} W_{1,1}^2W_{2,1}^2 \cot_2'''(W_1-W_2)ds_1ds_2\Big).\label{M-i-4'}\EDE
Combing (\ref{pamA},\ref{dWj1Vs1},\ref{pajWKS},\ref{dsin(Wj-Wk)}-\ref{paWj1b'}), we get
\BGE \frac{\pa _j M_{iB\to c4}}{M_{iB\to c4}}=\bb\frac{W_{j,2}}{W_{j,1}}\pa\ha w_j+  \frac 1\kappa \sum_{X\in\{W_k,V_1,V_2\}} \cot_2(W _j-X)W_{j,1}\pa \ha w_j.\label{pajM4}\EDE
This means that $M_{iB\to c4}^{(k,\tau_k)}$ is an $(\F^j_{t_j}\vee \F^k_\infty)_{t_j\ge 0}$-local martingale up to   $T^{\cal D}_j( \tau_k)$.


Let $F(x)$ be the hypergeometric function $ \,_2F_1\Big(\frac 4\kappa,1-\frac 4\kappa;\frac 8\kappa;x\Big)$ as before. Recall that $G(x)=\kappa x\frac{F'(x)}{F(x)}$, $\til F(x)  =x^{\frac 2\kappa} F(x)$, and $\til G(x)=\kappa x\frac{\til F'(x)}{\til F(x)}=G(x)+2$. From (\ref{ODE-hyper}) we get
$$\frac \kappa8 x^2\frac{\til F''(x)}{\til F(x)}=\Big[\Big(\frac 14-\frac 1\kappa \Big)\frac x{1-x}-\frac 1{2\kappa}\Big]\til G(x)+\frac 14\Big(\frac 6\kappa-1\Big) =0.$$ 
Recall that $W_1,V_1,W_2,V_2$ are real valued functions defined on $\cal D$ that satisfy $W_1>V_1>W_2>V_2>W_1-2\pi$. Define the  functions $R$ and $\Phi_j$ on $\cal D$ by
$$ R =\frac{\sin_2(W_1 -V_2 )\sin_2(V_1-W_2)}{\sin_2(W_1-W_2)\sin_2(V_1-V_2)}=-\frac{\sin_2(W_j-V_k)\sin_2(W_k-V_j)}{\sin_2(W_j-W_k)\sin_2(V_j-V_k)}\in(0,1).$$ 
$$\Phi_j=\cot_2(W_j-V_k)-\cot_2(W_j-W_k)=\frac{-\sin_2(W_k-V_k)}{\sin_2(W_j-V_k)\sin_2(W_j-W_k)}.$$ 
Note that $R$ equals the cross-ratio $[e^{iW_1},e^{iV_1};e^{iV_2},e^{iW_2}]$. Using  an identity of cross-ratio, we get
$$ 1-R=\frac{\sin_2(W_j-V_j)\sin_2(W_k-V_k)}{\sin_2(W_j-W_k)\sin_2(V_j-V_k)}.$$
Thus,
$$\frac{R\Phi_j}{1-R}=\frac{ \sin_2(W_k-V_j)}{\sin_2(W_j-V_j) \sin_2(W_j-W_k)}=\cot_2(W_j-W_k)-\cot_2(W_j-V_j).$$
Using (\ref{dsin(Wj-Wk)}-\ref{dsin(V1-V2)}), we get
\begin{align}
  \frac{\pa_j R}{R}=&\frac 12W _{j,1} \Phi_j\pa\ha w_j+\frac 12 [\cot_2(W _j-W _k)+\cot_2(W _j-V _k)] W _{j,1} ^2\Phi_j\pa t_j+\frac 12\Big(\frac{\kappa 	 }{2}-3\Big)W _{j,2}\Phi_j\pa t_j\nonumber \\
  &+\frac 12\cot_2 (W_j-V_j) W_{j,1} ^2 \Phi_j\pa t_j-\frac \kappa 4\cot_2(W _j-W _k)W _{j,1} ^2\Phi_j\pa t_j.\label{pajR/R}
\end{align}
Combining the above formulas in this paragraph and using a tedious but straightforward computation, we get
\begin{align}
  \frac{\pa_j \til F(R)}{\til F(R)}=& \frac 1{2\kappa}\til G(R)W _{j,1} \Phi_j\pa \ha w_j+\frac 1{2\kappa}\Big(\frac{\kappa 	 }{2}-3\Big)\til G(R)W _{j,2}\Phi_j\pa t_j \nonumber\\
  &+\frac 14\Big(\frac 6\kappa-1\Big)\cot_2(W_j-V_j)\til G(R)W _{j,1}^2  \Phi_j \pa t_j+\frac 14\Big(\frac 6\kappa-1\Big)  W _{j,1}^2  \Phi_j^2 \pa t_j.\label{paF/F}
\end{align}

Define another positive continuous function $M_{iB\to ch}$ on $\cal D$ by
$$M_{iB\to ch}:=e^{\frac{(\kappa-6)(\kappa-2)}{8\kappa} \mA+\frac{\bb}{6} (\mA-t_1-t_2)}\til F(R)  [W_{1,1}W_{2,1}V_{1,1}V_{2,1}]^{\bb}[\prod_{j=1}^2\sin_2(W_j-V_j) ]^{-2\bb} $$
\BGE \times \exp\Big(-\frac{\cc}6\int_0^{t_1}\!\int_0^{t_2} W_{1,1}^2W_{2,1}^2 \cot_2'''(W_1-W_2)ds_1ds_2\Big).\label{M-i-2}\EDE
Combining (\ref{pamA},\ref{dWj1Vs1},\ref{pajWKS},\ref{dsin(Wj-Vs)},\ref{dsin(Wk-Vs)},\ref{paWj1b'},\ref{paF/F}), we get
\BGE \frac{\pa _j M_{iB\to ch}}{M_{iB\to ch}}=\bb\frac{W_{j,2}}{W_{j,1}}\pa\ha w_j+\frac 1{2\kappa}\til G(R)W _{j,1} \Phi_j\pa\ha w_j-\bb\cot_2(W _j-V _j)W _{j,1}\pa\ha w_j.\label{pajM2}\EDE
This means that $M_{iB\to ch}^{(k,\tau_k)}$ is an $(\F^j_{t_j}\vee \F^k_\infty)_{t_j\ge 0}$-local martingale up to   $T^{\cal D}_j( \tau_k)$.

\subsection{Localization and Radon-Nikodym derivatives}\label{localization}
For $j=1,2$, let $\Xi_j$ denote the space of simple crosscuts of $\D$ that separate $e^{iw_j} $ from $e^{iv_1} $, $e^{iv_2} $, $e^{iw_{3-j}} $, and $0$. For $j=1,2$ and $\xi_j\in\Xi_j$, let $\tau^j_{\xi_j}$ be the first time that $\eta_j$ hits the closure of $\xi_j$. If such time does not exist, then $\tau^j_{\xi_j}$ is defined to be the lifetime of $\eta_j$. We see that $\tau^j_{\xi_j}$ is bounded above by the $\D$-capacity of the $\D$-hull generated by $\xi_j$, and so is finite.

Let $\Xi=\{(\xi_1,\xi_2)\in\Xi_1\times\Xi_2,\dist(\xi_1,\xi_2)>0\}$. For $\ulin\xi=(\xi_1,\xi_2)\in\Xi$, let $\tau_{\ulin\xi}=(\tau^1_{\xi_1},\tau^2_{\xi_2})$.
We may choose a countable set $\Xi^*\subset \Xi$ such that for every $\ulin\xi=(\xi_1,\xi_2)\in\Xi$ there is $(\xi_1^*,\xi_2^*)\in\Xi^*$ such that $\xi_j$ is enclosed by $\xi_j^*$, $j=1,2$. 

\begin{Lemma}
	For any $\ulin \xi\in\Xi$, $|\log(M_{iB\to c4})|$ and $|\log(M_{iB\to ch})|$ are uniformly bounded on $[\ulin 0,\tau_{\ulin\xi}]$ by constants depending only on $\kappa,\ulin\xi$. \label{uniform}
\end{Lemma}
\begin{proof}
	Fix $\ulin\xi=(\xi_1,\xi_2)\in\Xi$. Throughout this proof, a constant depends only on $\kappa,\ulin\xi$; by saying that a function is uniformly bounded on $[\ulin 0,\tau_{\ulin\xi}]$, we mean that it is bounded by a constant on $[\ulin 0,\tau_{\ulin\xi}]$. It suffices to show that $\mA$, $|\log(W_{j,1})|$, $|\log(V_{j,1})|$, $|\log\sin_2(X_1-Y_2)|$, $X,Y\in\{W,V\}$, $|\log\sin_2(W_1-V_1)|$, $|\log\sin_2(W_2-V_2)|$, $|\log(\til F(R))|$, and $|\int_0^{t_1}\!\int_0^{t_2} W_{1,1}^2W_{2,1}^2 \cot_2'''(W_1-W_2)ds_1ds_2|$
	are all uniformly bounded on $[\ulin 0,\tau_{\ulin\xi}]$.
	
	Let   $K_{\ulin \xi} $ be the $\D$-hull generated by $\xi_1\cup\xi_2$. Then $0\le t_1,t_2\le \mA$ are uniformly bounded by the constant $\dcap(K_{\ulin\xi})$ on $[\ulin 0,\tau_{\ulin\xi}]$.
	Note that $\TT\sem \lin{K_{\ulin\xi}}$ is a disjoint union of two arcs,  each of which contains one of $e^{iv_s }$, $s=1,2$. Denote the arcs $I_1$ and $I_2$ such that $e^{iv_s }\in I_s$, $j=1,2$. Each $I_s$ is divided by $e^{iv_s }$ into two open subarcs, which are denoted by $I_{s,1}$ and $I_{s,2}$ such that  $I_{s,j}$ shares one endpoint with $\xi_j$, $j=1,2$. Let the positive constant $c_{s,j}$ be the harmonic measure in $\D\sem K_{\ulin\xi}$ viewed from $0$ of the arc $I_{s,j}$.
	For any $\ulin t=(t_1,t_2)\in [\ulin 0,\tau_{\ulin\xi}]$, the harmonic measure in $\D\sem K_{(t_1,t_2)}$ viewed from $0$ of  the counterclockwise oriented arc from $e^{iv_1 }$ to the clockwise most point of $\eta_1([0,t])\cap\TT$ is bounded from below by $c_{1,1}$. Thus, $W_1-V_1\ge c_{1,1}*2\pi$ on $[\ulin 0,\tau_{\ulin\xi}]$. Similarly, $V_1 -W_2 \ge c_{1,2}*2\pi$, $W_2-V_2\ge c_{2,2}*2\pi$, and $V_2+2\pi-W_1\ge c_{2,1}*2\pi$ on $[\ulin 0,\tau_{\ulin\xi}]$. Let $S=\{W_1-V_1,W_2-V_2,W_1-V_2,W_1-W_2,V_1-W_2,V_1-V_2\}$. Then we see that $\sin_2(Z)$, $Z\in S$, are all bounded below by positive constants on $[\ulin 0,\tau_{\ulin\xi}]$.
	So we get the uniform boundedness of $|\log\sin_2(Z)|$, $|\cot_2(Z)|$, $|\cot_2'(Z)|$, and $|\cot_2'''(Z)|$ on $[\ulin 0,\tau_{\ulin\xi}]$. Since $0<W_{j,1} \le 1$ and $t_1,t_2$ are uniformly bounded, we get the uniform boundedness of $|\int_0^{t_1}\!\int_0^{t_2} W_{1,1}^2W_{2,1}^2 \cot_2'''(W_1-W_2)ds_1ds_2 |$ on $[\ulin 0,\tau_{\ulin\xi}]$.
	From (\ref{dWj1Vs1}) and that $W_{k,1}|_{t_j=0}\equiv 1$ and $V_{j,1}(0,0)=1$ we conclude that $\log(W_{j,1})$ and $\log(V_{j,1})$, $j=1,2$, are uniformly bounded on $[\ulin 0,\tau_{\ulin\xi}]$. From the definition of $R$ we know that $\log(R)$ is uniformly bounded on $[\ulin 0,\tau_{\ulin\xi}]$. Since $\til F(R)=R^{2/\kappa}F(R)$, and $F$ is positive and continuous on $[0,1]$, we see that $\log(\til F(R))$ is also uniformly bounded on $[\ulin 0,\tau_{\ulin\xi}]$.
\end{proof}


\begin{Corollary}
	For any $s\in\{4,h\}$ and $\ulin \xi\in\Xi$, $(M_{iB\to cs}(\ulin t\wedge \tau_{\ulin\xi}))_{\ulin t\in\cal Q}$ is an $(\F_{ \ulin t})$-$M_{iB\to cs}(\tau_{\ulin\xi})$-Doob martingale w.r.t.\ $\PP_{iB}$. \label{Doob}
\end{Corollary}
\begin{proof}
	Let $s\in\{4,h\}$ and $\ulin \xi=(\xi_1,\xi_2)\in\Xi$. We need to show that, for any $\ulin t=(t_1,t_2)\in\cal Q$, \BGE \EE_{iB}[M_{iB\to cs}(\tau_{\ulin\xi})|\F_{ \ulin t}]=M_{iB\to cs}(\ulin t\wedge \tau_{\ulin\xi}).\label{Doob-margingale0}\EDE
	From (\ref{pajM4},\ref{pajM2}) we know that $M_{iB\to cs}(\tau^1_{\xi_1},t_2)$, $0\le t_2<T^{\cal D}_2(\tau^1_{\xi_1})$, is an $(\F^{(1,\infty)}_{t_2})_{t_2\ge 0}$-local martingale. By the previous lemma, $M_{iB\to cs}(\tau^1_{\xi_1},\cdot)$ is uniformly bounded on $[0,\tau^2_{\xi_2}]$. From the assumption on $(\xi_1,\xi_2)$, we see that $\tau^2_{\xi_2}<T^{\cal D}_2(\tau^1_{\xi_1})$. So $M_{iB\to cs}(\tau^1_{\xi_1},\cdot\wedge \tau^2_{\xi_2})$ is an $(\F^{(1,\infty)}_{t_2})_{t_2\ge 0}$-$M_{iB\to cs}(\tau^1_{\xi_1}, \tau^2_{\xi_2})$-Doob-martingale. This means that
	\BGE \EE_{iB}[M_{iB\to cs}(\tau^1_{\xi_1}, \tau^2_{\xi_2})|\F^1_\infty\vee \F^2_{t_2}]=M_{iB\to cs}(\tau^1_{\xi_1},t_2\wedge \tau^2_{\xi_2}).\label{Doob-margingale1}\EDE
A similar argument using $M_{iB\to cs}(\cdot,t_2\wedge \tau^2_{\xi_2})$ in place of $M_{iB\to cs}(\tau^1_{\xi_1},\cdot)$ implies that
\BGE\EE_{iB}[M_{iB\to cs}(\tau^1_{\xi_1},t_2\wedge \tau^2_{\xi_2})|\F^1_{t_1}\vee \F^2_{\infty}]=M_{iB\to cs}(t_1\wedge \tau^1_{\xi_1},t_2\wedge \tau^2_{\xi_2}). \label{Doob-margingale2}\EDE
Since
$$M_{iB\to cs}(t_1\wedge \tau^1_{\xi_1},t_2\wedge \tau^2_{\xi_2})\in \F_{(t_1\wedge \tau^1_{\xi_1},t_2\wedge \tau^2_{\xi_2})}\subset \F_{(t_1,t_2)}= \F^1_{t_1}\vee \F^2_{t_2}\subset \F^1_{t_1}\vee \F^2_{\infty},$$
(\ref{Doob-margingale2}) implies that
	\BGE \EE_{iB}[M_{iB\to cs}(\tau^1_{\xi_1},t_2\wedge \tau^2_{\xi_2})|\F^1_{t_1}\vee \F^2_{t_2}]=M_{iB\to cs}(t_1\wedge \tau^1_{\xi_1},t_2\wedge \tau^2_{\xi_2}).\label{Doob-margingale3}\EDE
 Combining (\ref{Doob-margingale1},\ref{Doob-margingale3}) and using $\F^1_{t_1}\vee \F^2_{t_2}\subset \F^1_\infty\vee \F^2_{t_2}$, we get (\ref{Doob-margingale0}).
\end{proof}

The above corollary implies in particular that for any $s\in\{4,h\}$ and $\ulin\xi\in\Xi$,   we may define a probability measure $\PP_{cs}^{\ulin \xi}$ by $\frac{d \PP_{cs}^{\ulin \xi}}{d\PP_{iB}}=\frac{M_{iB\to cs}(\tau_{\ulin\xi})}{M_{iB\to cs}(\ulin 0)}$. Suppose $(\ha w_1,\ha w_2)$ follows the law $\PP_{cs}^{\ulin \xi}$. We now describe the behavior of the radial Loewner curves $\eta_1$ and $\eta_2$ driven by $\ha w_1$ and $\ha w_2$, respectively.
Fix $j\ne k\in\{1,2\}$. Let $\tau_k$ be an $(\F^k_{t_k})$-stopping time such that $\tau_k\le \tau^k_{\xi_k}$. From Lemma \ref{OST} and Corollary \ref{Doob}, for any $t_j\ge 0$,
$$\frac{d \PP_{cs}^{\ulin \xi}|\F^j_{t_j}\vee \F^k_{\tau_k}}{d\PP_{iB}|\F^j_{t_j}\vee \F^k_{\tau_k}}=\frac{M_{iB\to cs}^{(k,\tau_k)}(t_j\wedge \tau^j_{\xi_j})} {M_{iB\to cs}^{(k,\tau_k)}(0)}.$$
From Girsanov Theorem and (\ref{pajM4},\ref{pajM2}), we see that, under $\PP_{c4}^{\ulin \xi}$ and $\PP_{ch}^{\ulin \xi}$,  $\ha w_j$ respectively satisfies  the following two SDEs up to $\tau^j_{\xi_j}$:
\begin{align*}
  \pa\ha w_j=&\sqrt\kappa \pa B^4_{j,\tau_k}+\kappa \bb \frac{W_{j,2}^{(k,\tau_k)}}{W_{j,1}^{(k,\tau_k)}}\pa t_j+   \sum_{X\in\{W_k,V_1,V_2\}} \cot_2(W _j^{(k,\tau_k)}-X^{(k,\tau_k)})W^{(k,\tau_k)}_{j,1}\pa t_j,\\
  \pa\ha w_j=&\sqrt\kappa \pa B^h_{j,\tau_k}+\kappa \bb \frac{W_{j,2}^{(k,\tau_k)}}{W_{j,1}^{(k,\tau_k)}}\pa t_j+ \frac 1{2}\til G(R^{(k,\tau_k)})W _{j,1}^{(k,\tau_k)} \Phi_j^{(k,\tau_k)}\pa t_j \\ &-\kappa \bb\cot_2(W _j^{(k,\tau_k)}-V _j^{(k,\tau_k)})W _{j,1}^{(k,\tau_k)}\pa t_j,
\end{align*}
where $B^s_{j,\tau_k}(t_j)$  is a standard $(\F^j_{t_j}\vee \F^k_{\tau_k})_{t_j\ge 0}$-Brownian motion under $\PP_{cs}^{\ulin \xi}$, $s\in\{4,h\}$. Using (\ref{Wj},\ref{-3}) we get the SDE satisfied by $W_j^{(k,\tau_k)}$ under $\PP_{c4}^{\ulin \xi}$ and $\PP_{ch}^{\ulin \xi}$, respectively, up to $\tau^j_{\xi_j}$:
\begin{align*}
  \pa W_j^{(k,\tau_k)}=&\sqrt\kappa W_{j,1}^{(k,\tau_k)} \pa B^4_{j,\tau_k} +   \sum_{X\in\{W_k,V_1,V_2\}} \cot_2(W _j^{(k,\tau_k)}-X^{(k,\tau_k)})(W^{(k,\tau_k)}_{j,1})^2\pa t_j,\\
  \pa W_j^{(k,\tau_k)}=&\sqrt\kappa W_{j,1}^{(k,\tau_k)} \pa B^h_{j,\tau_k} + \frac 1{2}\til G(R^{(k,\tau_k)})\Phi_j^{(k,\tau_k)}(W _{j,1}^{(k,\tau_k)})^2 \pa t_j \\ &-\kappa \bb\cot_2(W _j^{(k,\tau_k)}-V _j^{(k,\tau_k)})(W _{j,1}^{(k,\tau_k)})^2\pa t_j.
\end{align*}
Recall the ODE (\ref{dWkVs}) satisfied by $W_k$ and $V_s$, $s=1,2$.
This implies that, under $\PP_{c4}^{\ulin \xi}$, conditionally on $\F^k_{\tau_k}$, $\eta_{j,\tau_k}(t_j)=g_k(\tau_k,\eta_j(t_j))$ is a radial SLE$_\kappa(2,2,2)$ curve with speed $(W^{(k,\tau_k)}_{j,1})^2$ started from $e^i(W_j^{(k,\tau_k)}(0))=g_k(\tau_k,e^{i\ha w_j(0)})$ with force points
 $e^i(W_k^{(k,\tau_k)}(0))=e^{i\ha w_k(0)}=g_k(\tau_k,\eta_k(\tau_k))$, $e^i( V_j^{(k,\tau_k)}(0))=g_k(\tau_k,e^{iv_j})$ and $e^i( V_k^{(k,\tau_k)}(0))=g_k(\tau_k,e^{iv_k})$, up to $\tau^j_{\xi_j}$; and under $\PP_{ch}^{\ulin \xi}$,
 conditionally on $\F^k_{\tau_k}$, $g_k(\tau_k,\eta_j(t_j))$ is a radial hSLE$_\kappa$ curve in $\D$ from $g_k(\tau_k,e^{i\ha w_j(0)})$ to $g_k(\tau_k,e^{iv_j})$ with force points $g_k(\tau_k,e^{iv_k})$ and $g_k(\tau_k,\eta_k(\tau_k))$, up to $\tau^j_{\xi_j}$, viewed from $0$. In particular, taking $\tau_k=0$, we see that the $j$-th marginal measure of $\PP_{cs}^{\ulin \xi}$ restricted to $\F^j_{\tau^j_{\xi_j}}$ agrees with $\PP^j_s$ restricted to $\F^j_{\tau^j_{\xi_j}}$. This means that the radial Loewner curves driven by $\ha w_1$ and $\ha w_2$, which jointly follow the law $\PP_{c4}^{\ulin \xi}$ (resp.\ $\PP_{ch}^{\ulin \xi}$), respectively stopped at $\tau^1_{\xi_1}$ and $\tau^2_{\xi_2}$, are two radial SLE$_\kappa(2,2,2)$ (resp.\ radial hSLE$_\kappa$) curves that locally commute with each other  in the sense of \cite{Julien}. Recall that $\PP_{2}$ is the joint law of the radial Loewner driving functions for a $2$-SLE$_\kappa$ in $\D$ with link pattern $(e^{iw_1}\to e^{iv_1};e^{iw_2}\to e^{iv_2})$ up to certain separation times.
 Because of the commutation relation between the two curves in a $2$-SLE$_\kappa$, we find that $\PP_{ch}^{\ulin \xi}|\F_{\ulin\xi}=\PP_{2}|\F_{\ulin\xi}$. 

Using the stochastic coupling technique developed and used in \cite{reversibility,duality} we may construct a probability measure  $\PP_{c4}$  on  $\Sigma\times\Sigma$ such that for any $\ulin\xi\in\Xi$, $\PP_{c4}|\F_{\tau_{\ulin\xi}}=\PP_{c4}^{\ulin \xi}|\F_{\tau_{\ulin\xi}}$.
Here is a brief review of the stochastic coupling technique for the setup here.  From (\ref{pajM4},\ref{pajM2}) and Girsanov Theorem we know that $M^{(k,0)}_{iB\to c4}(\tau^j_{\xi_j})/M_{iB\to c4}(0,0)$ is the Radon-Nikodym derivative of $\PP^j_4|\F^j_{\tau^j_{\xi_j}}$ against $\PP^j_B|\F^j_{\tau^j_{\xi_j}}$.
Define $M_{i4\to c4}$ on $\cal D$ by
$$M_{i4\to c4}(t_1,t_2)=\frac{M_{iB\to c4}(t_1,t_2)M_{iB\to c4}(0,0)}{M_{iB\to c4}(t_1,0)M_{iB\to c4}(0,t_2)}.$$
Then $M_{i4\to c4}(t_1,t_2)=1$ if $t_1\cdot t_2=0$; and under the probability measure $\PP_{i4}=\PP^1_4\times\PP^2_4$, for any finite $(\F^k_{t_k})$-stopping time $\tau_k$, $M^{(k,\tau_k)}_{i4\to c4}(t_j)$ is a local martingale. From Lemma \ref{uniform} we know that, for any $\ulin\xi \in\Xi$, $|\log M_{i4\to c4}|$ is bounded on $[\ulin 0,\tau_{\ulin\xi}]$.
Let $(\ulin\xi^k)_{k\in\N}$ be an enumeration of $\Xi^*$. From \cite[Theorem 6.1]{reversibility} we know that, for any $n\in\N$, there is a uniformly bounded $(\F_{\ulin t})_{\ulin t\in\cal Q}$-Doob-martingale $M^{(n)}_{i4\to c4}$ defined on $[0,\infty]\times[0,\infty]$ such that $M_{i4\to c4}^{(n)}(t_1,t_2)=1$ if $t_1\cdot t_2=0$, and for any $1\le k\le n$, $M_{i4\to c4}^{(n)}$ agrees with $M_{i4\to c4}$ on $[\ulin 0, \tau_{\ulin\xi^k}]$. We may then define a sequence of probability measures $\PP^{(n)}_{c4}$, $n\in\N$, by $d\PP^{(n)}_{c4}= M_{i4\to c4}^{(n)}(\infty,\infty) d\PP_{i4}$. Then every $\PP^{(n)}_{c4}$ is a coupling of $\PP^1_4$ and $\PP^2_4$, and for $1\le k\le n$,
$\frac{d \PP_{c4}^{(n)}|\F_{\tau_{\ulin \xi^k}}}{d\PP_{iB}|\F_{\tau_{\ulin \xi^k}}}
=\frac{M_{iB\to c4}(\tau_{\ulin\xi^k})}{M_{iB\to c4}(\ulin 0)}$.
By a tightness argument, $(\PP_{c4}^{(n)})$ contains a weakly convergent subsequence. Let $\PP_{c4}$ denote any subsequential limit. Then for any $\ulin\xi\in\Xi^*$,
$\frac{d \PP_{c4} |\F_{\tau_{\ulin \xi }}}{d\PP_{iB}|\F_{\tau_{\ulin \xi }}}=\frac{M_{iB\to c4}(\tau_{\ulin\xi })}{M_{iB\to c4}(\ulin 0)}$.
Since for every $\ulin\xi\in\Xi$, there is $\ulin\xi^*\in\Xi^*$ such that $\tau_{\ulin \xi }\le \tau_{\ulin \xi ^*}$, by the martingale property of $M_{iB\to c4}(\cdot\wedge \tau_{\ulin \xi^* })$, we get $\PP_{c4} |\F_{\tau_{\ulin \xi }}=\PP_{c4}^{\ulin\xi} |\F_{\tau_{\ulin \xi }}$, as desired.

We may use the same idea to construct $\PP_{ch}$. It satisfies  $\PP_{ch}|\F_{\tau_{\ulin\xi}}=\PP_{ch}^{\ulin \xi}|\F_{\tau_{\ulin\xi}}=\PP_{2}|\F_{\tau_{\ulin\xi}}$ for any $\xi\in\Xi$. At this moment we do not have a proof showing that $\PP_{ch}=\PP_{2}$, and we do not need this result.  We now have the following lemma.

\begin{Lemma}
	For  any finite $(\F_{ \ulin t})_{\ulin t\in\cal Q}$-stopping time $\ulin T$,
	$$\frac{d \PP_{c4} |\F_{\ulin T}\cap\{\ulin T\in\cal D\}}{d\PP_{iB}|\F_{\ulin T}\cap\{\ulin T\in\cal D\}}=\frac{M_{iB\to c4}(\ulin T)}{M_{iB\to c4}(\ulin 0)},\quad \frac{d \PP_{2} |\F_{\ulin T}\cap\{\ulin T\in\cal D\}}{d\PP_{iB}|\F_{\ulin T}\cap\{\ulin T\in\cal D\}}=\frac{M_{iB\to ch}(\ulin T)}{M_{iB\to ch}(\ulin 0)}.$$ \label{Lemma-RN-T}
\end{Lemma}
\begin{proof}
	We first work on $\PP_{c4}$. We have
	$\{\ulin T\in{\cal D}\}=\bigcup_{\ulin \xi\in\Xi^*}\{\ulin T\le  \tau_{\ulin\xi}\}$.
	Since by Lemma \ref{separable-lemma}, $\{\ulin T\le  \tau_{\ulin\xi}\}\in\F_{ \ulin T}$, it suffices to show that, for any $\ulin \xi\in\Xi^*$,
	\BGE \frac{d \PP_{c4} |\F_{\ulin T}\cap\{\ulin T\le  \tau_{\ulin\xi}\}}{d\PP_{iB}|\F_{\ulin T}\cap\{\ulin T\le  \tau_{\ulin\xi}\}}=\frac{M_{iB\to c 4}(\ulin T)}{M_{iB\to c4}(\ulin 0)}.\label{RN}\EDE
	By Lemma \ref{OST} and Corollary \ref{Doob}, we see that
	$\EE_{iB}[M_{iB\to c4}(\tau_{\ulin\xi})|\F_{ \ulin T}]=M_{iB\to c4}(\ulin T\wedge \tau_{\ulin\xi})$.
	Let $A\in \F_{ \ulin T}$ with $A\subset \{\ulin T\le \tau_{\ulin\xi}\}$. Then we have
	$\EE_{iB}[{\bf 1}_A M_{iB\to c4}(\tau_{\ulin\xi})]=\EE_{iB}[{\bf 1}_A M_{iB\to c4}(\ulin T)]$.
	Since $\frac{d \PP_{c4} |\F_{\tau_{\ulin\xi}}}{d\PP_{iB}|\F_{\tau_{\ulin\xi}}}=\frac{M_{iB\to c4}(\tau_{\ulin\xi})}{M_{iB\to c4}(\ulin 0)}$, and $A\subset \F_{\tau_{\ulin\xi}}$ by Lemma \ref{T<S}, we get
	$$M_{iB\to c4}(\ulin 0)\PP_{c4}[A]=\EE_{iB}[{\bf 1}_A M_{iB\to c4}(\tau_{\ulin\xi})]=\EE_{iB}[{\bf 1}_A M_{iB\to c4}(\ulin T)].$$
	Since this holds for any $A\in \F_{ \ulin T}$ with $A\subset \{\ulin T\le \tau_{\ulin\xi}\}$, we get (\ref{RN}) as desired.

A similar argument shows that (\ref{RN}) holds with $c4$ replaced by $ch$. Since $\F_{\ulin T}\cap\{\ulin T\le  \tau_{\ulin\xi}\}\subset \F_{\ulin\xi}$ and $\PP_{ch}$ agrees with $\PP_{2}$ on $\F_{\ulin\xi}$, we find that (\ref{RN}) holds with  $M_{iB\to c4}$ replaced by $M_{iB\to ch}$ and $\PP_{c4}$ replaced by $\PP_{2}$. So we obtain the second equality.
\end{proof}


We need the following lemma about the lifetime of a radial SLE$_\kappa(\ulin \rho)$ curve.

\begin{Lemma}
Let $\kappa>0$, $n\in\N$. Suppose $\ulin\rho=(\rho_1,\dots,\rho_n)\in\R^n$ satisfies $\rho_1,\rho_n\ge \frac\kappa 2-2$ and $\rho_k\ge 0$, $1\le k\le n$.  Let $e^{iw},e^{iv_1},\dots,e^{iv_n}$ be distinct points on $\TT$ such that $w>v_1>\cdots>v_n>w-2\pi$.  Let $\eta(t)$, $0\le t<T$, be a radial SLE$_\kappa(\ulin\rho)$ curve in $\D$ started from $e^{iw}$ aimed at $0$ with force points $e^{iv_1},\dots,e^{iv_n}$. Then a.s.\ $T=\infty$, $0$ is a subsequential limit of $\eta(t)$ as $t\to \infty$, and $\eta$ does not hit the arc $J:=\{e^{i\theta}:v_1\ge \theta\ge v_n\}$. \label{transience}
\end{Lemma}
\begin{proof}
   Let $\ha w(t)$ and $\ha v_j(t)$, $1\le j\le n$, $0\le t<T$, be the solutions of the system of SDE used to define this radial SLE$_\kappa(\ulin\rho)$ curve.
 For any $t\in[0,T)$, we have $\ha w(t)>\ha v_1(t)>\cdots\ha v_n(t)>\ha w(t)-2\pi$. If   $T<\infty$, then one of the following events $E^0_{n'},E^{2\pi}_{n'}$, $1\le n'\le n$, must happen:
 $$E^0_{n'}=\{\lim_{t\to T^-} \ha w(t)-\ha v_j(t)=0, 1\le j\le n'\}\cap \{\lim_{t\to T^-} \ha w(t)-\ha v_j(t)\in(0,2\pi),  n'+1\le j\le n\}, $$
 $$E^{2\pi}_k= \{\lim_{t\to T^-} \ha w(t)-\ha v_j(t)=2\pi, n'\le j\le n\}\cap \{\lim_{t\to T^-} \ha w(t)-\ha v_j(t)\in(0,2\pi),  1\le j\le n'-1\}.$$
 To prove that $\PP[T<\infty]=0$, it suffices to show that $\PP[E^0_{n'}]=\PP[E^{2\pi}_{n'}]=0$ for $1\le n'\le n$. By symmetry, we only need to consider $E^0_{n'}$, $1\le n'\le n$. If $\PP[E^0_{n'}]>0$, using Girsanov Theorem, we see that for a radial SLE$_\kappa(\rho_1,\dots,\rho_{n'})$ process in $\D$ from $e^{iw}$ to $0$ with force points $e^{iv_1},\dots,e^{iv_{n'}}$, there is a positive probability that the lifetime $T$ is finite and $\lim_{t\to T^-} \ha w(t)-\ha v_j(t)=0$, $1\le j\le n'$. For this new process, $X_{n'}(t):=\ha w(t)-\ha v_{n'}(t)$ satisfies the SDE:
 $$dX_{n'}(t)=\sqrt\kappa dB(t)+\sum_{j=1}^{n'}\frac{\rho_j}2 \cot_2(\ha w_1(t)-\ha v_j(t))dt+\cot_2(X_k(t))dt.$$
 Since $\cot_2(\ha w_1(t)-\ha v_j(t))> \cot_2(X(t))$ and $\rho_j\ge 0$ for $1\le j\le k-1$, the process $X_{n'}$ stochastically dominates the process $Y$, which satisfies the SDE: $dY(t)=\sqrt \kappa dB(t)+(1+\frac\sigma2)\cot_2(Y(t))dt$, where $\sigma=\sum_{j=1}^{n'} \rho_j\ge \rho_1\ge \frac\kappa 2-2$. It is easy to see that $\frac 12 Y_k(\frac 4\kappa t) $ is a radial Bessel process of dimension $\delta=1+\frac2\kappa(2+\sigma)  \ge 2$, which a.s.\ does not tend to $0$ at any finite time (cf.\ \cite[Appendix A]{Law-Bessel},\cite[Appendix B]{tip}). So the probability that the $X_{n'}(t)$ for the new process tends to $0$ at a finite time is also $0$, which implies that the probability of the $E^0_{n'}$ for the original process is $0$. Thus, a.s.\ $T=\infty$. By Koebe's $1/4$ Theorem, we see that $0$ is a subsequential limit of $\eta$ as $t\to \infty$. If $\eta$ hits the arc $J$, then when it happens, $\eta$ separates $0$ from either $e^{iv_1}$ or $e^{iv_n}$, and the process stops at this time. Since a.s.\ $T=\infty$, such hitting a.s.\ can not happen.
\end{proof}

Now we consider two radial Loewner curves $\eta_1$ and $\eta$, whose driving functions jointly follow $\PP_{c4}$.
From Lemma \ref{transience} (applied to $\kappa\in(0,8)$ and $\rho_1=\rho_2=\rho_3=2$) we know that the   the lifetimes of $\eta_1$ and $\eta_2$ are both a.s.\ $\infty$.
Fix $\tau_2<\infty$. Conditional on $\F^2_{\tau_2}$, $g_2(\tau_2,\eta_1(t_1))$, $t_1\ge 0$, is a radial SLE$_\kappa(2,2,2)$ curve in $\D$ started from $ g_2(\tau_2,e^{i w_1})$ with force points $ g_2(\tau_2,\eta_2(\tau_2))$, $ g_2(\tau_2,e^{iv_1})$ and  $g_2(\tau_2,e^{iv_2})$, up to the lifetime of $\eta_1$ or the first  time that $\eta_1$ hits $\eta_2[0,\tau_2]$. If $\eta_1$ hits $\eta_2[0,\tau_2]$, then it means that $g_2(\tau_2,\eta_1(t_1))$ hits the boundary arc of $\D$ with end points  $ g_2(\tau_2,e^{iv_1})$ and  $g_2(\tau_2,e^{iv_2})$ that contains $g_2(\tau_2,\eta_2(\tau_2))$, which is impossible by Lemma \ref{transience}. Thus, the whole $\eta_1$ does not intersect $\eta_2[0,\tau_2]$. From Lemma \ref{transience} we also know that $\eta_1$  a.s does not intersect the boundary arc of $\D$ with end points $e^{iv_1}$ and $e^{iv_2}$ that contains the initial point of $\eta_2$: $e^{iw_2}$. From the definition of $\cal D$, we have $\PP_{c4}$-a.s.\ $T^{\cal D}_1(\tau_2)=\infty$. Since this holds for any deterministic $\tau_2<\infty$, and the lifetime of $\eta_2$ is a.s.\ $\infty$, we get the following lemma.

\begin{Lemma}
	$\PP_{c4}$-a.s.\ ${\cal D}={\cal Q}=[0,\infty)^2$. \label{T=infty}
\end{Lemma}

Let $M_{ch\to c4}=\frac{M_{iB\to c4}}{M_{iB\to ch}}$ and $M_{c4\to ch}=M_{ch\to c4}^{-1}$. From Lemma \ref{Lemma-RN-T}   we see that, for any finite $(\F_{ \ulin t})_{\ulin t\in\cal Q}$-stopping time $\ulin T$,
\BGE \frac{d \PP_{2} |\F_{\ulin T}\cap\{\ulin T\in\cal D\}}{d\PP_{c4}|\F_{\ulin T}\cap\{\ulin T\in\cal D\}}=\frac{M_{c4\to ch}(\ulin T)}{M_{c4\to ch}(\ulin 0)}.\label{RN-4to2}\EDE
Let $G(w_1,v_1;w_2,v_2)$ be defined by
$$ G(w_1,v_1;w_2,v_2)= |\sin_2(w_1-v_1)\sin_2(w_2-v_2)|^{\frac8\kappa -1} |\sin_2(w_1-w_2)\sin_2(v_1-v_2)|^{\frac 4\kappa}\times $$ \BGE \times F\Big(\Big|\frac{\sin_2(w_1 -v_2 )\sin_2(v_1-w_2)}{\sin_2(w_1-w_2)\sin_2(v_1-v_2)}\Big|\Big)^{-1}.\label{G(WV)}\EDE
Then with $\alpha_0$ defined by (\ref{alpha0}), we have
\BGE M_{ch\to c4}=e^{\alpha_0\cdot\mA}G(W_1,V_1;W_2,V_2).\label{M2to4}\EDE
From Lemma \ref{T=infty} and (\ref{RN-4to2}) we see that for any finite $(\F_{ \ulin t})_{\ulin t\in\cal Q}$-stopping time $\ulin T$,
\BGE \EE_{2}[{\bf 1}_{\{\ulin T\in\cal D\}}e^{\alpha_0\cdot\mA }G(W_1,V_1;W_2,V_2)|_{\ulin t=\ulin T}] = G(w_1,v_1;w_2,v_2).\label{EchG}\EDE

\section{A Time Curve in the Time Region}\label{section-Time-curve}
In the previous section we have derived many random processes with two time parameters defined on the time region $\cal D$. The goal of this section is to find a parameterization $\ulin u$ which transfers every two-time-parameter process to a one-time-parameter process. The transform then allows the application of It\^o's calculus to those  one-time-parameter processes.

Throughout this section, we suppose $v_1 -v_2 =\pi$.  Let $\theta=V_1-V_2\in(0,2\pi)$. Then $\theta(0,0)=\pi$. We are going to get a continuous and strictly increasing curve $\ulin u:[0,T^u)\to\cal D$ with $\ulin u(0)=\ulin 0$ such that $\theta(\ulin u(t))=\pi$ and $\mA(\ulin u(t))=t$ for any $t\in[0,T^u)$, and the curve can not be further extended with this property. Note that
\BGE \pa_j \theta=W_{j,1}^2( \cot_2(V_1-W_j)-\cot_2(V_2-W_j))\pa t_j=\frac{-W_{j,1}^2\sin_2(\theta)}{\sin_2(W_j-V_1)\sin_2(W_j-V_2)}\pa t_j.\label{pa-j-theta'}\EDE
So $\pa_1\theta<0$ and $\pa_2 \theta>0$. Thus, $\theta(t,0)<\pi$ for $t>0$; and $\theta(0,t)>\pi$ for $t>0$. Let
$$S_1= \{t_1\ge 0: \exists t_2> 0\mbox{ such that }(t_1,t_2)\in{\cal D}\mbox{ and }\theta(t_1,t_2)>\pi\}.$$
Suppose $t_1\in S_1$, and $t_2>0$ is such that $(t_1,t_2)\in\cal D$ and $\theta(t_1,t_2)>\pi$. Then for any $t_1'\in[0,t_1)$, $(t_1',t_2)\in\cal D$ and $\theta(t_1',t_2)>\theta(t_1,t_2)>\pi$, which implies that $t_1'\in S_1$. On the other hand, since $\cal D$ is relatively open in $\R_+^2$, by the continuity of $\theta$, we can find $t_1''>t_1$ such that $(t_1'',t_2)\in\cal D$ and $\theta(t_1'',t_2)>\pi$, which implies that $t_1''\in S_1$. So $S_1=[0,T^u_1)$ for some $T^u_1\in(0,\infty]$. For every $t_1\ge T^u_1$ and any $t_2\ge 0$ such that $(t_1,t_2)\in\cal D$, we must have $\theta(t_1,t_2)<\pi$. For $t_1\in[0,T^u_1)$, applying the intermediate value theorem to $\theta(t_1,\cdot)$ and using the strict  monotonicity of $\theta$ in $t_2$, we conclude that there is a unique $t_2\ge 0$ such that $(t_1,t_2)\in\cal D$ and $\theta(t_1,t_2)=\pi$. Let $u_{1\to 2}$ denote the map $[0,T^u_1)\ni t_1\mapsto t_2$. Since $\theta$ is strictly decreasing in $t_1$ and strictly increasing in $t_2$, $u_{1\to 2}$ is strictly increasing. A symmetric argument shows that there exists $T^u_2\in(0,\infty]$ such that for any $t_2\ge T^u_2$ and any $t_1\ge 0$ such that $(t_1,t_2)\in\cal D$, we have $\theta(t_1,t_2)>\pi$; for any $t_2\in[0,T^u_2)$, there is a unique $t_1\ge 0$ such that $(t_1,t_2)\in\cal D$ and $\theta(t_1,t_2)=\pi$; and the map $u_{2\to 1}:[0,T^u_2)\ni t_2\mapsto t_1$ is strictly increasing. Thus, $u_{1\to 2}$ maps $[0,T^u_1)$ onto $[0,T^u_2)$, and $u_{2\to 1}$ is its inverse. Moreover, both $u_{1\to 2}$ and $u_{2\to 1}$ are continuous. Since $\mA$ is continuous and strictly increasing in both $t_1$ and $t_2$, we see that the map $[0,T^u_1)\ni t_1\mapsto \mA(t_1,u_{1\to 2}(t_1))$ is continuous and strictly increasing. Since $u_{1\to 2}(0)=0$ and $\mA(0,0)=0$, the range of $ \mA(t_1,u_{1\to 2}(t_1))$ is $[0,T^u)$ for some $T^u\in(0,\infty]$. Let $u_1$ denote the inverse of this map, and let $u_2=u_{1\to 2}\circ u_1$. Then for $j=1,2$, $u_j$ is a continuous and strictly increasing function that maps $[0,T^u)$ onto $[0,T^u_j)$; and $\ulin u:=(u_1,u_2):[0,T^u)\to \cal D$ is a strictly increasing curve that satisfies  $\theta(u_1(t),u_2(t))=\pi$ and $\mA(u_1(t),u_2(t))=t$ for any $0\le t<T^u$, and $\lim_{t\to T^-} \ulin u(t)=(T^u_1,T^u_2)$. We see that $(T^u_1,T^u_2)$ does not belong to $\cal D$ because if it does then $\theta(T^u_1,T^u_2)=\pi$,
which contradicts the statement that for every $t_1\ge T^u_1$ and any $t_2\ge 0$ such that $(t_1,t_2)\in\cal D$, we have $\theta(t_1,t_2)<\pi$.
Since $\mA$ is increasing in $t_1$ and $t_2$, we get $u_1(t)=\mA(u_1(t),0)\le \mA(u_1(t),u_2(t))=t$. Similarly, $u_2(t)\le t$. 

For any function $X$ defined on $\cal D$, we define $X^u(t)=X(\ulin u(t))$, $0\le t<T^u$. For example, if $X=\ha w_j$, $j=1,2$, then $\ha w^u_j(t)=\ha w_j(u_j(t))$. Let $Z_j=W_j-V_j>0$, $j=1,2$. Then $Z_2^u\in(0,\pi)$ because $Z_2^u<V_1^u-V_2^u=\pi$, and $Z_1^u\in (0,\pi)$ because $Z_1^u<V_2^u+2\pi-V_1^u=\pi$. We are going to derive SDEs for $Z_j^u$, $j=1,2$, in the following lemma.

\begin{Lemma}
  Under $\PP_{c4}$,  there are two independent standard Brownian motions $B^u_j(t)$, $j=1,2$,  such that for $j=1,2$, $Z_j^u$ satisfies the SDE \BGE dZ_j^u=\sqrt{\frac{\kappa \sin(Z_j^u)}{\sin(Z_1^u)+\sin(Z_2^u)}}dB_j^u+ \frac{4\cos(Z_j^u)}{\sin(Z_1^u)+\sin(Z_2^u)}\,dt,\quad 0\le t<\infty. \label{dZjTu}\EDE
\label{ZSDE}
\end{Lemma}

From (\ref{pa-j-theta'}) and that $\theta^u\equiv \pi$, we get
$$0=\frac{-2(W_{1,1}^u)^2 }{\sin(Z^u_1(t))}u_1'(t)+\frac{2(W_{2,1}^u)^2 }{\sin(Z^u_2(t))}u_2'(t).$$
From $\mA(u_1(t),u_2(t))=t$ and  (\ref{pamA}) we get
$$1=(W^u_{1,1}(t))^2 u_1'(t)+(W^u_{2,1}(t))^2 u_2'(t).$$
Combining, we get
\BGE (W^u_{j,1})^2 u_j'=\frac{\sin(Z^u_{j})}{\sin(Z^u_1)+\sin(Z^u_2)},\quad j=1,2.\label{uj''}\EDE

So far $u_1$ and $u_2$ are defined on $[0,T^u)$. If $T^u<\infty$, we extend $u_1$ and $u_2$ to $[0,\infty)$ such that for $t\ge T^u$, $u_j(t)=T^u_j$, $j=1,2$. From $u_j(t)\le t$ we get $T^u_j\le T^u<\infty$, $j=1,2$. Thus, the extended $u_1$ and $u_2$ are finite and continuous. Below is a lemma on the extended $\ulin u$.

\begin{Lemma}
  For any $t\in[0,\infty)$, $\ulin u(t)=(u_1(t),u_2(t))$ is an $(\F_{\ulin t})_{\ulin t\in \cal Q}$-stopping time. \label{uisst}
\end{Lemma}
\begin{proof}
   Fix $t\ge 0$ and $\ulin s=(s_1,s_2)\in\cal Q$. We need to show that $\{\ulin u(t)\le \ulin s\}\in \F_{\ulin s}$. For this purpose, we consider three events. Let $A_1$ denote the event that the curve $\ulin u\cap \cal D$ intersects $\{s_1\}\times [0,s_2) $; and let $A_2$ denote the event that the curve $\ulin u\cap \cal D$ intersects $  [0,s_1)\times \{s_2\}$. Then $A_1\cap A_2=\emptyset$,
   $$A_1=\bigcup_{t_2\in [0,s_2)\cap\Q} \{(s_1,t_2)\in{\cal D},\theta(s_1,t_2)>\pi \}\in \F_{ \ulin s},$$
   and similarly $A_2\in\F_{ \ulin s}$. Here we used the fact that $\cal D$ is an $(\F_{ \ulin t})_{\ulin t\in\cal Q}$-stopping region and $\theta$ is $(\F_{ \ulin t})_{\ulin t\in\cal Q}$-adapted. Let $A_0=(A_1\cup A_2)^c\in\F_{\ulin s}$.
We have
\begin{align*}
\{\ulin u(t)\le \ulin s\}\cap A_0&=A_0\cap(\{\ulin s\not\in{\cal D}\}\cup\{\ulin s\in {\cal D}, \theta(\ulin s)=\pi,\mA(\ulin s)\ge t\});\\
\{\ulin u(t)\le \ulin s\}\cap A_1&=\bigcap_{n\in\N} \bigcup_{r_2<t_2\in[0,s_2)\cap\Q} \{(s_1,t_2)\in{\cal D}, \theta(s_1,t_2)>\pi>\theta(s_1,r_2),\mA(s_1,r_2)>t-\frac 1n \};\\
\{\ulin u(t)\le \ulin s\}\cap A_2&=\bigcap_{n\in\N} \bigcup_{r_1<t_1\in[0,s_1)\cap\Q} \{(t_1,s_2)\in{\cal D}, \theta(t_1,s_2)<\pi<\theta(r_1,s_2),\mA(r_1,s_2)>t-\frac 1n\}.
\end{align*}
Since the events on the righthand side are all $\F_{\ulin s}$-measurable, so is $\{\ulin u(t)\le \ulin s\}$, as desired.
\end{proof}


We now get a new filtration $(\F^u_t:=\F_{\ulin u(t)})_{t\ge 0}$ by Lemma \ref{T<S} since $\ulin u$ is non-decreasing. For $\ulin \xi=(\xi_1,\xi_2)\in\Xi$,  let $\tau^u_{\ulin\xi}$ denote the first $t\ge 0$ such that $u_1(t)=\tau^1_{\xi_1}$ or $u_2(t)=\tau^2_{\xi_2}$, whichever comes first. 
Note that such time exists and is finite because $[\ulin 0,\tau_{\ulin \xi }]\subset\cal D$.

\begin{Lemma}
 For $\ulin\xi\in\Xi$, 	$\ulin u(\tau^u_{\ulin \xi})$ is an $(\F_{\ulin t})_{\ulin t\in\cal Q}$-stopping time, $\tau^u_{\ulin \xi}$ is an $(\F^u_t)_{t\ge 0}$-stopping time, and for any $t\ge 0$, $\ulin u(t\wedge\tau^u_{\ulin \xi})$ is an $(\F_{\ulin t})_{\ulin t\in\cal Q}$-stopping time.   \label{u-st}
\end{Lemma}
\begin{proof} Let $\ulin\xi\in\Xi$.
	Note that for any $\ulin t=(t_1,t_2)\in\cal Q$, by Lemmas \ref{T<S} and \ref{measurable},
	$$\{\ulin u(\tau^u_{\ulin\xi})\le \ulin t\}\cap \{u_1(\tau^u_{\ulin \xi})
	=\tau^1_{\xi_1}\} = \{\tau^1_{\xi_1}\le t_1\}\cap \{\theta(\tau^1_{\xi_1},\tau^2_{\xi_2}\wedge t_2)\ge \pi\} \in\F_{\ulin t}.$$
	Similarly, $\{\ulin u(\tau^u_{\ulin\xi})\le \ulin t\}\cap \{u_2(\tau^u_{\ulin \xi})=\tau^2_{\xi_1}\}\in \F_{\ulin t}$. Since either $u_1(\tau^u_{\ulin \xi})=\tau^1_{\xi_1}$ or $u_2(\tau^u_{\ulin \xi})=\tau^2_{\xi_2}$, we get $\{\ulin u(\tau^u_{\ulin\xi})\le \ulin t\}\in\F_{\ulin t}$. Thus, $\ulin u(\tau^u_{\ulin \xi})$ is an $(\F_{\ulin t}) $-stopping time.
	
	To prove that $\tau^u_{\ulin\xi}$ is an $(\F^u_t)_{t\ge 0}$-stopping time, it suffices to show that, for any $t\ge 0$ and $\ulin s\in\cal Q$, $\{\tau^u_{\ulin\xi}\le t\}\cap\{\ulin u(t)\le \ulin s\} \in\F_{\ulin s}$.
	We may choose a sequence $\ulin \xi^n=(\xi^n_1,\xi^n_2)_{n\in\N}$ in $\Xi$, which approximates $\ulin \xi$ such that $\{\tau^u_{\ulin\xi}\le t\}=\bigcap_{n=1}^\infty \{\tau^u_{\ulin\xi^n}< t\}$. Then it suffices to show that, for any $n\in\N$,
	$\{\tau^u_{\ulin\xi^n}<t\}\cap\{\ulin u(t)\le \ulin s\}\in\F_{\ulin s}$.
	Since $\ulin u$ is strictly increasing on $[0,\tau^u_{\ulin \xi^n}+\eps)$, 
	$$\{\tau^u_{\ulin\xi^n}<t\}\cap\{\ulin u(t)\le \ulin s\}=\{\ulin u(\tau^u_{\ulin\xi^n})<\ulin u(t)\le \ulin s\} =\bigcup_{\ulin r\in\Q^2\cap[\ulin 0,\ulin s]}\Big(\{\ulin u(\tau^u_{\ulin\xi^n})\le \ulin r\} \cap \{\ulin r<\ulin u(t)\le \ulin s\} \Big). $$
	Since $\ulin u(\tau^u_{\ulin \xi})$ and $\ulin u(t)$ are $(\F_{\ulin t})$-stopping times, the events in the union all belong to $\F_{\ulin s}$. So $\{\tau^u_{\ulin\xi^n}<t\}\cap\{\ulin u(t)\le \ulin s\}\in\F_{\ulin s}$, as desired.
	
	Let $t\ge 0$ and $\ulin s \in\cal Q$. Note that
	$$  \{\ulin u(t\wedge \tau^u_{\ulin\xi})\le \ulin s\} =(\{t<\tau^u_{\ulin\xi}\}\cap\{\ulin u(t)\le \ulin s\})\cup (\{\tau^u_{\ulin\xi}\le t\}\cap \{\ulin u(\tau^u_{\ulin\xi})\le \ulin s\} ).$$ 
 	The first event $\{t<\tau^u_{\ulin\xi}\}\cap\{\ulin u(t)\le \ulin s\}$ belongs to $\F_{\ulin s}$ because  from that $\tau^u_{\ulin\xi}$ is an $(\F^u_t)$-stopping time we know  $\{t<\tau^u_{\ulin\xi}\}\in\F^u_t=\F_{\ulin u(t)}$. The other event $\{\tau^u_{\ulin\xi}\le t\}\cap \{\ulin u(\tau^u_{\ulin\xi})\le \ulin s\}$ equals
	$$ \bigcap_{n\in\N} (\{\tau^u_{\ulin\xi^n}< t\}\cap \{\ulin u(\tau^u_{\ulin\xi^n})<\ulin s\})= \bigcap_{n\in\N} \bigcup_{\ulin r\in\Q^2\cap [\ulin 0,\ulin s)}(\{\ulin r\in{\cal D},\mA(\ulin r)<t\}\cap \{\ulin u(\tau^u_{\ulin\xi^n})\le \ulin r\}),$$
	where we used that $\tau^u_{\ulin\xi^n}=\mA(\ulin u(\tau^u_{\ulin\xi^n}))$. The event on the RHS of the above displayed formula belongs to $\F_{ \ulin s}$ because $\mA$ is $(\F_{ \ulin t})$-adapted and $\ulin u(\tau^u_{\ulin\xi^n})$ is an $(\F_{ \ulin t})$-stopping time.  Thus,  the event  $\{\tau^u_{\ulin\xi}\le t\}\cap \{\ulin u(\tau^u_{\ulin\xi})\le \ulin s\}$ also belongs to $\F_{ \ulin s}$. Then we get $ \{\ulin u(t\wedge \tau^u_{\ulin\xi})\le \ulin s\}\in\F_{ \ulin s}$, as desired.
\end{proof}

Since under $\PP_{iB}$,  for $j=1,2$, $\ha w_j(t_j)=w_j +\sqrt{\kappa}B_j(t_j)$, $t_j\ge 0$, where $(B_1(t_1))$ and $(B_2(t_2))$ are independent standard  Brownian motions, we get five $(\F_{ \ulin t})_{\ulin t\in\cal Q}$-martingales under $\PP_{iB}$: $\ha w_j(t_j)$, $\ha w_j(t_j)^2-\kappa t_j$, $j=1,2$, and $\ha w_1(t_1)\ha w_2(t_2)$.
Using  Lemmas \ref{OST}  and \ref{uisst} and the facts that $u_1(t),u_2(t)\le t$, we conclude that $\ha w_j^u(t )$,  $\ha w_j^u(t )^2-\kappa u_j(t)$, $j=1,2$, and $\ha w^u_1(t )\ha w^u_2(t )$ are all $(\F^u_t)$-martingales under $\PP_{iB}$. So we get quadratic variations and co-variation for $\ha w_j^u$, $j=1,2$:
\BGE \langle \ha w_j^u\rangle_t=\kappa u_j(t) ,\quad j=1,2;\quad \langle \ha w_1^u,\ha w_2^u\rangle_t\equiv 0.\label{variation}\EDE


Fix $\ulin \xi=(\xi_1,\xi_2)\in\Xi$. From Lemmas \ref{Doob} and \ref{OST} we know that $M_{iB\to cs}(u_1(t)\wedge \tau^1_{\xi_1},u_2(t)\wedge \tau^2_{\xi_2})$, $t\ge 0$, $s\in\{4,h\}$, is an $(\F^u_t)_{t\ge 0}$-martingale. Since $\tau^u_{\ulin\xi}$ is an $(\F^u_t)_{t\ge 0}$-stopping time, we see that
$$M_{iB\to cs}(u_1(t\wedge \tau^u_{\ulin\xi})\wedge \tau^1_{\xi_1},u_2(t\wedge \tau^u_{\ulin\xi})\wedge \tau^2_{\xi_2})=M_{iB\to cs}(u_1(t\wedge \tau^u_{\ulin\xi}) ,u_2(t\wedge \tau^u_{\ulin\xi}) )=M_{iB\to cs}^u(t\wedge \tau^u_{\ulin\xi}),\quad t\ge 0,$$
is an $(\F^u_t)_{t\ge 0}$-martingale. Recall that $M_{iB\to cs}^u=M_{iB\to cs}\circ \ulin u$. Since $[0,T^u)=\bigcup_{\xi\in\Xi^*} [0,\tau^u_{\ulin\xi}]$ and $\Xi^*$ is countable, we conclude that  $M_{iB\to cs}^u(t)$, $0\le t<T^u$, is an $(\F^u_t)_{t\ge 0}$-local martingale.

We now compute the SDE for $M_{iB\to c4}^u(t)$, $0\le t<T^u$, in terms of $\ha w^u_1$ and $\ha w^u_2$.
Using (\ref{M-i-4'}) we may express $M^u_{iB\to c4}$ as a product of several factors. Among these factors, $(W_{1,1}^u)^{\bb}$, $(W_{2,1}^u)^{\bb}$, $\sin_2(W_1^u-W_2^u)^{\frac 2\kappa}$, and $\sin_2(W_j^u-V_s^u)^{\frac 2\kappa}$, $j,s\in\{1,2\}$, contribute the martingale part of $M^u_{iB\to c4}$; and other factors are differentiable in $t$. For $j\ne k\in\{1,2\}$, using (\ref{dWkVs},\ref{-3},\ref{1/2-4/3}) we get the $(\F^u_t)$-adapted SDEs:
\BGE dW_j^u = W_{j,1}^ud\ha w^u_j+\Big(\frac \kappa 2-3\Big)W_{j,2}u_j'dt+\cot_2(W_j^u-W_k^u)(W_{k,1}^u)^2u_k'dt,\label{dWju}\EDE
$$\frac{dW_{j,1}^u}{W_{j,1}^u}=\frac{W_{j,2}^u}{W_{j,1}^u}\,d\ha w^u_j+\mbox{drift terms},$$
which imply that, for $s=1,2$,
$$\frac{d\sin_2(W_j^u-V_s^u)^{\frac 2\kappa}}{\sin_2(W_j^u-V_s^u)^{\frac 2\kappa}}=\frac 1\kappa  \cot_2(W_j^u-V_s^u)W_{j,1}^ud\ha w^u_j+\mbox{drift terms},$$
$$\frac{d\sin_2(W_1^u-W_2^u)^{\frac 2\kappa}}{\sin_2(W_1^u-W_2^u)^{\frac 2\kappa}}=\frac 1\kappa \cot_2(W_1^u-W_2^u) [ {W_{1,1}^u}\,d\ha w^u_1- {W_{2,1}^u}\,d\ha w^u_2 ]+\mbox{drift terms},$$
$$\frac{d(W_{j,1}^u)^{\bb}}{(W_{j,1}^u)^{\bb}}=\bb\frac{W_{j,2}^u}{W_{j,1}^u}\,d\ha w^u_j+\mbox{drift terms}.$$
Since we already know that $\ha w^u_1(t)$, $\ha w^u_2(t)$,  $M_{iB\to c4}^u(t)$, $0\le t<T^u$, are $(\F^u_t)_{t\ge 0}$-local martingales, we get
\BGE \frac{d M_{iB\to c4}^u}{M_{iB\to c4}^u}=\sum_{j=1}^2 \Big[\bb\frac{W_{j,2}^u}{W_{j,1}^u} +\sum_{X\in \{W_{3-j},V_1,V_2\}}  \frac 1\kappa \cot_2(W_j^u-X^u) {W_{j,1}^u}   \Big]\,d\ha w^u_j.\label{dMito4u}\EDE
One may also compute (\ref{dMito4u}) directly, and conclude that $M_{iB\to c4}^u(t)$ is an $(\F^u_t)$-local martingale.

From Lemmas \ref{Lemma-RN-T} and \ref{u-st} we know that, for any $\ulin\xi\in\Xi$ and $t\ge 0$, 
\BGE \frac{d \PP_{c4}|\F_{ \ulin u(t\wedge \tau^u_{\ulin\xi})}}{d \PP_{iB}|\F_{ \ulin u(t\wedge \tau^u_{\ulin\xi})}}=\frac{M_{iB\to c4}^u(t\wedge \tau^u_{\ulin\xi} )}{M_{iB\to c4}^u(0)}.\label{RNito4xi}\EDE
We will use  a Girsanov argument to derive the SDEs for $\ha w_j^u$, $j=1,2$, under $\PP_{c4}$.

\begin{Lemma}
Under $\PP_{c4}$,  there are two independent standard Brownian motions $B^u_j(t)$, $j=1,2$,  such that $\ha w_j^u$ satisfies the SDE
$$d\ha w_j^u=\sqrt{\kappa u_j' }dB_j^u +\Big[\kappa \bb\frac{W_{j,2}^u}{W_{j,1}^u} +\sum_{X\in\{W_{3-j},V_1,V_2\}}  \cot_2(W_j^u-X^u) {W_{j,1}^u}   \Big]u_j'dt,\quad 0\le t<\infty.$$
\end{Lemma}
\begin{proof}
For $j=1,2$, define a process $\til w_j^u$, which has initial value $w_j$, and satisfies the SDE
\BGE d\til w_j^u=d\ha w_j^u-\Big[\kappa \bb\frac{W_{j,2}^u}{W_{j,1}^u} +\sum_{X\in\{W_{3-j},V_1,V_2\}}   \cot_2(W_j^u-X^u) {W_{j,1}^u}   \Big]u_j'dt.\label{tilwju}\EDE
From (\ref{dMito4u}) we know that $\til w_j^u(t)M_{iB\to c4}^u(t)$, $0\le t<T^u$, is an $(\F^u_t)$-local martingale under $\PP_{iB}$.

We claim that,  for any $j\in\{1,2\}$ and $\ulin \xi\in\Xi$, $|\til w_j^u|$ is bounded on $[0,\tau^u_{\ulin \xi}]$ by a constant depending only on $\kappa,\ulin\xi,w_1,v_1,w_2,v_2$.
The proof is similar to that of Lemma \ref{uniform}. We may write $\til w_j^u(t)=\ha w_j^u(t)+A_j(t)$ using (\ref{tilwju}). From that proof of Lemma \ref{uniform} we know that $|\log(W^u_{j,1})|$, $|W_{j,2}^u|$, $\cot_2(W^u_j-W^u_{3-j})$, $W^u_j-V^u_s$, $s=1,2$, are all uniformly bounded on $[0,\tau^u_{\ulin \xi}]$. Since $\tau^u_{\ulin\xi}=\mA(\ulin u(\tau^u_{\ulin \xi}))$ and $K(\ulin u(\tau^u_{\ulin \xi}))$ is contained in the $\D$-hull generated by $\xi_1\cup \xi_2$, $\tau^u_{\ulin\xi}$ is also uniformly bounded. From (\ref{uj''}) we know that $u_j'$ is uniformly bounded on $[0,\tau^u_{\ulin \xi}]$. The above argument shows that $|A_j|$ is uniformly bounded on $[0,\tau^u_{\ulin \xi}]$. In order to prove the  uniform boundedness of $| \ha w_j^u|$  on $[0,\tau^u_{\ulin \xi}]$, it suffices to show that $|\ha w_j|$ is uniformly bounded on $[0,\tau^j_{\xi_j}]$.  For $\ha w_2(t)$, we have
\BGE \til g_2(t_2,v_1)>\ha w_2(t_2)>\til g_2(t_2,v_2),\quad 0\le t_2\le \tau^2_{\xi_2}. \label{<<}\EDE Since $\pa_{t_2} \til g_2(t_2,v_s)= \cot_2(\til g_2(t_2,v_s)-\ha w_2(t_2))$, by the uniform boundedness of  $|\cot_2(\til g_2(t_2,v_s)-\ha w_2(t_2))|=|\cot_2(V_s(0,t_2)-W_2(0,t_2))|$ and $t_2$ on $[0,\tau^2_{\xi_2}]$, we see that $|\til g_2(t_2,v_s)|$ is uniformly bounded on $ [0,\tau^2_{\xi_2}]$ for $s=1,2$. Using (\ref{<<}) we get the uniform boundedness of $\ha w_2(t_2)$ on $[0,\tau^j_{\xi_j}]$. The argument for $\ha w_1(t_1)$ is similar except that  we use $\til g_1(t_1,v_2+2\pi)>\ha w_1(t)>\til g_1(t_1,v_1)$. So the claim is proved.

From Lemma \ref{uniform} and the above claim, we see that, for any $j\in\{1,2\}$ and $\ulin \xi\in\Xi$, $\til w_j^u(t\wedge \tau^u_{\ulin \xi})M_{iB\to c4}^u(t\wedge \tau^u_{\ulin \xi})$, $t\ge 0$, is an $(\F^u_t)$-martingale under $\PP_{iB}$. Since this process is $(\F_{ \ulin u(t\wedge \tau^u_{\ulin \xi})})$-adapted, and $\F_{ \ulin u(t\wedge \tau^u_{\ulin \xi})}\subset \F_{ \ulin u(t)}=\F^u_t$, we see that it is an $(\F_{ \ulin u(t\wedge \tau^u_{\ulin \xi})})$-martingale. From (\ref{RNito4xi}) we see that $\til w_j^u(t\wedge \tau^u_{\ulin \xi})$, $t\ge 0$, is an $(\F_{ \ulin u(t\wedge \tau^u_{\ulin \xi})})$-martingale under $\PP_{c4}$. We now show that $(\til w_j^u(t\wedge \tau^u_{\ulin \xi}))$ is an $(\F^u_t)$-martingale under $\PP_{c4}$. To check this, we need to show that for any $t\ge s\ge 0$ and $A\in \F^u_s$,
\BGE \EE_{c4}[{\bf 1}_A \til w_j^u(t\wedge \tau^u_{\ulin \xi})]=\EE_{c4}[{\bf 1}_A \til w_j^u(s\wedge \tau^u_{\ulin \xi})].\label{E4=}\EDE
Write $A=A_1\cup A_2$, where $A_1=A\cap \{\tau^u_{\ulin \xi}< s\}$ and $A_2=A\cap \{\tau^u_{\ulin \xi}\ge s \}$. Since $t\wedge \tau^u_{\ulin \xi}=s\wedge \tau^u_{\ulin \xi}$ on $A_1$, (\ref{E4=}) holds with $A_1$ in place of $A$. From Lemma \ref{T<S}, $A_2=A\cap \{\ulin u(s)\le \ulin u(s\wedge \tau^u_{\ulin \xi}) \}\in \F_{ \ulin u(s\wedge \tau^u_{\ulin \xi})}$. So (\ref{E4=}) also holds with $A_2$ in place of $A$.
Combining, we get (\ref{E4=}), as desired. Thus, $\til w_j^u(t\wedge \tau^u_{\ulin \xi})$, $t\ge 0$, is an $(\F^u_t)$-martingale under $\PP_{c4}$. From Lemma \ref{T=infty} we know that $\PP_{c4}$-a.s.\ $T^u=\infty$.  Since $T^u=\sup_{\xi\in\Xi^*} \tau^u_{\ulin\xi}$,  we see that $\til w_j^u(t)$,  $0\le t<\infty$, is an  $(\F^u_t)$-local martingale  under $\PP_{c4}$.

From (\ref{variation}) we know that, under $\PP_{iB}$,
\BGE \langle \til w_j^u(\cdot\wedge \tau^u_{\ulin \xi} )\rangle_t=\kappa u_j(t\wedge \tau^u_{\ulin \xi}) ,\quad j=1,2;\quad \langle \til w_1^u(\cdot\wedge \tau^u_{\ulin \xi}),\til w_2^u(\cdot\wedge \tau^u_{\ulin \xi} )\rangle_t\equiv 0\label{varitaion4}\EDE
Since $\PP_{c4}\ll \PP_{iB}$ on $\F_{ \ulin u(t\wedge \tau^u_{\ulin \xi})}$ for any $t\ge 0$, we also have (\ref{varitaion4}) under $\PP_{c4}$. Since $T^u=\sup_{\xi\in\Xi^*} \tau^u_{\ulin\xi}$, we conclude that, under $\PP_{c4}$, \BGE \langle \til w_j^u \rangle_t=\kappa u_j(t ) ,\quad j=1,2;\quad \langle \til w_1^u ,\til w_2^u \rangle_t\equiv 0,\quad 0\le t<T^u=\infty.\label{varitaion4'}\EDE
Since $(\til w^u_j)$, $j=1,2$, are  $(\F^u_t)$-local martingales  under $\PP_{c4}$, we see that there are two independent standard Brownian motions $B^u_j(t)$, $j=1,2$, under $\PP_{c4}$, such that
$d\til w_j^u(t)=\sqrt{\kappa u_j'(t)}dB^u_j(t)$, $0\le t<\infty$. Using (\ref{tilwju}) we then complete the proof.
\end{proof}

\begin{proof}[Proof of Lemma \ref{ZSDE}]
Recall that $Z_j=W_j-V_j$, $j=1,2$. Since $W_1>V_1>W_2>V_2>W_1-2\pi$, and $\theta^u=V_1^u-V_2^u=\pi$, we have $Z_j^u\in(0,\pi)$, $j=1,2$. Let $k=3-j$.
Using (\ref{dWkVs}) we get
$$dV_j^u=-\cot_2(W_j^u-V_j^u)(W_{j,1}^u)^2u_j'dt-\cot_2(W_k^u-V_j^u)(W_{k,1}^u)^2u_k'dt.$$
Combining this formula with  (\ref{uj''},\ref{dWju}),
and that $V_j^u-V_{k}^u=\pm\pi$, we get (\ref{dZjTu}).
\end{proof}

\section{Transition Density} \label{section-transition-density}
In this section, we are going to find out the transition density of the process $(Z^u_1,Z^u_2)$ that satisfies (\ref{dZjTu}).
Define $B^u_+(t)$ and $B^u_-(t)$ such that
$$B^u_\pm(t)=\int_0^t \sqrt{\frac{ \sin(Z_1^u(s))}{\sin(Z_1^u(s))+\sin(Z_2^u(s))}}\,dB_1^u(s)\pm \int_0^t \sqrt{\frac{  \sin(Z_2^u(s))}{\sin(Z_1^u(s))+\sin(Z_2^u(s))}}\,dB_2^u(s).$$
Then both $B^u_+(t)$ and $B^u_-(t)$ are standard $(\F^u_t)$-Brownian motions, and their quadratic covariation satisfies
\BGE d\langle B^u_+,B^u_-\rangle_t=\cot_2({Z_1^u+Z_2^u})\tan_2( {Z_1^u- Z_2^u})dt.  \label{cov-Bpm}\EDE

Let $Z_\pm ^u=(Z_1^u\pm Z_2^u)/2$. Then $Z_+ ^u\in(0,\pi)$, $Z_-^u\in(-\frac\pi 2,\frac\pi 2)$, and they satisfy the SDEs
$$ dZ_+^u=\frac{\sqrt{\kappa}}2dB_+ ^u+2\cot(Z_+^u)dt,\quad 0\le t<\infty.$$
$$ dZ_-^u=\frac{\sqrt{\kappa}}2dB_- ^u-2\tan(Z_-^u)dt,\quad 0\le t<\infty.$$


We are going to follow the argument in \cite[Appendix B]{tip} to derive the transition density of  $(Z_+^u,Z_-^u)$. 
Let $X=\cos(Z_+^u) $ and $Y=\sin(Z_-^u)$. Then $X,Y\in(-1,1)$, and satisfy the SDEs
\BGE dX =-\frac{\sqrt{\kappa}}2\sqrt{1-X ^2}dB_+ ^u-\Big(2+\frac \kappa 8\Big)X  dt.\label{dX}\EDE
\BGE dY =+\frac{\sqrt{\kappa}}2\sqrt{1-Y ^2}dB_- ^u-\Big(2+\frac \kappa 8\Big)Y  dt.\label{dY}\EDE
From (\ref{cov-Bpm}) we have
\BGE d\langle X ,Y\rangle_t= -\frac \kappa 4XYdt.\label{cov-XY}\EDE
Since $X(t)^2+Y(t)^2=1-\sin(Z^u_1(t))\sin(Z^u_2(t))<1$, we see that $(X(t),Y(t))\in\D$ for all $t\ge 0$. We will find out the transition density $p_t((x,y),(x^*,y^*))$ for the joint process $(X,Y)$.

First, we assume that the transition density $p_t((x,y),(x^*,y^*))$ for $(X,Y)$ exists, and make some observations. For any fixed $({x^*},{y^*})\in\D$ and $t_0>0$, the process $M^{({x^*},{y^*})}_t:=p(t_0-t,(X(t),Y(t)),({x^*},{y^*}))$, $0\le t<t_0$, is a martingale. Assuming further that $p$ is smooth in $(t,x,y)$, then we get a PDE:
\BGE-\pa_t p+{\cal L}p=0,\label{PDE'}\EDE
where ${\cal L}$ is the second order differential operator:
$${\cal L}:=\frac\kappa 8(1-x^2)\pa_x^2 +\frac\kappa 8(1-y^2)\pa_y^2 -\frac \kappa 4 xy\pa_x\pa_y -(2+\frac \kappa 8)x\pa_x -(2+\frac \kappa 8)y\pa_y .$$
We will derive the eigenvectors and eigenvalues of ${\cal L}$. Note that for integers $n,m\ge 0$,
$$\L(x^ny^m)=-\frac\kappa 8(n+m)(n+m+\frac{16}\kappa)x^ny^m+\frac\kappa 8n(n-1)x^{n-2}y^m+\frac\kappa 8m(m-1)x^{n}y^{m-2}.$$
Define
\BGE \lambda_n=-\frac\kappa 8 n(n+\frac{16}\kappa),\quad n\in\N\cup\{0\}.\label{lambdan}\EDE Then $\L(x^ny^m)$ equals to $\lambda_{n+m} x^ny^m$ plus a polynomial in $x,y$ of degree less than $n+m$. Hence, for each $n,m\in\N\cup\{0\}$, there is a polynomial $P_{(n,m)}(x,y)$ of degree $n+m$, which can be written as $x^ny^m$ plus a polynomial of degree less than $n+m$, such that
$$\L P_{(n,m)}=\lambda_{n+m} P_{(n,m)}.$$
Let $\Psi(x,y)=(1-x^2-y^2)^{\frac8\kappa-1}$, and define the inner product $$\langle f,g\rangle_\Psi:=\int\!\int_{\D} f(x,y)g(x,y)\Psi(x,y)dxdy.$$
Since $\Psi\equiv 0$ on $\TT$, using integration by parts, we can show that for smooth functions $f$ and $g$ on $\lin{\D}$, $\langle \L f,g\rangle_\Psi=\langle f,\L g\rangle_\Psi$. In fact, if we let $a_{xx}=\frac\kappa 8(1-x^2)$, $a_{yy}=\frac \kappa 8(1-y^2)$, $a_{xy}=a_{yx}=-\frac \kappa 8xy$, $b_x=-(2+\frac\kappa 8)x$, and $b_y=-(2+\frac\kappa 8)y$, then both $\langle \L f,g\rangle_\Psi$ and $\langle f,\L g\rangle_\Psi$ equal
$$-\int\!\int_{\D} [(\pa_x f) a_{xx}\Psi (\pa_x g)+(\pa_x f)a_{xy}\Psi(\pa_y g)+(\pa_y f) a_{yx}\Psi (\pa_x g) +(\pa_y f) a_{yy}\Psi (\pa_y g) ]dxdy.$$
Here we use $\pa_x(a_{xx}\Psi)+\pa_y(a_{xy}\Psi)=b_x\Psi$ and $\pa_y(a_{yx}\Psi)+\pa_y(a_{yy}\Psi)=b_y\Psi$.

Thus, $P_{(n,m)}$ is orthogonal to $P_{(n',m')}$ w.r.t.\ $\langle\cdot\rangle_\Psi$ if $n+m\ne n'+m'$; and we may construct a sequence of polynomials $v_{(n,s)} $, $n=0,1,2,\dots$, $s=0,1,\dots,n$, such that $v_{(n,s)}$ is a polynomial in $x,y$ of degree $n$, $\L v_{(n,s)}=\lambda_n v_{(n,s)}$, and $\{v_{(n,s)}\}$ form an orthonormal basis w.r.t.\ $\langle \cdot\rangle_\Psi$. Here every $v_{(n,s)}$ is a linear combination of $P_{(j,k)}$ over $j,k\in\N\cup\{0\}$ such that $j+k=n$.
On the other hand, if a sequence of polynomials $v_{(n,s)} $, $n=0,1,2,\dots$, $s=0,1,\dots,n$,   form an orthonormal basis w.r.t.\ $\langle \cdot\rangle_\Psi$, and each $v_{(n,s)}$ has degree $n$, then $\L v_{(n,s)}=\lambda_n v_{(n,s)}$. This is because $v_{(n,s)}$ is orthogonal to all polynomials of degree less than $n$, and so it must be a linear combination of $P_{(j,k)}$ over $j,k\in\N\cup\{0\}$ such that $j+k=n$. From \cite[Section 1.2.2]{orthogonal-2}, we may choose $v_{(n,s)}$ such that for each $n\ge 0$, $v_{(n,0)},v_{(n,1)},\dots,v_{(n,n)}$ are given by
$$v_{n,j,1}=h_{n,j,1}P_j^{(\frac8\kappa-1,n-2j)}(2r^2-1)r^{n-2j}\cos((n-2j)\theta),\quad 0\le 2j\le n,$$
$$v_{n,j,2}=h_{n,j,2}P_j^{(\frac8\kappa-1,n-2j)}(2r^2-1)r^{n-2j}\sin((n-2j)\theta),\quad 0\le 2j\le n-1,$$
where $P_j^{(\frac8\kappa-1,n-2j)}$ are Jacobi polynomials of index $(\frac8\kappa-1,n-2j)$, $(r,\theta)$ is the polar coordinate of $(x,y)$: $x=r\cos\theta$ and $y=r\sin\theta$, and $h_{n,j,i}>0$ are normalization constants. Using the polar integration and Formula (\cite[Table 18.3.1]{NIST:DLMF})
\BGE \int_{-1}^1 P_j^{(\alpha,\beta)}(x)^2(1-x)^\alpha(1+x)^\beta dx=\frac{2^{\alpha+\beta+1}\Gamma(j+\alpha+1)\Gamma(j+\beta+1)} {j!(2j+\alpha+\beta+1)\Gamma(j+\alpha+\beta+1)}\label{L2-norm}\EDE
with $\alpha=\frac 8\kappa-1$ and $\beta=n-2j$, we compute
\BGE h_{n,j}:=h_{n,j,1}=h_{n,j,2}=\sqrt{\frac {1+{\bf 1}_{n\ne 2j}}\pi\cdot\frac{j!(n+\frac 8\kappa)\Gamma(n-j+\frac 8\kappa)}{\Gamma(j+\frac 8\kappa)\Gamma(n-j+1)}}.\label{h}\EDE
Using the super norm (over $[-1,1]$) of $P_j^{(\alpha,\beta)}$ (\cite[18.14.1,18.14.2]{NIST:DLMF}):
\BGE \|P_j^{(\alpha,\beta)}\|_\infty=\frac{\Gamma(\max\{\alpha,\beta\}+j+1)}{j!\Gamma(\max\{\alpha,\beta\}+1)},\quad\mbox{if } \max\{\alpha,\beta\}\ge -1/2\mbox{ and }\min\{\alpha,\beta\}> -1,\label{supernorm}\EDE
we get
\BGE \|v_{n,j,1}\|_\infty =\|v_{n,j,2}\|_\infty=h_{n,j}\max\Big\{\frac{\Gamma(\frac 8\kappa+j)}{j!\Gamma(\frac 8\kappa)},\frac{\Gamma(n-j+1)}{j!\Gamma(n-2j+1)}\Big\}.\label{v}\EDE

For $t>0$, $(x,y),(x^*,y^*)\in\D$, we define
\BGE p_t((x,y),(x^*,y^*))=\sum_{n=0}^\infty\sum_{s=0}^n \Psi(x^*,y^*)v_{(n,s)}(x,y)v_{(n,s)}(x^*,y^*)e^{\lambda_nt}. \label{pt}\EDE
Let $p_\infty(x^*,y^*)$ be the  term for $n=s=0$. Since $\lambda_0=0$ and $P^{\alpha,\beta}_0\equiv 1$, we have
\BGE p_\infty(x^*,y^*)=\frac {8}{\pi\kappa} \Psi(x^*,y^*)=\frac {8}{\pi\kappa} (1-(x^*)^2-(y^*)^2)^{\frac8\kappa-1}.\label{p-infty}\EDE

\begin{Lemma}
  For any $t_0>0$, the series in (\ref{pt}) converges uniformly on $[t_0,\infty)\times \lin\D\times \lin\D$, and there is $C_{t_0}\in(0,\infty)$ depending only on $\kappa$ and $t_0$ such that
$$ |p_t((x,y),(x^*,y^*))-p_\infty(x^*,y^*)|\le  C_{t_0} e^{-(2+\frac\kappa 8)t}p_\infty(x^*,y^*),\quad t\ge t_0,\quad (x,y),(x^*,y^*)\in\lin\D.$$ 
Moreover, for any $t>0$ and $(x^*,y^*)\in \lin\D$,
\BGE p_\infty(x^*,y^*)=\int\!\int_{\D} p_\infty(x,y)p_t((x,y),(x^*,y^*))dxdy.\label{invar}\EDE \label{asym-lemma}
\end{Lemma}
\begin{proof}
  The uniform convergence of the series in (\ref{pt}) and the first formula follows from Stirling's formula, (\ref{h},\ref{v}), and the facts that $\lambda_1= -(2+\frac\kappa 8)>\lambda_n$ for any $n>1$ and $\lambda_n\asymp -\frac\kappa 8n^2$ for big $n$.   Formula (\ref{invar}) follows from the orthogonality of $v_{n,s}$ w.r.t.\ $\langle\cdot,\cdot\rangle_\Psi$ and the uniform convergence of the series in (\ref{pt}).
\end{proof}

\begin{Lemma} Under $\PP_{c4}$,
$p_t((x,y),(x^*,y^*))$ is the transition density for $(X(t),Y(t))$ that satisfies (\ref{dX},\ref{dY},\ref{cov-XY}), and $p_\infty$ is the invariant density.
\end{Lemma}
\begin{proof}
Fix $(x,y)\in\D$. Let $(X(t),Y(t))$, $t\ge 0$, be the process that satisfies (\ref{dX},\ref{dY},\ref{cov-XY}) with initial value $(x,y)$. Fix $t_0>0$. For the first statement, it suffices to show that, for any $f\in C(\lin\D,\R)$,
\BGE \EE_{c4}[f(X_{t_0},Y_{t_0})]=\int\!\int_{\D} p_{t_0}((x,y),(x^*,y^*))f(x^*,y^*)dx^*dy^*.\label{Efp}\EDE

Since  $\L v_{(n,s)}=\lambda_n v_{(n,s)}$, every function $v_{(n,s)}(x,y) e^{\lambda_nt}$ solves (\ref{PDE'}). Let $f$ be a polynomial in $x,y$. Let $a_{(n,s)}=\langle f,v_{(n,s)}\rangle_\Psi$. Then $f(x,y)=\sum_{n=0}^\infty\sum_{s=0}^n a_{(n,s)}v_{(n,s)}(x,y)$, where all but finitely many $a_{(n,s)}$ are not zero. Define $$f(t,(x,y))=\sum_{n=0}^\infty\sum_{s=0}^n  a_{(n,s)} v_{(n,s)}(x,y)e^{\lambda_nt}= \int\!\int_{\D} p_{t_0}((x,y),(x^*,y^*))f(x^*,y^*)dx^*dy^*.$$
Then $f(t,(x,y))$ solves (\ref{PDE'}) since it is a linear combination of $v_{(n,s)}(x,y) e^{\lambda_nt}$. Let $(X(t),Y(t))$ be a stochastic process in $\D$, which solves (\ref{dX},\ref{dY},\ref{cov-XY}) with initial value $(x,y)$.  Fix $t_0>0$ and define $M_t=f(t_0-t,(X(t),Y(t)))$, $0\le t\le t_0$. By It\^o's formula, $(M_t)$ is a bounded martingale w.r.t.\ $\PP_4$, which implies that $\EE_{c4}[f(X({t_0}),Y(t_0))]=\EE_{c4}[M_{t_0}]=M_0=f(t_0,(x,y))$. So we get (\ref{Efp}) for a polynomial $f$. Formula (\ref{Efp}) for a general $f\in C(\lin\D,\R)$ follows from Stone-Weierstrass theorem. The statement on $p_\infty$ follows immediately from (\ref{invar}).
\end{proof}

\begin{Corollary}
	Under $\PP_{c4}$, the transition density for $(Z^u_1,Z^u_2)$ that satisfies (\ref{dZjTu}) is $$ p^Z_t((z_1,z_2),(z_1^*,z_2^*)) =p_t((\cos_2(z_1+z_2),\sin_2(z_1-z_2)),(\cos_2(z_1^*+z_2^*),\sin_2(z_1^*-z_2^*)))\frac{\sin z_1^* +\sin z_2^*}{4},$$
and the invariant density is
$$p^Z_\infty ( z_1^*,z_2^*)= p_\infty( \cos_2(z_1^*+z_2^*),\sin_2(z_1^*-z_2^*))\frac{\sin z_1^* +\sin z_2^*}{4}.$$
\end{Corollary}
\begin{proof}
This follows from the above lemma and the fact that $X(t):=\cos_2(Z^u_1(t)+Z^u_2(t))$ and $Y(t):=\sin_2(Z^u_1(t)-Z^u_2(t))$ satisfy (\ref{dX},\ref{dY},\ref{cov-XY}).
\end{proof}

Next, we will
derive the transition density $\til p^Z_t((z_1,z_2),(z_1^*,z_2^*))$ under $\PP_{2}$ for $(Z^u_1,Z^u_2)$. Now $B^u_1$ and $B^u_2$ are not standard Brownian motions under $\PP_{2}$, and we no longer have $\PP_{2}$-a.s.\ $T^u=\infty$. In fact, we will see that $\PP_{2}$-a.s.\ $T^u<\infty$. By saying that $\til p^Z_t((z_1,z_2),(z_1^*,z_2^*))$ is the transition density for $(Z^u_1,Z^u_2)$ under $\PP_{2}$, we mean that, for any $t>0$ and $(z_1,z_2)\in(0,\pi)^2$,  if $(Z^u_1,Z^u_2)$ starts from $(z_1,z_2)$, then for any bounded measurable function $f$ on $(0,\pi)^2$, we have
$$\EE_{2}[{\bf 1}_{\{T^u>t\}}f(Z^u_1(t),Z^u_2(t))]=\int_0^\pi\!\int_0^\pi \til p^Z_t((z_1,z_2),(z_1^*,z_2^*))f(z_1^*,z_2^*)dz_2^*dz_1^*.$$
In particular, we have $\PP_{2}[T^u>t]=\int_0^\pi\!\int_0^\pi \til p^Z_t((z_1,z_2),(z_1^*,z_2^*)) dz_2^*dz_1^*$.

From (\ref{RN-4to2}) we know that, for any $t\ge 0$,
$$\frac{d\PP_{2}|\F^u_t\cap\{T^u>t\}}{d\PP_{c4}|\F^u_t\cap\{T^u>t\}}=\frac{M_{c4\to ch}^u(t)}{M_{c4\to ch}^u(0)} .$$
Let $G^u(z_1,z_2)$ be a function defined for $z_1,z_2\in(0,\pi)$ such that
$$G^u(z_1,z_2)=[\sin_2 z_1\sin_2 z_2]^{\frac 8\kappa-1}\cos_2(z_1-z_2)^{\frac 4\kappa}F\Big(\frac{\cos_2 z_1\cos_2 z_2}{\cos_2(z_1-z_2)}\Big)^{-1}.$$
From (\ref{G(WV)},\ref{M2to4}) we get $G(W^u_1,V^u_1;W^u_2,V^u_2)=G^u(Z^u_1,Z^u_2)$ and
$$M^u_{c4\to ch}(t)=e^{-\alpha_0t} G^u(Z^u_1(t),Z^u_2(t))^{-1}.$$
Recall that $\PP_{c4}$-a.s.\ $T^u=\infty$. So we obtain the following lemma.

\begin{Lemma}
	Under $\PP_{2}$, $(Z^u_1(t),Z^u_2(t))$, $0\le t<T^u$, is  a Markov process with transition density
$$\til p^Z_t((z_1,z_2),(z_1^*,z_2^*)):=  e^{-\alpha_0t} p^Z_t((z_1,z_2),(z_1^*,z_2^*)) \frac{G^u(z_1,z_2)}{G^u(z_1^*,z_2^*)}.$$
\label{Lemma-til-p}
\end{Lemma}

Note that for some explicit constant $C\in(0,\infty)$ depending on $\kappa$,
$$\frac{p^Z_\infty(z_1,z_2)} {G^u(z_1,z_2)}=C[ {\cos_2 z_1\cos_2 z_2} ]^{\frac8\kappa -1}    \sin_2(z_1+z_2)  {\cos_2(z_1-z_2)}  ^{1-\frac 4\kappa}   {F \Big(\frac{\cos_2 z_1\cos_2 z_2}{\cos_2(z_1-z_2)} \Big)}.$$
So $\frac{p^Z_\infty(z_1,z_2)} {G^u(z_1,z_2)} $ extends to a  continuous function on $[0,\pi]^2$, which vanishes at the corners. We may normalize it to get a probability density, i.e., we define
\BGE {\cal Z}=\int_0^\pi\!\int_0^\pi \frac{p^Z_\infty(z_1,z_2)} {G^u(z_1,z_2)} dz_1dz_2\in(0,\infty),\label{calZ}\EDE
\BGE \til p^Z_\infty(z_1,z_2)=\frac 1{{\cal Z}}\frac{p^Z_\infty(z_1,z_2)} {G^u(z_1,z_2)},\quad z_1,z_2\in[0,\pi].\label{tilpZinfty}\EDE

From now on, if a quantity $Q$ depends on $t\in(0,\infty)$ and other variables $\ulin x$, and $f$ is a positive function on $(0,\infty)$, we write $Q$ as $O(f(t))$, if for any $t_0>0$ there is $C_{t_0}\in(0,\infty)$ depending only on $\kappa$ and $t_0$ such that for any $t\ge t_0$ and any $\ulin x$, $|Q(t,\ulin x)|\le Cf(t)$.

\begin{Lemma}
\begin{enumerate}
  \item [(i)] For any $t>0$ and $z_1^*,z_2^*\in[0,\pi]$, \BGE \int_0^\pi\!\int_0^\pi  \til p^Z_\infty(z_1,z_2)\til p^Z_t((z_1,z_2),(z_1^*,z_2^*))dz_1dz_2=\til p^Z_\infty(z_1^*,z_2^*)e^{-\alpha_0t}.
	\label{invariant-psudo}\EDE
This means, under the law $\PP_{2}$, if the process $(Z^u_1,Z^u_2)$ starts from a random point $(z_1,z_2)\in (0,\pi)^2$ with density $\til p^Z_\infty$, then for any deterministic $t\ge 0$, the density of $(Z^u_1(t),Z^u_2(t))$ (that is survived)  at time $t$ is $e^{-\alpha_0 t} \til p^Z_\infty$. 
\item [(ii)]	For any $(z_1,z_2)\in(0,\pi)^2$ and a process $(Z^u_1,Z^u_2)$ started from $(z_1,z_2)$, we have
	\BGE\PP_{2}[T^u>t]={\cal Z} G^u(z_1,z_2) e^{-\alpha_0t}(1+O(e^{-(2+\frac\kappa 8)t}));
	\label{P[T>t]}\EDE
	\BGE \til p^Z_t((z_1,z_2),(z_1^*,z_2^*))=\PP_{2}[T^u>t]\til p^Z_\infty(z_1^*,z_2^*)(1+O(e^{-(2+\frac\kappa 8)t})).\label{tilptT>t}\EDE \label{property-til-p}
Here we emphasize that the implicit constants in the $O$ symbols do not depend on $(z_1,z_2)$.
\end{enumerate}
\end{Lemma}
\begin{proof}
	Part (i) follows easily from (\ref{invar}). For part (ii), suppose $(Z^u_1,Z^u_2)$  starts from $(z_1,z_2)$. Using Lemmas \ref{asym-lemma}, \ref{Lemma-til-p} and formulas ( \ref{calZ},\ref{tilpZinfty}), we get
	$$\PP_{2}[T^u>t]=\int_0^\pi\!\int_0^\pi \til p^Z_t((z_1,z_2),(z_1^*,z_2^*))dz_1^*dz_2^*$$
	$$=\int_0^\pi\!\int_0^\pi e^{-\alpha_0t}  p^Z_t((z_1,z_2),(z_1^*,z_2^*)) \frac{G^u(z_1,z_2)}{G^u(z_1^*,z_2^*)}dz_1^*dz_2^*$$
	$$=\int_0^\pi\!\int_0^\pi e^{-\alpha_0t}  p^Z_\infty(z_1^*,z_2^*)(1+O(e^{-(2+\frac\kappa 8)t})) \frac{G^u(z_1,z_2)}{G^u(z_1^*,z_2^*)}dz_1^*dz_2^*$$
	$$={\cal Z} G^u(z_1,z_2) e^{-\alpha_0t}(1+O(e^{-(2+\frac\kappa 8)t})),$$
	which is (\ref{P[T>t]}); and
	$$ \til p^Z_t((z_1,z_2),(z_1^*,z_2^*))=e^{-\alpha_0t}  p^Z_t((z_1,z_2),(z_1^*,z_2^*)) \frac{G^u(z_1,z_2)}{G^u(z_1^*,z_2^*)}$$
	$$=e^{-\alpha_0t}  p^Z_\infty(z_1^*,z_2^*)(1+O(e^{-(2+\frac\kappa 8)t})) \frac{G^u(z_1,z_2)}{G^u(z_1^*,z_2^*)}$$
	$$=e^{-\alpha_0 t} {\cal Z} \til p^Z_\infty(z_1^*,z_2^*)(1+O(e^{-(2+\frac\kappa 8)t})) G^u(z_1,z_2),$$
	which together with (\ref{P[T>t]}) implies (\ref{tilptT>t}).
\end{proof}

\section{Proofs of Main Theorems} \label{section-proofs}
We will prove the main theorems of the paper in this section. We will   need the Domain Markov Property for $2$-SLE in the following form.

\begin{Lemma}
	 Let $(\eta_1,\eta_2)$ be a $2$-SLE$_\kappa$ in a simply connected domain $D$ with link pattern $(a_1\to b_1;a_2\to b_2)$. Suppose, for $j=1,2$, $\eta_j$ is parametrized by the chordal capacity viewed from $b_j$ (determined by a conformal map from $D$ onto $\HH$ that takes $b_j$ to $\infty$). Note that the lifetime of the parametrized $\eta_1$ and $\eta_2$ are both $\infty$. Let $(\F^j_t)_{t\ge 0}$ be the filtration generated by $\eta_j$, $j=1,2$, which together generate a  separable $\cal Q$-indexed filtration $(\F_{ \ulin t})_{\ulin t\in\cal Q}$. Let $\ulin T=(T_1,T_2)$ be a finite $(\F_{ \ulin t})$-stopping time. Let $D_{\ulin T}^j$ denote the connected component of $D\sem (\eta_1([0,T_1])\cup\eta_2([0,T_2]))$ whose boundary contains $b_j$, $j=1,2$ Then
	\begin{enumerate}
		\item[(i)] Conditioning on $\F_{ \ulin T}$ and the event  that $D_{\ulin T}^1=D_{\ulin T}^2=:D_{\ulin T}$ and $\eta_1(T_1)\ne \eta_2(T_2)$,  $\eta_1|_{[T_1,\infty]}$ and $\eta_2|_{[T_2,\infty]})$ form a  $2$-SLE$_\kappa$ in $D_{\ulin T}$ with link pattern $(\eta_1(T_1)\to b_1;\eta_2(T_2)\to b_2)$.
		\item[(ii)] Conditioning on $\F_{ \ulin T}$ and the event that $D_{\ulin T}^1\ne D_{\ulin T}^2$,   $\eta_j|_{[T_j,\infty]}$ is a chordal SLE$_\kappa$ curve  in $D_{\ulin T}^j$ from $\eta_j(T_j)$ to $b_j$, $j=1,2$, and $\eta_1|_{[T_1,\infty]}$ and $\eta_2|_{[T_2,\infty]}$ are independent.
	\end{enumerate} \label{DMP-bi-chordal}
\end{Lemma}
\begin{proof}
	By the property of $2$-SLE$_\kappa$, conditioning on  $\F^2_\infty$, $\eta_1$ is a chordal SLE$_\kappa$ curve from $a_1$ to $b_1$ in a connected component of $D\sem \eta_2$. Let $\F^{(2,\infty)}_{t_1}=\F^1_{t_1}\vee \F^2_\infty$. Then we get a filtration $(\F^{(2,\infty)}_{t_1})_{t_1\ge 0}$. Since for any $t_1\ge 0$,  $\{T_1\le t_1\}=\bigcup_{n\in\N} \{\ulin T\le (t_1,n)\}$, we see that $T_1$ is an $(\F^{(2,\infty)}_{t_1})$-stopping time. If $A\in\F_{\ulin T}$, then from
$A\cap \{T_1\le t_1\}=\bigcup_{n\in\N} A\cap \{\ulin T\le (t_1,n)\}$ , $t_1\ge 0$,
we see that $A\in \F^{(2,\infty)}_{T_1}$. Thus, $\F_{\ulin T}\vee \F^2_\infty\subset  \F^{(2,\infty)}_{T_1}$.

By the DMP of chordal SLE$_\kappa$, conditioning on $\F^{(2,\infty)}_{T_1}$,  $\eta_1|_{[T_1,\infty]}$ has the law of a chordal SLE$_\kappa$ curve from $\eta_1(T_1)$ to $b_1$ in a connected component of $D\sem (\eta_1([0,T_1])\cup \eta_2)$, which is denoted by $D_{T_1,\infty}$.
Note that the triple $(D_{T_1,\infty};\eta_1(T_1),b_1)$ is measurable w.r.t.\ $\F_{ \ulin T}\vee \F^2_\infty$. Since  $\F_{ \ulin T}\vee \F^2_\infty\subset\F^{(2,\infty)}_{T_1}$, we conclude that, conditioning on $\F_{ \ulin T}\vee \F^2_\infty$, $\eta_1|_{[T_1,\infty]}$ also has the law of a chordal SLE$_\kappa$ in $D_{T_1,\infty}$ from $\eta_1(T_1)$ to $b_1$.
Since $\F_{ \ulin T}\vee \F^2_\infty$ agrees with the $\sigma$-algebra generated by $\F_{ \ulin T}$ and $\eta_2|_{[T_2,\infty]}$,  we can say that, conditioning first on $\F_{ \ulin T}$ and then on $\eta_2|_{[T_2,\infty]}$, $\eta_1|_{[T_1,\infty]}$ has the law of a chordal SLE$_\kappa$ in $D_{T_1,\infty}$ from $\eta_1(T_1)$ to $b_1$. Similarly, conditioning first on $\F_{ \ulin T}$ and then on $\eta_1|_{[T_1,\infty]}$, $\eta_2|_{[T_2,\infty]}$ has the law of a chordal SLE$_\kappa$ in $D_{\infty,T_2}$. On the event that $D_{\ulin T}^1\ne D_{\ulin T}^2$, $D_{\infty,T_j}$ does not depend on $\eta_{3-j}([T_{3-j},\infty])$, so $\eta_1|_{[T_1,\infty]}$ and $\eta_2|_{[T_2,\infty]}$ are conditionally independent given $\F_{\ulin T}$ on this event. This is (ii). On the event that $D_{\ulin T}^1=D_{\ulin T}^2=:D_{\ulin T}$ and $\eta_1(T_1)\ne \eta_2(T_2)$, the conditional joint law of $\eta_1|_{[T_1,\infty]}$ and $\eta_2|_{[T_2,\infty]}$ given $\F_{\ulin T}$ agrees with that of the $2$-SLE$_\kappa$ in $D_{\ulin T}$ with link pattern $(\eta_1(T_1)\to b_1;\eta_2(T_2)\to b_2)$. So we get (i).
%
\end{proof}

\begin{proof} [Proof of Theorem \ref{main-thm1}]
We first work on (\ref{main-est-1}).
  By Koebe's distortion theorem, it suffices to prove the theorem for $D=\D$ and $z_0=0$. By symmetry, we may assume that   $a_j=e^{iw_j}$ and $b_j=e^{iv_j}$, $j=1,2$, and $w_1>v_1>w_2>v_2>w_1-2\pi$.  We use $p(w_1,v_1,w_2,v_2;r)$ to denote the probability that both $\ha \eta_1$ and $\ha \eta_2$ have distance less than $r$ from $0$. Since $G_{\D;e^{iw_1},e^{iv_1};e^{iw_2},e^{iv_2}}(0)$ agrees with the $G(w_1,v_1;w_2,v_2)$ defined by (\ref{G(WV)}), it suffices to show that, for some constant $C_0\in(0,\infty)$,
\BGE p(w_1,v_1,w_2,v_2;r)= C_0G(w_1,v_1;w_2,v_2) r^{\alpha_0} (1+O (r^{\beta_0})),\quad \mbox{as } r\to 0^+.\label{Green-disk}\EDE

  For $j=1,2$, suppose $\ha \eta_j$ is oriented from $a_j$ to $b_j$, and let $\eta_j$ be the part of $\ha \eta_j$ from $a_j$ up to $b_j$ or the first time that $\ha \eta_j$ separates $0$ from any of $b_j,a_{3-j},b_{3-j}$ if such time exists.
    Then we may parametrize $\eta_1$ and $\eta_2$ by the radial capacity viewed from $0$ such that they are radial Loewner curves with lefetime $\til T_1$ and $\til T_2$ driven by some functions $\ha w_1$ and $\ha w_2$ with initial values $w_1$ and $w_2$, respectively. Then the law of $(\ha w_1,\ha w_2)$ is  $\PP_{2}^{w_1,v_1;w_2,v_2}$ as defined in Section \ref{section-Probability-measures}.


 We use the symbols in  Section \ref{section-Ensemble}. Now we write $K_{\ulin t}$ and $g_{\ulin t}$ for $K(t_1,t_2)$ and $g((t_1,t_2),\cdot)$, and let $D_{\ulin t}=\D\sem K_{\ulin t}$. Recall that $g_{\ulin t} $ maps $D_{\ulin t} $ conformally onto $\D$, fixes $0$, and has derivative $e^{\mA(\ulin t)}$ at $0$. Moreover, $g_{\ulin t}( \eta_j(t_j))=e^{iW_j(\ulin t)}$ and $g_{\ulin t}(e^{iv_j})=e^{iV_j(\ulin t)}$, $j=1,2$. Suppose $\ulin T=(T_1,T_2)$ is an $(\F_{\ulin t})_{\ulin t\in\cal Q}$-stopping time such that $T_j<\til T_j$, $j=1,2$. Then $\ulin T$ corresponds to a stopping time w.r.t.\ the ${\cal Q}$-indexed filtration  generated by $\ha \eta_1,\ha \eta_2$, which are parametrized by chordal capacities viewed from $b_1,b_2$, respectively. By Lemma \ref{DMP-bi-chordal}, conditionally on $\F_{\ulin T}$ and the event that $\ulin T\in\cal D$, the $g_{\ulin T}$-images of the parts of $\ha \eta_j$ from $\eta_j(u_j(t_0))$ to $e^{iv_j}$, $j=1,2$, together form a $2$-SLE$_\kappa$ in $\D$ with link pattern $(e^{iW_1(\ulin T)}\to e^{iV_1(\ulin T)};e^{iW_2(\ulin T)}\to e^{iV_2(\ulin T)})$.

Suppose that $v_1-v_2=\pi$ so that the time curve $\ulin u$ in Section \ref{section-Time-curve} can be defined, and we may use the results and symbols there. 
Fix $r\in(0,1/4)$. Suppose that it happens that $\dist(0,\ha \eta_j)<r$, $j=1,2$.  Then the parts of $\ha \eta_1$ and $\ha \eta_2$ up to their respective hitting times at  $\{|z|=r\}$ do not intersect. Because if they did intersect, then they together could disconnect $e^{iv_1}$ and $e^{iv_2}$ in $\D$, and the rest parts of $\ha \eta_1$ and $\ha \eta_2$ would grow in different domains, and could not both visit the disc $\{|z|<r\}$, which is a contradiction. Thus, for $j=1,2$, the above part of $\ha \eta_j$ does not disconnect $0$ from any of $e^{iv_j},e^{iw_{3-j}},e^{iv_{3-j}}$, and so belongs to $\eta_j$. Let $\tau_j$ be the first hitting time of $\eta_j$ at $\{|z|=r\}$, $j=1,2$. Then  $(\tau_1,\tau_2)\in\cal D$. By Koebe's $1/4$ theorem, we get $\tau_j\ge -\log(4r)$, $j=1,2$. Recall that $\theta(\tau_1,0)<\pi<\theta(0,\tau_2)$. So there is $\ulin s=(s_1,s_2)\in \{(\tau_1,t_2):0\le t_2\le \tau_2\}\cup\{(t_1,\tau_2):0\le t_1\le \tau_1\}$ such that $\theta(\ulin s)=\pi$. This implies that $\mA(s_1,s_2)<T^u$ and $s_j=u_j(\mA(s_1,s_2))$, $j=1,2$. Using (\ref{mA-est}) we get $T^u>s_1\vee s_2\ge -\log(4r)$.

Now fix $t_0\in[0,-\log(4r)]$. Then $\{\dist(0,\ha \eta_j)<r,j=1,2\}\subset \{T^u> t_0\}$. When $T^u>t_0$ happens, since $u_j(t_0)\le t_0$, by Koebe's $1/4$ theorem, $\dist(0,\eta_j[0,u_j(t_0)])\ge r$, $j=1,2$. Thus, $\dist(0,\ha \eta_j)<r$, $j=1,2$, if and only if $T^u>t_0$ and the parts of $\ha \eta_j$ after $\eta_j(u_j(t_0))$, $j=1,2$, both visit the disc $\{|z|<r\}$. Suppose $T^u>t_0$ does happen. Let $R_1<R_2\in(0,1)$ be such that $\frac{e^{-t_0}R_1}{(1-R_1)^2}=\frac{e^{-t_0}R_2}{(1+R_2)^2}=r$. Since $(g^u_{t_0})'(0)=e^{\mA(\ulin u(t_0))}=e^{t_0}$ and $r\le \frac 14 e^{-t_0}$, by Koebe's distortion theorem, $\{|z|<r\}\subset D_{\ulin u(t_0)}$, and
\BGE \{|z|<R_1\}\subset g^u_{t_0}(\{|z|<r\})\subset \{|z|<R_2\}.\label{R1R2}\EDE
By rotation symmetry, there is a function $p(z_1,z_2;r)$ such that $p(w_1,v_1,w_2,v_2;r)=p(w_1-v_1,w_2-v_2;r)$ if $v_1-v_2=\pi$. 
From the conditional joint law of the $ g_{\ulin u(t_0)}$-images of the parts of $\ha \eta_j$ after $\eta_j(u_j(t_0))$, $j=1,2$, given $\F^u_{t_0}$, and the facts that $V_1^u(t_0)-V_2^u(t_0)=\pi$ and $Z_j^u=W_j^u-V_j^u$, $j=1,2$,  we get
\BGE p(Z_1^{u}(t_0),Z_2^u(t_0);R_1)\le \PP[\dist(0,\ha \eta_j)<r,j=1,2|\F^u_{t_0},T^u>t_0]\le p(Z_1^{u}(t_0),Z_2^u(t_0);R_2).\label{p-condition}\EDE

We will first find the asymptotic of  $p(r):=\int_0^\pi\int_0^\pi p(z_1,z_2;r)\til p^Z_\infty(z_1,z_2)dz_1dz_2$  as $r\to 0^+$. Such  $p(r)$ is the probability that the two curves $\ha \eta_1$ and $\ha \eta_2$ in a $2$-SLE$_\kappa$ in $\D$ with link pattern $(e^{iz_1}\to 1;-e^{iz_2}\to -1)$  both get within distance $r$ from $0$, where $z_1,z_2$ are random numbers in $(0,\pi)$ with joint density $\til p^Z_\infty$. From (\ref{P[T>t]}) we know that,  $\PP[T^u>t_0]=e^{-\alpha_0t_0}$, and conditioning on $T^u>t_0$, $(Z^u_1(t_0),Z^u_2(t_0))$ also has joint density $\til p^Z_\infty$.

Suppose $0<t<T^u$. Let $d_j=\dist(0,\eta_j([0,u_j(t)]))$. Since $\mA(u_1(t),u_2(t))=t$, by Schwarz Lemma, we have $d_1\wedge d_2\le e^{-t}$. By symmetry we may assume that $d_1\le d_2$. Since $\theta(u_1(t),u_2(t))=\pi$, we know that the harmonic measure of the union of $\eta_2([0,u_2(t)])$ and the subarc of $\TT$ between $e^{iv_1 }$ and $e^{iv_1 }$ that contains $e^{iw_2 }$ in $\D\sem K(u_1(t),u_2(t))$ viewed from $0$ is exactly $1/2$. Using Beurling estimate, we get $1/2\le 2(d_1/d_2)^{1/2}$, which implies that $d_2\le 16 d_1$. Since $d_1\le e^{-t}$, we get $d_1,d_2\le 16e^{-t}$, and so $\dist(0,\eta_j)\le d_j\le 16 e^{-t}$, $j=1,2$. This means that $p(r)\ge \PP[T^u>t]=e^{-\alpha_0t}>0$ if $r>16 e^{-t}$. So $p$ is positive.
By (\ref{p-condition}) we get
$$e^{-\alpha_0t_0}p(R_1)\le p(r)\le e^{-\alpha_0t_0}p(R_2),\quad \mbox{if }\frac{e^{-t_0}R_1}{(1-R_1)^2}=\frac{e^{-t_0}R_2}{(1+R_2)^2}=r.$$
 Let $q(r)=r^{-\alpha_0}p(r)$. Suppose $r,R\in(0,1)$ satisfy that $r<\frac{R}{(1+R)^2}$. By choosing $t_0>0$ such that $e^{t_0}=\frac{R/r}{(1+R)^2}$, we conclude from the above formula that $p(r)\le e^{-\alpha_0 t_0}p(R)=(\frac{R/r}{(1+R)^2})^{-\alpha_0} p(R)$. Thus, $q(r)\le (1+R)^{2\alpha_0} q(R)$. Similarly, by choosing $t_0>0$ such that $e^{t_0}=\frac{R/r}{(1-R)^2}$, we get $q(r)\ge (1-R)^{2\alpha_0} q(R)$. So we have
\BGE (1-R)^{2\alpha_0} q(R)\le q(r)\le (1+R)^{2\alpha_0} q(R),\quad \mbox{if }r<\frac{R}{(1+R)^2}.\label{Rr}\EDE
Thus, $\lim_{r\to 0^+}\log (q(r))$ converges to a finite number, which implies that $\lim_{r\to 0^+} q(r)$ converges to a finite positive number. Let $L$ denote the limit. Fixing $R\in(0,1)$ and sending $r\to 0^+$ in (\ref{Rr}), we get $L(1+R)^{-2\alpha_0}\le q(R)\le L(1-R)^{-2\alpha_0}$. So $p(r)=Lr^{\alpha_0}(1+O(r))$ as $r\to 0$ for some $L\in(0,\infty)$.

Next, we  find the asymptotic of $p(z_1,z_2;r)$ as $r\to 0^+$ for any $z_1,z_2\in(0,\pi)$.   From Lemma \ref{property-til-p} we know that, for any $t_0>0$, $\PP[T^u>t_0]={\cal Z} G^u(z_1,z_2) e^{-\alpha_0t_0}(1+O(e^{\lambda_1t_0}))$, and conditionally on $\F^u_{t_0}$ and $T^u>t_0$, the joint density of $(Z^u_1(t_0),Z^u_2(t_0))$ is $\til p^Z_\infty(z_1^*,z_2^*)(1+O(e^{\lambda_1t_0}))$, where $\lambda_1=-2-\frac{\kappa}{8}$. Fix $r\in(0,1/4)$ and choose $t_0>0$ such that $t_0<-\log(4r)$. We now still have (\ref{R1R2}). Note that $R_j=e^{t_0}r(1+O(e^{t_0}r))$, $j=1,2$, if $e^{t_0}r$ is small. From (\ref{p-condition}) we get
$$p(z_1,z_2;r)={\cal Z} G^u(z_1,z_2) e^{-\alpha_0t_0} (1+O(e^{\lambda_1t_0}))p(e^{t_0}r(1+O(e^{t_0}r)))$$
$$={\cal Z}L G^u(z_1,z_2) e^{-\alpha_0t_0} [e^{t_0}r(1+O(e^{t_0}r))]^{\alpha_0}(1+O(e^{\lambda_1t_0}))(1+O(e^{t_0}r))$$
$$={\cal Z}L G^u(z_1,z_2) r^{\alpha_0}(1+O(e^{\lambda_1t_0})+O(e^{t_0}r)).$$ Since $\beta_0=\frac{-\lambda_1}{1-\lambda_1}$,
letting $C_0={\cal Z}L$ and choosing $e^{t_0}$ such that $e^{t_0}=r^{\frac{-1}{1-\lambda_1}}$, we get
$$p(z_1,z_2;r)=C_0G^u(z_1,z_2)r^{\alpha_0} (1+O(r^{\beta_0})).$$
This means that we obtain (\ref{Green-disk}) in the case that $v_1-v_2=\pi$.

Finally, we consider the case that $\theta(0,0)=v_1-v_2\ne\pi$.  First, suppose that $\theta(0,0)<\pi$. Recall that $\theta(t_1,t_2)$ is increasing in $t_2$. Let $\tau_2$ be the first $t_2$ such that $(0,t_2)\in\cal D$ and $\theta(0,t_2)=\pi$, if such time exists; otherwise, let $\tau_2$ be the lifetime $\til T_2$ of $\eta_2$. Then $\tau_2$ is an $(\F^2_t)$-stopping time. From (\ref{pa-j-theta'}) we know that $\pa_2\theta(0,t_2)\ge 2\cot(\theta(0,t_2)/4)$, which implies that $\cos(\theta(0,t_2)/4)\le e^{-t/2}\cos(\theta(0,0)/4)< e^{-t/2}$. If $\log(2)<\til T_2$, then $\cos(\theta(0,\log(2))/4)< 1/\sqrt 2$, which implies that $\theta(0,\log(2))>\pi$, and so $\tau_2<\log(2)$. If $\log(2)\ge \til T_2$, we then have $\tau_2\le \til T_2\le \log(2)$. Thus, in both cases, $\tau_2$ is bounded above by $\log(2)$, and we get an $(\F_{ \ulin t})$-stopping time $(0,\tau_2)$.

Moreover, if $(0,\tau_2)\not\in\cal D$, then $\tau_2=\til T_2$, which means that the conformal radius of $\D\sem \ha \eta_2$ viewed from $0$ is $e^{-\til T_2}\ge 1/2$, and from Koebe's $1/4$ theorem, we get $\dist(0,\ha \eta_2)\ge 1/8$. Thus, if $\dist(0,\ha \eta_2)< 1/8$, then $(0,\tau_2) \in\cal D$, and we get $V_1(0,\tau_2)-V_2(0,\tau_2)=\pi$.
Conditional on $\F_{(0,\tau_2)}$ and the event that $(0,\tau_2)\in\cal D$, the $g_{
(0,\tau_2)}$-image of $\ha \eta_1$ and the part of $\ha \eta_2$ after $\eta_2(\tau_2)$ form a $2$-SLE$_\kappa$ in $\D$ with link pattern $(e^{iW_1(0,\tau_2)}\to e^{iV_1(0,\tau_2)};e^{iW_2(0,\tau_2)}\to e^{iV_2(0,\tau_2)})$. Since $V_1(0,\tau_2)-V_2(0,\tau_2)=\pi$,
by Koebe distortion theorem and the result in the case $v_1-v_2=\pi$, we get that, if $r<1/8$,
$$p(w_1,v_1,w_2,v_2;r)=\EE[{\bf 1}_{\{(0,\tau_2)\in\cal D\}}p(Z_1(0,\tau_2),Z_2(0,\tau_2); e^{\tau_2}r(1+O(r)))]$$
$$=C_0\EE[{\bf 1}_{\{(0,\tau_2)\in\cal D\}}  e^{\alpha_0\mA(0,\tau_2)} G(W_1,V_1;W_2,V_2)|_{(0,\tau_2)}]  r^{\alpha_0}(1+O(r^{\beta_0}))$$
$$=C_0 G(w_1,v_1;w_2,v_2) r^{\alpha_0}(1+O(r^{\beta_0})),$$
where the last step follows from (\ref{EchG}). The proof of the case that $\theta(0,0)>\pi$ is similar.
So we have proved (\ref{Green-disk}) in all cases, which implies (\ref{main-est-1}).

Finally, from (\ref{main-est-1}) we know that there are constants $\rho\in(0,1)$ and $C_1>0$ such that if $\frac rR<\rho$, then $\PP[\dist(z_0,\ha \eta_j)<r,j=1,2]\le C_1G_{D;a_1,b_1;a_2,b_2}(z_0) r^{\alpha_0}$. Using (\ref{Green-bound}) we then get (\ref{main-est-2}) in the case  $\frac rR<\rho$. Since $\PP[\dist(z_0,\ha \eta_j)<r,j=1,2]\le 1$, we get (\ref{main-est-2}) for all $r>0$.
\end{proof}

\begin{proof} [Proof of Theorem \ref{main-thm2}]
  The proof is almost the same as that of the previous theorem except that we need a new way to prove that $\PP[\dist(z_0,\ha \eta_1\cap\ha \eta_2)<r]>0$ for all $r\in(0,R)$. To prove this, first note that from the previous theorem, the probability of the event $E_r$ that both $\ha \eta_1$ and $\ha \eta_2$ visit the disc $\{|z-z_0|<r\}$ is positive, and when this event happens, the connected component of $D\sem \ha \eta_1$ whose boundary contains $a_2,b_2$, denoted by $D_2$, contains a part of the circle $\{|z-z_0|=r\}$ but not the whole circle. Thus, $\pa D_2\cap \{|z-z_0|<r\}$ is not empty. Since conditionally on $\ha \eta_1$, $\ha \eta_2$ is a chordal SLE$_\kappa$ curve in $D_2$, and $\kappa\in(4,8)$, the conditional probability that $\ha \eta_2$ intersects $\pa D_2\cap \{|z-z_0|<r\}$ given $\ha \eta_1$ and $E_r$ is positive, and when  $\ha \eta_2$ intersects $\pa D_2\cap \{|z-z_0|<r\}$, we have $\dist(z_0,\ha \eta_1\cap\ha \eta_2)<r$. So we get the desired positiveness.
\end{proof}

\end{document}